\documentclass[reqno,a4paper]{amsart}
\usepackage{amssymb,amsthm}
\usepackage{ifthen}
\usepackage{datetime}
\usepackage[linktocpage=true]{hyperref}
\usepackage{graphicx}
\usepackage{color}
\usepackage[normalem]{ulem}

\usepackage{wasysym}

\usepackage{caption}
\usepackage{subfigure}

\title[Morawetz estimate for linearized gravity in Schwarzschild]{Morawetz estimate for linearized gravity in Schwarzschild}

\author[L. Andersson]{Lars Andersson} \email{laan@aei.mpg.de}
\address{Albert Einstein Institute, Am M\"uhlenberg 1, D-14476 Potsdam,
  Germany 
}

\author[P. Blue]{Pieter Blue} \email{P.Blue@ed.ac.uk}
\address{The School of Mathematics and the Maxwell Institute,  University of Edinburgh,
Edinburgh, Scotland, UK}

\author[J. Wang]{Jinhua Wang} \email{wangjinhua@xmu.edu.cn}
\address{School of Mathematical Sciences, Xiamen University, Xiamen 31005, China}

\newcounter{mnotecount}[section]

\newcounter{mymnotecount}[section]
\renewcommand{\themymnotecount}{\thesection.\arabic{mymnotecount}}
\newcommand{\mymnote}[1]{\protect{\stepcounter{mymnotecount}}${\raisebox{0.5\baselineskip}[0pt]{\makebox[0pt][c]{\color{magenta}{\tiny\em$\bullet$\themymnotecount}}}}$\marginpar{\raggedright\tiny\em$\!\!\!\!\!\!\,\bullet$\themymnotecount: \blue{#1}}\ignorespaces}
\renewcommand{\mymnote}[1]{}

\newtheorem{theorem}{Theorem}
\newtheorem{lemma}[theorem]{Lemma}

\newtheorem{remark}[theorem]{Remark}
\newtheorem{corollary}[theorem]{Corollary}
\numberwithin{equation}{section} 
\numberwithin{theorem}{section}
\newcounter{step}

\newcommand{\Naturals}{\mathbb{N}}

\newcommand{\Indicator}{\mathbf{1}}

\newcommand{\KDelta}{\Delta}

\newcommand{\pv}{\partial_v}
\newcommand{\pu}{\partial_u}

\newcommand{\dr}{\partial_r}

\newcommand{\dt}{\partial_t}

\newcommand{\p}{\partial}
\def\laplacianslash{\mbox{$\triangle \mkern -13mu /$ \!}}
\def\doubleint{\int\!\!\!\!\!\int}
\def\nablaslash{\mbox{$\nabla \mkern -13mu /$ \!}}
\def\laplacianslash{\mbox{$\triangle \mkern -13mu /$ \!}}

\newcommand{\di}{\mathrm{d}} %{\text{d}}

\newcommand{\vecfont}[1]{\mathbf{#1}}
\newcommand{\tensorfont}[1]{#1}

\newcommand{\solu}{\psi}
\newcommand{\Solu}{\Psi}
\newcommand{\gMetric}{\tensorfont{g}}
\newcommand{\M}{\mathcal{M}}

\newcommand{\scri}{\mathcal{I}^+}
\newcommand{\hori}{\mathcal{H}^+}
\newcommand{\N}{\mathcal{N}}
\newcommand{\So}{ \Psi^{(1)} }
\newcommand{\psifirst}{\psi^{(1)}}
\newcommand{\D}{\mathcal{D}}

\newcommand{\uu}{{\underline{U}}}

\newcommand{\CurlyU}{\mathcal{U}}

\newcommand{\newcommandMG}[3]{\newcommand{#1}{#3}} % A notation
\newcommandMG{\fnMna} {{f_{1}}} {z}
\newcommandMG{\fnMnb} {{f_{2}}} {w}
\newcommandMG{\fnMca} {f_{1,1}}   {z_1}
\newcommandMG{\fnMcb} {f_{2,1}}   {w_1}
\newcommandMG{\fnMda} {f_{1,2}}   {z_2}
\newcommandMG{\fnMdb} {f_{2,2}}   {w_2}
\newcommand{\localiseAwayFromPhotonOrbits}{\Indicator_{r\not\eqsim 3M}}
\newcommand{\vecMclassical}{\vecfont{A}}
\newcommand{\fnMclassical}{q}
\newcommand{\fnMpclassical}{q_{\text{reduced}}}
\newcommand{\fnMncClassical}{f}

\newcommand{\CurlyALzZero}{\mathcal{A}}
\newcommand{\CurlyULzZero}{\mathcal{U}}
\newcommand{\CurlyVLzZero}{\mathcal{V}}

\newcommand{\newepsilon}[3]{\newcommand{#1}{#2}}

\newepsilon{\epsilonInvIntegratedMorawetz}{C}{C_{\text{Morawetz}}}

\begin{document}

\begin{abstract}
The equations governing the perturbations of the Schwarzschild metric satisfy the Regge-Wheeler-Zerilli-Moncrief system. Applying the technique introduced in \cite{A-Blue-wavekerr}, we prove an integrated local energy decay estimate for both the Regge-Wheeler and Zerilli equations. In these proofs, we use some constants that are computed numerically. Furthermore, we make use of the $r^p$ hierarchy estimates \cite{D-R-09-rp, S-13} to prove that both the Regge-Wheeler and Zerilli variables decay as $t^{-\frac{3}{2}}$ in fixed regions of $r$.
\end{abstract}

\maketitle
\tableofcontents

\section{Introduction}\label{introduction}

The Schwarzschild spacetime is an $1+3-$dimensional Lorentzian manifold with
 the Lorentz metric taking the following form in Boyer-Lindquist coordinates $(x^\alpha) = (t,r,\theta,\phi)$, 
\begin{align}
\gMetric_{\mu\nu} \di x^\mu \di x^\nu \nonumber ={}&{}
-\left(1-\frac{2M}{r}\right)\di t^2 +\left(1-\frac{2M}{r}\right)^{-1}\di r^2+r^2 \di \sigma_{S^2},
\label{eq:SchMetric}
\end{align}
in the exterior region which is given by $\M=\mathbb{R} \times [2M, \infty) \times S^2$.
For notational convenience, we let
\begin{equation}\label{mu-eta}
 \eta=1-\mu, \quad \mu=\frac{2M}{r}, \quad \Delta=r^2-2Mr.
\end{equation}
We use $r^\ast$ to denote the Regge-Wheeler tortoise coordinate  
\begin{equation}\label{Regge-Wheeler tortoise coordinate}
r^\ast = r + 2M \log(r-
2M)-3M-2M\log M,
\end{equation}
and use the retarded and advanced Eddington-Finkelstein coordinates $u$ and $v $ defined by $u = t - r^\ast, \,\, v = t + r^\ast$.
%We shall refer to 
%\begin{equation}\label{out-dom}
% \D\doteq\{  (u,v) \in \M :  u\leq0, v\geq0 \}
%\end{equation}
%as the {\em domain of outer communications}.

In the region near the event horizon $\hori$, located at $r=2M$, or inside the black hole, we are also going to consider the coordinate system $(v,r,\theta,\phi)$, where $v$ and  $r$ are defined as above. In the $(v,r,\theta,\phi)$ coordinate system the metric is
\begin{align}
\gMetric_{\mu\nu} \di x^\mu \di x^\nu \nonumber ={}&{}
-(1- \mu )\di v^2 +2\di r\di v+r^2 \di \sigma_{S^2}.
\label{eq:SchMetric-horizon}
\end{align}

The study of the equations governing the perturbations of the vacuum Schwarzschild metric was initiated by Regge-Wheeler \cite{R-W-75} and later completed by Vishveshwara \cite{V-70} and Zerilli \cite{Zerilli1970evenparity}. In fact, perturbations of odd and even parity were treated separately. The perturbations of odd parity are governed by the  Regge-Wheeler equation, which is similar to the wave equation for scalar field on the Schwarzschild manifold. Later, Zerilli considered the even case and showed, by decomposing into spherical harmonics (belonging to the different $\ell$ values), that the even parities are governed by the Zerilli equations. A gauge-invariant formulation was also carried out by Moncrief \cite{M-75-1, M-75-2} and Clarkson-Barrett \cite{C-B-03}.
% and recently by Dafermos-Holzegel-Rodnianski \cite{D-H-R-16}. 
In \cite{C-B-03},  Clarkson-Barrett extended the $1 + 3$ covariant perturbation formalism to a `$1 + 1 + 2$ covariant sheet' formalism by introducing a radial unit vector in addition to the timelike congruence, and decomposing all covariant quantities with respect to this. On the other hand, Dafermos-Holzegel-Rodnianski \cite{D-H-R-16} used the double null foliation of Schwarzschild spacetime to derive the $1 + 1 + 2$ covariant perturbation formalism. Bardeen and Press \cite{B-P-73} analyzed the perturbation equations using the Newman-Penrose formalism. Teukolsky \cite{Teukolsky1973}  extended this to the Kerr family and found that the extreme Newman-Penrose components satisfy the {\it Teukolsky equation}. The Bardeen Press equation is Teukolsky equation restricted to Schwarzschild case. More relations between the Bardeen-Press, Regge-Wheeler, Zerilli, and Teukolsky equations were established by Chandrasekhar \cite{Chand-75, Chand-92-bib}. 
   
We shall prove boundedness, an integrated local energy decay estimate, and pointwise decay for solutions to Regge-Wheeler equation and Zerilli equations, both of which take the form of 
\begin{equation}\label{eq-V}
 \Box_g \solu -V_g\solu = 0,  
\end{equation}
in the exterior region of the Schwartzchild spacetime. Here, $V_g$ is the Regge-Wheeler or Zerilli potential.

We briefly  recall some earlier results about linear wave on black hole spacetimes. Integrated local energy decay estimates were proved for the wave equation outside Schwarzschild black holes \cite{Blue-Soffer-wavesch,Blue-Sterbenz,D-R-09}. The existence of a uniformly bounded energy and integrated local energy decay estimates were proved for $|a|\ll M$ \cite{A-Blue-wavekerr,Dafermos-Rodnianski-smalla,Tataru-Tohaneanu} and more recently for all $|a|<M$ \cite{D--R--S-R}. In addition, there is related work by Finster-Kamran-Smoller-Yau \cite{Kerr-Smoller-decay} in the $|a|<M$ range. The red-shift effect was first used to control linear waves near the event horizon in \cite{D-R-09} (also see \cite{D-R-13}). Furthermore, Dafermos and Rodnianski \cite{D-R-09-rp} introduced an $r^p$ hierarchy, weighted estimates to prove energy decay. Using this method, Schlue \cite{S-13} improved the decay rate for linear wave in fixed regions of $r$ outside a Schwarzschild black hole to $t^{-3/2+\delta}$, and for time derivative to $t^{-2+\delta}$. This could be compared to an earlier result by Luk \cite{L-10}, where he introduced a commutator that is analogous to the scaling and derived similar decay rate. Moschidis \cite{rp-ASF-Mos} performed the $r^p$-weighted energy method to general asymptotically flat spacetimes with hyperboloidal foliation, and proved the decay rate for wave is $\tau^{-2}$ (where $\tau$ is the hyperboloidal time function) providing that an integrated local energy decay statement holds. There is much more work using different methods to improve the decay rate of linear wave. In fact, the local uniform decay rate for linear waves can be improved as $t^{-3}$, see Tataru, Donninger-Schlag-Soffer, etc. \cite{D-R-05, D-S-A-12, M-D-T-12, T-13}.
Besides, we should also mention the results by Pasqualotto \cite{P-spin-1-2}, where he proved pointwise decay for the Maxwell system in Schwarzschild spacetime and by Ma \cite{Ma17spin1Kerr}, where a uniform boundedness of energy and a Morawetz estimate are proved for each extreme Newman-Penrose component on slowly rotating Kerr background.

There has also been a lot of work on energy bounds for the linearized Einstein equation on a Schwarzschild black hole. Integrated energy decay estimates were proved for the Regge-Wheeler equaion \cite{Blue-Soffer-wavesch2}. More recently, Dafermos-Holzegel-Rodnianski \cite{D-H-R-16} exhibited the decay $t^{-1}$ for solutions to Teukolsky equations, which are deduced by exploiting the associated quantities satisfying the Regge-Wheeler type equation. Hung, Keller and Wang \cite{H-K-16, H-K-W-17} worked with the Regge-Wheeler-Zerilli-Moncrief system, and established the decay estimate $t^{-1}$ for solutions to Regge-Wheeler and Zerilli equations. In the recent work of Ma \cite{Ma17spin2Kerr}, Morawetz estimates for the extreme components are proved in Kerr spacetime.

In this paper, we begin with the Regge-Wheeler-Zerilli-Monciref system, using the techniques in \cite{A-Blue-wavekerr} to prove the integrated decay estimate for both Regge-Wheeler and Zerilli variables. Based on this, we apply the $r^p$ hierarchy estimate \cite{D-R-09-rp, rp-ASF-Mos, S-13} to prove that solutions to Regge-Wheeler or Zerilli equations uniformly decay as $t^{-1}$, and the time derivatives decay as $t^{-2}$. Hence, the pointwise decay could be improved as $t^{-3/2}$ for finite $r$. Since both Regge-Wheeler and  Zerilli variables have angular frequence $\ell \geq 2$, it should be possible to improve the decay rate in this paper to at least $t^{-7/2}$ by vector field methods and to the rate given by the Price law, $t^{-7}$, by other methods.

\subsection{Regge-Wheeler and Zerilli equations}\label{Regge-Wheeler and Zerilli equations}
The Regge-Wheeler equation is given by
\begin{equation}\label{RW}
 \Box_g \solu -V_{g}^{RW} \solu = 0,  \quad V_{g}^{RW}=-\frac{8M}{r^3}.
\end{equation}
The Zerilli equation is given for each spherical harmonic mode by
\begin{equation}\label{Zerilli}
 \Box_g \solu -V_{g}^{Z}\solu = 0,  \quad V_{g}^{Z}=-\frac{8M}{r^3}\frac{(2\bar\lambda+3)(2\bar\lambda r+3M)r}{4(\bar\lambda r +3M)^2},
\end{equation}
where $2\bar\lambda =(\ell-1)(\ell+2)\geq4$.
These two equations are related by Chandrasekhar transformation \cite{Chand-75, Chand-92-bib}.
We note that, decomposed into spherical harmonics, $\ell \geq 2$ represents gravitational wave perturbations. In contrast, the equations of linearized gravity which lead to the Regge-Wheeler and Zerilli equations allow only a finite dimensional space of solutions for the $\ell =0$ and $\ell =1$ spherical harmonic modes. These correspond to perturbations of the mass (corresponding to moving from one Schwarzschild solution to another) and of the angular momentum  (corresponds to changing the non-rotating Schwarzschild background to a rotating Kerr solution and to gauge transformations), see \cite{J-99, S-T-01}. For this reason, we only consider solution to \eqref{RW} or \eqref{Zerilli} with support on $\ell \geq 2$.
% This provides a lower bound for the eigenvalue of spherical Laplace--Beltrami operator
%\begin{equation}\label{eigenvalue-laplacian}
%-\mathring{\laplacianslash}_{S^2}=\ell(\ell+1)\geq 6.
%\end{equation}

\begin{remark}[Modes $\ell \geq 2$]\label{rk:mode}
We only consider solution to Regge-Wheeler \eqref{RW} or Zerilli equations\eqref{Zerilli} with modes $\ell\geq2$. For these, the spectrum of $-\mathring{\laplacianslash}_{S^2}$ acting on functions with $\ell \geq 2$ is $\ell(\ell+1)\geq 6$. Hence upon integrating over $S^2(t,r)$, 
\begin{equation}\label{angular-int}
\int_{S^2(t, r)} |\nablaslash\solu|^2 \geq \int_{S^2(t, r)} \frac{6}{r^2}|\solu|^2,
\end{equation}
where $\nablaslash$ is the induced covariant derivative on the sphere $S^2(t,r)$ of constant $r$ and $t$.
\end{remark}

\subsection{Statement of main results}\label{Statement of the Main Theorem}
We now state our main results. 

We shall make of the following hypersurfaces, which are illustrated in Figure \ref{fig:sigma-t}. 
Near the event horizon $\hori$, we fix $r_{NH}>2M$ with corresponding tortoise coordinate value $r_{NH}^\ast$
and let
\begin{equation}\label{sigma-t-int}
 \Sigma^i_\tau\doteq \{v=\tau +r_{NH}^\ast \} \cap  \{  r < r_{NH} \}.
\end{equation}
While away from the horizon, we use the usual time function $t$ and let
\begin{equation}\label{sigma-t-e}
 \Sigma^e_\tau\doteq \{t=\tau \} \cap  \{  r \geq  r_{NH} \}.
\end{equation}
The hypersurfaces $\Sigma_\tau$ is given by 
%\footnote{ One can also use some cut off function to smooth the hypersurface $\Sigma_\tau$, making sure that the norm of the gradient of the cut off function is global bounded. Without confusion, we might denote the smooth hypersurface by $\Sigma_\tau$.}
\begin{equation}\label{sigma-t}
\Sigma_\tau=\Sigma^i_\tau \cup \Sigma^e_\tau.
\end{equation}

\begin{figure}
\centering
\includegraphics[width=3.2in]{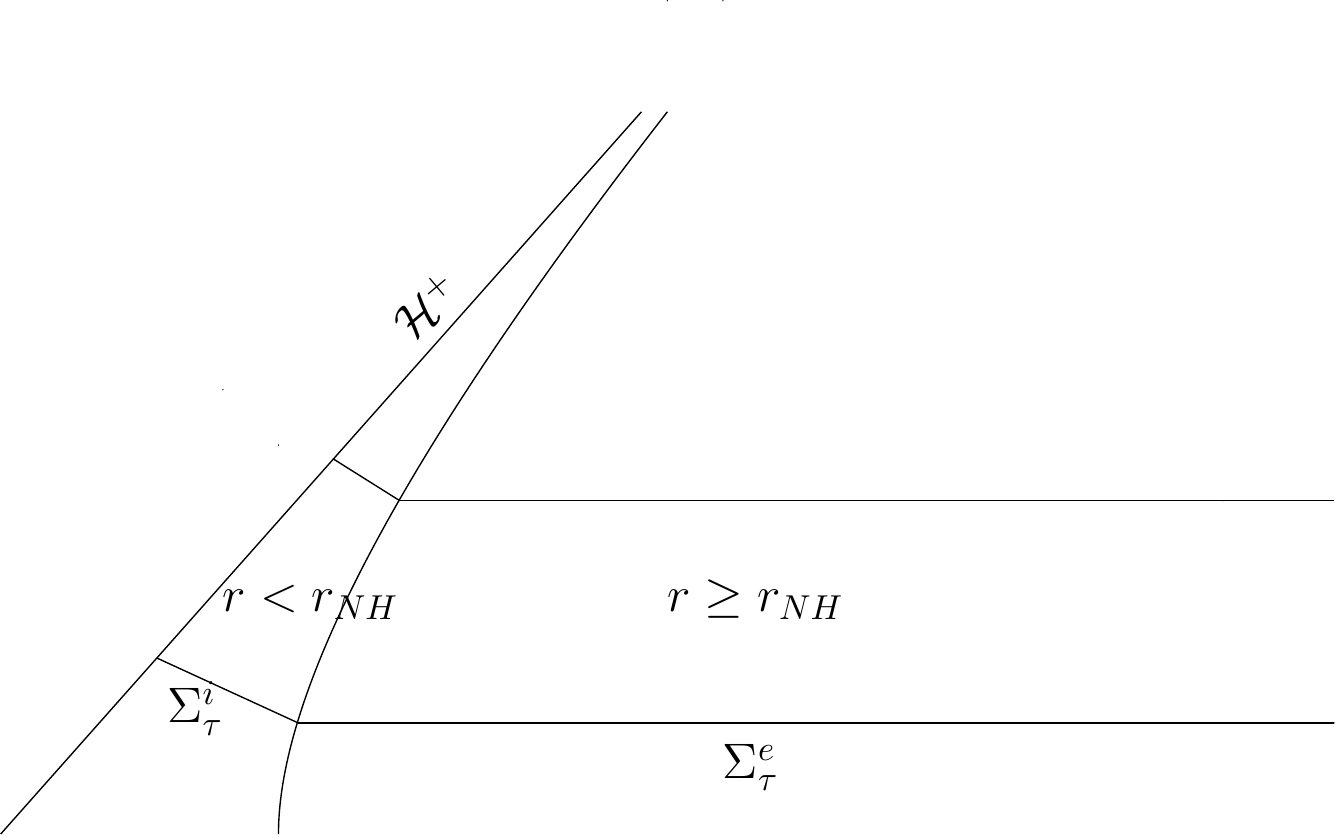}
\caption{The hypersurfaces $\Sigma_\tau=\Sigma^i_\tau \cup \Sigma^e_\tau$}
\label{fig:sigma-t}
\end{figure}

Let $\di \mu_{\Sigma}$ denote the volume form of $\Sigma.$ $n^\alpha_\Sigma$ is the future normal vector of $\Sigma$. We will drop the subscript $\Sigma$ on $n^\alpha_\Sigma$ when there is no confusion. We use $D$ to denote the Levi-Civita connection associated with the Schwarzschild metric $\gMetric_{\mu\nu}$. 
Let $\{ \Omega_i\}, \,\, i= 1,2,3$ be three rotational Killing vector fields about orthogonal axes on the $2$-sphere $S^2(t,r)$ with constant $t$ and $r$, and $\nablaslash$ the induced covariant derivative on $S^2(t,r)$. We shall allow us to use the short hand notation $\Omega^k$ for $\sum\limits_{i_1+i_2+i_3\leq k} \Omega_1^{i_1}\Omega_2^{i_2}\Omega_3^{i_3}$ for all $k\in\Naturals.$
 We fix a globally defined time-like vector field $N$ in the future development of the initial hypersurface $ J^{+}(\Sigma_{\tau_0})$, such that $N = \p_{t} $ for $r \geq r_{NH}>2M$. This would be defined in Section \ref{preliminaries}.
 
The energy-momentum tensor and the  corresponding  momentum vector for \eqref{eq-V} is
\begin{equation}\label{energy-tensor-1-momentum vector}
\begin{split}
\mathcal{T}_{\alpha\beta}(\varphi)&=D_\alpha\varphi D_\beta\varphi-\frac{1}{2} g_{\alpha\beta}( D^\gamma \varphi D_\gamma \varphi +V_g\varphi^2 ),\\
P_{\alpha}^\xi(\varphi)&= \mathcal{T}_{\alpha \beta}\cdot \xi^\beta.
\end{split}
\end{equation}
for a vector field $\xi^\mu$.
We take $\xi=N$ or $\xi=\dt$ to define the energy class.
The non-degenerate energy $E^N(\varphi, \Sigma_\tau)$ associated with $N$ is given by
\begin{equation}\label{def-nondeg-energy}
\begin{split}
&E^N(\varphi, \Sigma_\tau) \doteq  \int_{\Sigma_\tau} P^N_\alpha(\varphi)n_{\Sigma_\tau}^\alpha  \di \mu_{\Sigma_\tau}\\
&\sim \int_{\Sigma^i_\tau} \left( |\p_r\varphi|^2 +|\nablaslash\varphi|^2+\frac{|\varphi|^2}{r^2} \right) r^2\di r\di\sigma_{S^2} \\
&+ \int_{\Sigma^e_\tau} \left( |\p_t\solu|^2 + |\p_r\varphi|^2 +|\nablaslash\varphi|^2+\frac{|\varphi|^2}{r^2} \right) r^2\di r^\ast\di\sigma_{S^2},
\end{split}
\end{equation}
where the integral in the second line of \eqref{def-nondeg-energy} is written in $(v,r,\theta,\phi)$ coordinate.
For $\p_t,$ the associated energy $E^{\p_t}(\varphi,\tau)$ is given by
\begin{equation}\label{def-deg-energy}
\begin{split}
&E^{\p_t}(\varphi,\tau) \doteq  \int_{\{t=\tau\}} P^{\p_t}_\alpha(\varphi)n^\alpha  \di \mu_{\{t=\tau\}} \\
&\sim \frac{1}{2}
\int_{\{t=\tau\}} \left( |\p_t\varphi|^2  + |\p_{r^\ast}\solu|^2+(1-\mu) \left( |\nablaslash\varphi|^2+\frac{|\varphi|^2}{r^2} \right) \right) r^2\di r^\ast\di\sigma_{S^2}.
\end{split}
\end{equation}
Note that this energy is defined as an integral on the hypersurface $\{t=\tau\}$ not $\Sigma_\tau$.

We prove the uniform boundedness, integrated decay estimates  and pointwise decay for solution to the Regge-Wheeler or  Zerilli equations.

\begin{theorem}[Uniform Boundedness of the %Non-Degenerate Positive 
Energy]\label{uniform-bound}
Let $\solu$ be a solution to the Regge-Wheeler \eqref{RW} or Zerilli equations \eqref{Zerilli}, 
then for $\tau>\tau_0$,
\begin{equation}\label{uniform-energy}
E^N(\solu, \Sigma_\tau) \lesssim E^N(\solu,\Sigma_{\tau_0}).
\end{equation}
\end{theorem}
The uniform boundedness of higher order energies is stated in Corollary \ref{uniform-bound-high}.

\begin{theorem}[Integrated Decay Estimate]\label{Integrated decay estimate}
Given $M >0$, there are positive constants $\bar{r}$ and a function $\localiseAwayFromPhotonOrbits$ that is identically one for $|r-3M| >\bar{r}$ and zero otherwise, such that for all smooth function $\solu$ solving the Regge-Wheeler equation \eqref{RW} or Zerilli equations \eqref{Zerilli}, we have
\begin{equation}\label{Morawetz-1-rw-Z}
\begin{split}
&\int_{\tau_0}^{\tau} \di t\int_{\bar\Sigma_t} \left(\frac{\Delta^2}{r^6} |\p_r\solu|^2+\frac{ |\solu|^2}{r^4}+\localiseAwayFromPhotonOrbits \frac{1}{r^3} \left( |\p_t\solu|^2 + |r\nablaslash\solu|^2 \right) \right) r^2\di r \di \sigma_{S^2}   \\
&\lesssim E^{\p_t}(\solu, \tau_0),
\end{split}
\end{equation}
where $\bar\Sigma_\tau=\{t=\tau\}$.
\end{theorem}

\begin{remark}
$\localiseAwayFromPhotonOrbits$ is due to the trapped null geodesic at $r=3M$.
However, with a loss of regularity, commuting with $T$, we have,
\begin{equation}\label{Morawetz-2-nontrap}
\begin{split}
&\int_{\tau_0}^{\tau} \di t\int_{\bar\Sigma_t} \left(\frac{\Delta^2}{r^6} |\p_r\solu|^2+\frac{ |\solu|^2}{r^4}+ \frac{|\p_t\solu|^2}{r^3} + \frac{|r\nablaslash\solu|^2}{r^3} \right) r^2 \di r \di \sigma_{S^2}  \\
 \lesssim & E^{\p_t}(\solu, \tau_0) + E^{\p_t}(T\solu, \tau_0).
\end{split}
\end{equation}
Alternatively, the sharp cut-off $\localiseAwayFromPhotonOrbits$ can be replaced by a function that vanishes quadratically in $(r-3M)$. 
\end{remark}

Combining the red shift effect, the uniform boundedness of the energy (Theorem \ref{uniform-bound}), and the integrated decay estimate (Theorem \ref{Integrated decay estimate}), one can  use the globally time-like vector field $N$ to obtain a non-degenerate local integrated decay estimate (Corollary \ref{nondeg-Integrated decay estimate}). This can be generalized to the high order derivative cases (Corollary \ref{coro-nondeg-Integrated decay estimate-high-order}). 

To state the decay estimate, we introduce additional notation. Let $R>3M$ be a large constant. We define the interior region
\begin{equation}\label{def-interior-sigma-t-prime}
\Sigma^\prime_{\tau} =\Sigma_{\tau} \cap \{r\leq R\}.
\end{equation}
\begin{theorem}[Energy Decay and Pointwise Decay]\label{pointwise-energy-intro}
Let $R>3M$, and define the weighted derivatives $\mathbb{D}=\{r\pv, (1-\mu)^{-1}\pu, r\nablaslash\}$.
Let $\solu$ be a solution of Regge-Wheeler \eqref{RW} or Zerilli equations \eqref{Zerilli}, with initial data (setting $u_0=\tau_0-R^\ast$ and $v_0=\tau_0+R^\ast$) satisfying 
\begin{equation}\label{initial data-improved decay-high-intro}
\begin{split}
I\doteq& \sup_{n \in \Naturals }  \int_{v_0}^{\infty} \di v \int_{S^2}  \sum_{i\leq n}   \sum_{j\leq2, k+l \leq 8 }  |\mathbb{D}^i (r\pv)^j \Omega^l \dt^k \solu|^2 r^2 \di \sigma_{S^2} \big|_{u=u_0}\\
 +& \int_{\Sigma^\prime_{\tau_0}}   \sum_{k+l\leq n+9}  P^N_\mu(\dt^k\Omega^l \solu) n^\mu \di \mu_{\Sigma^\prime_{\tau_0}}<\infty
\end{split}
\end{equation}
for all $n\in \Naturals$. 
Then we have the energy decay estimate
\begin{equation}\label{intro-energy-decay-high}
\begin{split}
& \sup_{n \in \Naturals }  \int_{\tau+R^\ast}^{\infty} \di v \int_{S^2}  P^N_\alpha(\mathbb{D}^n\solu) \p_v^\alpha \cdot r^2 \di \sigma_{S^2} \big|_{u=\tau- R^\ast} \\
+  & \sup_{k,l \in \Naturals }   \int_{\Sigma^\prime_{\tau}}  P^N_\mu(\dt^k\Omega^l \solu) n^\mu \di \mu_{\Sigma^\prime_\tau} \lesssim \frac{I}{\tau^2},
\end{split}
\end{equation}
and
\begin{equation}\label{intro-decay-energy-Improved}
\begin{split}
&\sup_{n \in \Naturals } \int_{\tau+R^\ast}^{\infty} \di v \int_{S^2}   P^N_\alpha(\mathbb{D}^n \dt\solu)\p_v^\alpha \cdot r^2 \di \sigma_{S^2} \big|_{u=\tau- R^\ast}\\
+  &\sup_{k,l \in \Naturals }   \int_{\Sigma^\prime_{\tau}}  P^N_\mu(\dt^k\Omega^l \dt\solu) n^\mu \di \mu_{\Sigma^\prime_\tau} \lesssim \frac{I}{\tau^4}.
\end{split}
\end{equation}
For $t=\tau\geq \tau_0, \,  u\geq u_0$, we have the uniform pointwise decay estimate 
\begin{equation}\label{intr-pointwise decay-uni-high}
r^{\frac{1}{2}} \sup_{n \in \Naturals} |\mathbb{D}^n\solu|(\tau, r) \lesssim \frac{I}{\tau}, \quad  \quad r^{\frac{1}{2}} \sup_{n\in \Naturals} |\mathbb{D}^n \dt\solu|(\tau, r) \lesssim \frac{I}{\tau^2},
\end{equation}
and the improved interior pointwise decay estimate
\begin{equation}\label{intro-improved-pointwise decay-high}
 \sup_{n \in \Naturals}  |\mathbb{D}^n\solu| (\tau,r)\lesssim \frac{I}{\tau^{\frac{3}{2}}},  \quad \text{for} \quad r<R.
\end{equation}
\end{theorem}

\subsection{Comment on the proof}\label{Overview of the Proof}
In this section we present some of the idea in the proof leading to the main results. \\

{\bf Integrated local energy decay.}
The proof is based on the techniques in Andersson and Blue \cite{A-Blue-wavekerr}. We use the radial multiplier vector field as in \cite{A-Blue-wavekerr} (reducing to the Schwarzschild case), which is explained in Section \ref{Local Integrated decay estimate} in details. 
Using the fact that $\ell \geq 2$ as stated in remark \ref{rk:mode}, we could improve the lower bound $\underline{V}$ of the potential in Morawetz estimates. This allows us to find a positive $C^2$ hypergeometric function which further gives a Hardy type estimate. In the Zerilli case, there are some difficulties in getting a smooth lower bound for the Morawetz potential in the whole region $[2M, \infty)$ and in applying the same strategy as in Regge-Wheeler case. However, we are allowed to relax the requirement of smooth lower bound, and find a continuous lower bound $V_{joint}$ for the Morawetz potential. After that, similar to the Regge-Wheeler case, we could work on the second order ordinary differential equation with the potential $\underline{V}$ being replaced by $V_{joint}$, and find a positive $C^2$ solution. Then the Hardy type estimate follows. The crucial part would be in finding a positive $C^2$ solution to the hypergeometric differential equation (associated to the Zerilli case). This could be done by further analysis for the two Frobenius solutions to the hypergeometric differential equation. Especially, the asymptotic expansion at the singularity $r=\infty$ is needed to prove the positivity.

{\bf $r^p$ hierarchy.} In section \ref{decay estimate-RW-Zerilli}, we use a multiplier of the form $r^p\pv$ which gives the $r^p$ hierarchy of estimates and this yields the energy decay. This approach, originated in the context of wave equation \cite{D-R-09-rp}, can also be adapted to Regge-Wheeler and Zerilli equations. Proceeding to the high order case, we further commute the equations with $r\pv$ and derive a first order $r^p$ weighted inequality for all $0<p\leq 2$. (For the wave equation, the range $0<p<2$ was treated in \cite{S-13} and the end point $p=2$ was reached in  \cite{rp-ASF-Mos}.) Based on this, the $r^p$ hierarchy of estimate yields the decay rate $t^{-3/2}$ for $\psi(t,r)$ with finite $r$, which should be compared with \cite{S-13} (or \cite{L-10}), where the decay is as $t^{-3/2+\delta}$ in a compact region. Technically, when $p=2$, one would lose the control for the angular derivatives (see the $r^p$ inequality \eqref{ineq-rp-0} where the coefficient of the angular derivative terms in the spacetime integral vanishes). To achieve the first order $r^p$ weighted inequality for $p=2$,  we additionally perform the integration by parts twice on the leading error term which involves angular Laplacian (see \eqref{eq-A-2}, \eqref{A-2} and \eqref{esimate-A-2-main}). In doing this, we gain the additional factor $(2-p)^2$, which further entails that the leading angular derivative term vanishes when $p=2$. In summary, this approach adapts some idea in \cite{rp-ASF-Mos} to the Regge-Wheeler and Zerilli equations and improves the decay estimates from $t^{-3/2+\delta}$ to $t^{-3/2}$ in finite radius region.

The paper is organized as follows: We begin in Section \ref{preliminaries} with preliminaries, introducing basic notation and background. Section \ref{Local Integrated decay estimate} is devoted to the integrated local energy decay estimate and uniform boundedness. The energy decay  and pointwise decay are proved in Section \ref{decay estimate-RW-Zerilli}.
\\

{\bf Acknowledgements} We are grateful to Steffen Aksteiner, Siyuan Ma and Vincent Moncrief for many helpful discussions and suggestions. J.W. was supported by a Humboldt Foundation post-doctoral fellowship at the Albert Einstein Institute during the period 2014-16, when part of this work was done.  She is also supported by Fundamental Research Funds for the Central Universities (Grant No. 20720170002).

\section{Preliminaries}\label{preliminaries}

In this section, we present some more notations and basic estimates that we shall use throughout the paper.

\subsection{Notations}\label{notations}
Let us introduce the notation $T$ to denote the coordinate vector field $\p_t$ with respect to $(t,r)$ coordinates. It is time-like only when $r>2M$. The globally defined time-like vector field $N$ could be defined as 
\begin{equation*}
N=\dt + \frac{y_1(r)}{1-\mu}\pu + y_2(r)\pv,
\end{equation*}
where $y_1, y_2 >0$ are supported near the event horizon, say, $r\leq r_{NH}$, and $y_1=1, y_2=0$ at the event horizon. Notice that we can also write $N$ in the $(v,r, \theta,\phi)$ coordinates as
\begin{equation*}
N=(1+ 2y_2(r))\pv - (y_1(r)- y_2(r)(1-\mu))\dr.
\end{equation*}

We shall let $D$ denote the covariant derivative associated with the Schwarzschild metric, $\partial$ the derivative in terms of coordinates, and $\nablaslash$ the induced covariant derivative on the sphere $S^2(t,r)$ of constant $t$ and $r$. $\laplacianslash$ is the induced Laplacian on $S^2(t,r)$.
For volume forms, we denote the spacetime volume form by $\di \mu_g$, volume form on $\Sigma_\tau$ by $\di \mu_{\Sigma_\tau}$. And let $S^2$ be the unit $2$-sphere with $\mathring{\laplacianslash}_{S^2}$ the induced Laplacian on $S^2$.\\

 The notation $x\lesssim 􏰑y$ means $x\leq c y$ for a universal constant $c$, and the notation $x\sim 􏰯y$ means $x􏰑 \lesssim y$ and $y \lesssim􏰑 x$. All objects are smooth unless otherwise stated.

\subsection{Energy estimate for wave equation}\label{Energy Estimate}
We would like to study the solutions to the wave equation \eqref{eq-V} on Schwarzschild spacetime.
The energy-momentum tensor for \eqref{eq-V} is
\begin{equation}\label{energy-tensor-1}
\mathcal{T}_{\alpha\beta}(\varphi)=D_\alpha\varphi D_\beta\varphi-\frac{1}{2} g_{\alpha\beta}( D^\gamma \varphi D_\gamma \varphi +V_g\varphi^2 ).
\end{equation}
Given a vector field $\xi^\mu$, the  corresponding  momentum vector is defined by
\begin{equation}\label{momentum vector}
P_{\alpha}^\xi(\varphi)= \mathcal{T}_{\alpha \beta}\cdot \xi^\beta.
\end{equation}
The corresponding energy on a hypersurface $\Sigma_\tau$ is
\begin{equation}\label{def-energy -t}
E^\xi(\varphi,\Sigma_\tau)=-\int_{\Sigma_\tau} P^\xi_{\alpha}(\varphi) n^{\alpha}_{\Sigma_\tau} \di \mu_{\Sigma_\tau},
\end{equation}
where $n^{\alpha}_{\Sigma_\tau}$ is outward normal to $\Sigma_\tau$.
The energy identity takes the form
\begin{equation}\label{energy identity-stockes}
E^\xi(\varphi,\Sigma_\tau)-E^\xi(\varphi,\Sigma_{\tau_0})=-\doubleint_{\D} D^\alpha P_\alpha^\xi \di \mu_g
\end{equation}
where $\D$ is the region enclosed between $\Sigma_\tau$ and $\Sigma_{\tau_0}$.
The associated current $K^\xi(\varphi)$ is defined as
\begin{equation}\label{def-current}
K^\xi(\varphi)=D^\alpha P_\alpha^\xi(\varphi).
\end{equation}
In applications, $\xi$ will be taken as the vector field $\dt$ or $ N.$

\subsection{Hypergeometric Differential Equation}\label{Hypergeometric Differential Equation}
We refers to \cite{AbramowitzEtAl}($\S 15$) for more background of \emph{hypergeometric functions}.
\begin{equation}\label{hypergeometric-differential-equation}
z(1-z)\frac{d^2w}{d^2z}+(c-(a+b+1)z)\frac{dw}{dz}-abw=0.
\end{equation}
This is the \emph{hypergeometric differential equation}.

\subsubsection{Fundamental Solutions}\label{Fundamental Solutions}
Solution of the hypergeometric differential equation \eqref{hypergeometric-differential-equation} has regular singularities at $z=0,1,\infty$ with corresponding exponent pairs $\{0, 1-c\}, \{0, c-a-b\}, \{ a, b\}$ respectively \cite{AbramowitzEtAl} ($\S 15$). When none of the exponent pairs differ by an integer, that is, when none of the $c, c-a-b, a-b$ is an integer, we have the following pairs $f_1(z), f_2(z)$ of fundamental solutions. They are also numerically satisfactory in the neighborhood of the corresponding singularity. 
\begin{itemize}
\item Adapted to {\bf Singularity $z=0$}
\begin{equation}\label{f-s-asym-0}
\begin{split}
f_1(z)&=F(a, b; c;z),\\
f_2(z)&=z^{1-c}F(a-c+1, b-c+1; 2-c; z).
\end{split}
\end{equation}
\item Adapted to {\bf Singularity $z=1$} 
\begin{equation}\label{f-s-asym-1}
\begin{split}
f_1(z)&=F(a, b; a+b+1-c; 1-z),\\
f_2(z)&=(1-z)^{c-a-b}F(c-a, c-b; c-a-b+1; 1-z).
\end{split}
\end{equation}
\item Adapted to {\bf Singularity $z=\infty$}
\begin{equation}\label{f-s-asym-infty}
\begin{split}
f_1(z)&=z^{-a}F(a, a-c+1;a-b+1;\frac{1}{z}),\\
f_2(z)&=z^{-b}F(b, b-c+1;b-a+1;\frac{1}{z}).
\end{split}
\end{equation} 
\end{itemize}

\subsubsection{Integral Representations}\label{Integral Representations}
The hypergeometric function $F(a,b;c;z)$ has the following integral representation \cite{AbramowitzEtAl} ($\S15$):  For $0<\Re b<\Re c$ and $z\not\in[1,\infty)$
\begin{align}
F(a,b;c;z)
={}&{} 
\frac{\Gamma(c)}{\Gamma(b)\Gamma(c-b)} \int_0^1 t^{b-1} (1-t)^{c-b-1} (1-zt)^{-a} \di t .
\label{eq:Hardy-Hypergeometric-Integral-Representation}
\end{align}
And $F$ is symmetric in its first two arguments, $F(a,b;c;z)=F(b,a;c;z)$.

\section{Integrated local energy decay and Uniform bound}\label{Local Integrated decay estimate}

\subsection{Morawetz Vector Field}
In this section, we consider momentum of the form associated to solution $\psi$ of the wave equation \eqref{eq-V},
\begin{equation*}
P^\alpha=P^\alpha[\psi, \xi,q]=\mathcal{T}^{\alpha}_{\beta}\cdot \xi^\beta+q D^\alpha\psi\cdot\psi-\frac{1}{2} D^\alpha q\cdot\psi^2.
\end{equation*}
We shall consider a generalized Morawetz vector field $\vecfont{A}$ of the form:
\begin{align*}
\vecMclassical ={}&{} -\fnMna\fnMnb\fnMncClassical \dr ,\,\,\,\,\,\,
\fnMncClassical= \dr z ,\\
\fnMpclassical ={}&{} \frac12\left(\dr\fnMna\right)\fnMnb\fnMncClassical ,\,\,\,
\fnMclassical  =\frac12\left(\dr \vecMclassical^r\right) + \fnMpclassical .
\end{align*}
Defining the functions $\fnMnb, \fnMna$:
\begin{align*}
\fnMnb=\frac{r^5}{6}, \quad \KDelta= r^2-2Mr, \quad \fnMna= \frac{\KDelta}{r^4},
\end{align*}
we have
\begin{equation}\label{bulk-Morwetz}
D_\alpha P^\alpha[\solu, \vecMclassical,q] 
= \CurlyALzZero (\dr\solu)^2 + \CurlyULzZero^{\alpha\beta} \partial_\alpha\solu \partial_\beta\solu +\CurlyVLzZero|\solu|^2 , 
\end{equation}
with
\begin{align*}
\CurlyALzZero={}&{}-\fnMna^{1/2} \KDelta^{3/2} \dr\left(\fnMnb \frac{\fnMna^{1/2}}{\KDelta^{1/2}}\fnMncClassical \right) , \\
\CurlyVLzZero={}&{} \frac{1}{4}\dr(\KDelta\dr(\fnMna\dr(\fnMnb f))) + \frac{1}{2}\fnMnb\fnMncClassical\dr(zV),
\end{align*}
and $\CurlyU^{\alpha\beta} \p_\alpha \solu \p_\beta \solu$ involves angular derivatives 
\begin{align*}
\CurlyU^{\alpha\beta} \p_\alpha \solu \p_\beta \solu=  \CurlyU |r\nablaslash\solu|^2 + \p_r \left(\frac{r^4}{2\Delta}  \right) zwf (\dt\solu)^2 , \quad \CurlyU=\frac12 \fnMnb\fnMncClassical^2.
\end{align*}
%$\vecfont{Q}$ is defined as 
%\begin{equation}\label{Q}
%\vecfont{Q}=\mathring{\laplacianslash}-\partial^2_\phi.
%\end{equation}
With the choice of $z$ and $w$, we have $$\CurlyALzZero=\frac{M\KDelta^2}{r^4}.$$

As in \cite{A-Blue-wavekerr}, we define 
\begin{equation*}
\mathcal{B}^\alpha=\CurlyU\mathcal{L}^{\alpha\beta}\p_{\beta}\solu\cdot\solu,
\end{equation*}
where the second order operator $\mathcal{L}=\p^2_t + \mathring{\laplacianslash}_{S^2}$.
With the boundary term $\mathcal{B}^\alpha$, we have

\begin{equation}\label{bulk-Morwetz-re}
D_\alpha \left( P^\alpha[\solu, \vecMclassical,q]-\mathcal{B}^\alpha \right)
= \CurlyALzZero (\dr\solu)^2 + \CurlyULzZero^{\alpha\beta} \partial_\alpha\solu \partial_\beta\solu +\CurlyVLzZero|\solu|^2.
\end{equation}
Here $$\CurlyU^{\alpha\beta}=\CurlyU \mathcal{L}^{\alpha\beta},$$ (we still denote it by $\CurlyU^{\alpha\beta}$ without confusion) and $\CurlyALzZero, \CurlyVLzZero$ are the same as above.

For the junk term $\CurlyULzZero^{\alpha\beta} \partial_\alpha\solu \partial_\beta\solu$ in \eqref{bulk-Morwetz-re}, recalling that $$\CurlyU^{\alpha\beta}=\frac12 \fnMnb\fnMncClassical^2 \mathcal{L}^{\alpha\beta},$$ we have
 $$\CurlyULzZero^{\alpha\beta} \partial_\alpha\solu \partial_\beta\solu \geq \frac{1}{2} wf^2 \left( |r\nablaslash\solu|^2 +(\partial_t\solu)^2 \right).$$ Hence upon integrating over $S^2(t,r)$, using the fact that the spectrum of $-\mathring{\laplacianslash}_{S^2}$ acting on functions with $\ell \geq 2$ is $\ell(\ell+1)\geq 6,$ we have 
\begin{equation}\label{U-lower}
\int_{S^2(t, r)} \CurlyULzZero^{\alpha\beta} \partial_\alpha\solu \partial_\beta\solu \geq \int_{S^2(t, r)} 6 \uu |\solu|^2,
\end{equation}
where $\uu$ is 
\begin{equation}\label{U}
\uu= \frac{1}{3}\frac{(r-3M)^2}{r^3}.
\end{equation}
Putting this together we have for some constant $\epsilon>0$, 
\begin{equation}\label{bulk-Morwetz-re-1}
\begin{split}
D_\alpha \left( P^\alpha[\solu, \vecMclassical,q]-\mathcal{B}^\alpha \right)
\geq {}&{}  \epsilon \left( \CurlyALzZero (\dr\solu)^2 +(\CurlyVLzZero + 6\uu)|\solu|^2 \right) \\
{}&{} +(1-  \epsilon) \CurlyU \mathcal{L}^{\alpha\beta} \partial_\alpha\solu \partial_\beta\solu.
\end{split}
\end{equation}
Hence, the Morawetz estimate reduces to prove the Hardy inequality
\begin{equation}\label{Hardy}
\int_{2M}^{\infty} \left( A |\p_r\varphi|^2+ V |\varphi|^2 \right)  \di r \geq \epsilon_{\text{Hardy}} \int_{2M}^{\infty}  \left( \frac{\Delta^2}{r^4} |\p_r\varphi|^2+\frac{\varphi^2}{r^2} \right) \di r.
\end{equation}
As in the proof of Lemma 3.12 in \cite{A-Blue-wavekerr}, one has to show that there is a positive $C^2$ solution of the ordinary differential equation
\begin{equation*}
-\partial_r(A \partial_r) \phi + V\phi =0.
\end{equation*}

\subsection{Integrated decay estimate}\label{proof of integrated decay}
We would like to introduce a lemma relating the above ODE and the hypergeometric functions.
Recall that
\begin{align*}
\KDelta&=r^2-2Mr .
\end{align*}

\begin{lemma}\label{lemma-hypergeometric-tansformations}
Let $A=M\frac{\KDelta^2}{r^4}$ and $V=\frac{M}{r^4}(V_2r^2+V_1Mr^1+V_0M^2)$. 

Let
\begin{align*}
\alpha{}=&\frac12 +\frac{\sqrt{4V_2+2V_1+V_0+1}}{2} ,\\
\beta{}=&\frac12-\frac{\sqrt{9+V_0}}{2} ,\\
a{}=& \frac{1+\sqrt{1+4V_2+2V_1+V_0}-\sqrt{9+V_0}-\sqrt{1+4V_2}}{2},\\
b{}=& \frac{1+\sqrt{1+4V_2+2V_1+V_0}-\sqrt{9+V_0}+\sqrt{1+4V_2}}{2},\\
c{}=& 1+\sqrt{1+4V_2+2V_1+V_0} .
\end{align*}
Assume none of $c$, $c-a-b$, and $a-b$ are integers. 

For the ODE
\begin{align}
-\partial_r(A\partial_r\phi) +V\phi = 0 ,
\label{eq:CoreODE}
\end{align}
a pair of fundamental solutions, which we call the Frobenius solutions adapted to $r=2M$, is
\begin{align*}
&(r-2M)^\alpha r^\beta F\left(a,b,c,-\frac{r-2M}{2M}\right) ,\\
&(r-2M)^\alpha r^\beta \left(-\frac{r-2M}{2M}\right)^{1-c}F\left(a-c+1,b-c+1,2-c,-\frac{r-2M}{2M}\right) . 
\end{align*}
The second can also be expressed as
\begin{align*}
(r-2M)^\alpha r^\beta \left(-\frac{r-2M}{2M}\right)^{1-c}
\left(\frac{r}{2M}\right)^{c-a-b}
F\left(1-a,1-b,2-c,-\frac{r-2M}{2M}\right). 
\end{align*}
Another pair of fundamental solutions, which we call the Frobenius solutions adapted $r=\infty$, is
\begin{align*}
&(r-2M)^\alpha r^\beta \left(-\frac{r-2M}{2M}\right)^{-a}F\left(a,a-c+1,a-b+1,-\frac{2M}{r-2M}\right) ,\\
&(r-2M)^\alpha r^\beta \left(-\frac{r-2M}{2M}\right)^{-b}F\left(b,b-c+1,b-a+1,-\frac{2M}{r-2M}\right) .
\end{align*}
\end{lemma}
\begin{proof}
We follow the argument from \cite{A-Blue-wavekerr}. Let
\begin{equation}\label{hypergeometric-tranform}
\begin{split}
v&=A^{1/2} u,\\
x&=r-2M.
\end{split} 
\end{equation}
The ODE \eqref{eq:CoreODE} then becomes 
and find that the resulted ordinary differential equation 
\begin{equation}\label{ode-transformed}
-\partial_x^2v+Wv=0,
\end{equation}
and
\begin{align*}
W{}=&\frac{V}{A} +\frac12\frac{\partial_x^2A}{A}-\frac{(\partial_x A)^2}{A^2} \\
=&\frac{V_2r^2  +(V_1-4)Mr +(V_0+8)M^2}{r^2(2M-r)^2} \\
=&\frac{V_2x^2 +(4V_2+V_1-4)Mx +(4V_2+2V_1+V_0)M^2}{x^2(x+2M)^2} .
\end{align*}
Now, let $\tilde{v}$ be such that $v=x^\alpha (x+d)^{\beta}\tilde{v}$, so that the above ODE becomes\footnote{Observe that equation (3.23) in \cite{A-Blue-wavekerr} is missing a minus sign in front of $x^2(x+d)^2\partial_x^2\tilde{v}$, but the rest of the argument there is correct.}
\begin{align*}
0{}=&x^{\alpha-2}(x+d)^{\beta-2}P,\\
P{}=&-x^2(x+d)^2\partial_x^2\tilde{v}\\
&-2x(x+d)((\alpha+\beta)x+\alpha d)\partial_x\tilde{v}\\
&+\big(-(\alpha(\alpha-1)(x+d)^2+2\alpha\beta x(x+d)+\beta(\beta-1)x^2)
+W x^2(x+d)^2\big) \tilde{v} .
\end{align*}

Let $d=2M$, so that
\begin{align*}
W x^2(x+d)^2 
{}=& V_2x^2 +(4V_2+V_1-4)Mx +(4V_2+2V_1+V_0)M^2 ,
\end{align*}
and the coefficient of $\tilde{v}$ in $P$ above is
\begin{align*}
&(-\alpha(\alpha-1)-2\alpha\beta-\beta(\beta-1)+V_2)x^2 \\
&+(-4\alpha(\alpha-1)-4\alpha\beta+4V_2+V_1-4)Mx\\
&+(-4\alpha(\alpha-1)+4V_2+2V_1+V_0)M^2 .
\end{align*}
We choose $\alpha$ so that the $M^2$ coefficient vanishes, i.e.
\begin{align*}
0{}=&-4\alpha(\alpha-1)+4V_2+2V_1+V_0 \\
=&-4\alpha^2+4\alpha +(4V_2+2V_1+V_0) ,\\
\alpha{}=&\frac12 \pm\frac{\sqrt{4V_2+2V_1+V_0+1}}{2} .
\end{align*}
From the second line above, we also have the identity $-\alpha(\alpha-1)=-(4V_2+2V_1+V_0)/4$. We now choose $\beta$ so that the ratio of the $Mx$ coefficient to the $x^2$ coefficient is $2$, i.e.
\begin{align*}
2
{}=&\frac{-4\alpha(\alpha-1)-4\alpha\beta+4V_2+V_1-4}{-\alpha(\alpha-1)-2\alpha\beta-\beta(\beta-1)+V_2}\\
-2\alpha(\alpha-1)-4\alpha\beta-2\beta(\beta-1)+2V_2
{}=&-4\alpha(\alpha-1)-4\alpha\beta+4V_2+V_1-4\\
-2\beta(\beta-1)
{}=&-2\alpha(\alpha-1)+2V_2+V_1-4\\
=&\left(-2V_2-V_1-\frac12V_0\right)+2V_2+V_1-4\\
=&-\frac12V_0-4,\\
0{}=&4\beta^2-4\beta+\left(-V_0-8\right)\\
\beta&=\frac12\pm\frac{\sqrt{9+V_0}}{2}.
\end{align*}
We choose the $+$ sign in $\alpha$ and $-$ sign in $\beta$. 

The substitution $x=-2Mz$ now yields the ODE
\begin{align*}
0{}
=&z(1-z)\partial_z^2 \tilde{v}
+(2\alpha -(2\alpha+2\beta)z)\partial_z\tilde{v}
-(2\alpha\beta+1+V_0/2+V_1/2)\tilde{v} .
\end{align*}
This is now in the form of a standard hypergeometric differential equations, and the corresponding parameters thus satisfy
\begin{align*}
c{}=&2\alpha ,\\
a+b+1{}=&2\alpha+2\beta ,\\
ab{}=&(2\alpha\beta+1+V_0/2+V_1/2) ,
\end{align*}
which implies the parameters $a$ and $b$ are %\mnote{Assuming I have applied Maple correctly, which can -and thus should- be checked by hand}
\begin{align*}
a{}=&\alpha+\beta-\frac12-\frac{\sqrt{4\alpha(\alpha-1)+4\beta(\beta-1)-2V_1-2V_0-7}}{2},\\
b{}=&\alpha+\beta-\frac12+\frac{\sqrt{4\alpha(\alpha-1)+4\beta(\beta-1)-2V_1-2V_0-7}}{2} .
\end{align*}
Thus, the parameters are
\begin{align*}
a{}=& \frac{1+\sqrt{1+4V_2+2V_1+V_0}-\sqrt{9+V_0}-\sqrt{1+4V_2}}{2},\\
b{}=& \frac{1+\sqrt{1+4V_2+2V_1+V_0}-\sqrt{9+V_0}+\sqrt{1+4V_2}}{2},\\
c{}=& 1+\sqrt{1+4V_2+2V_1+V_0} .
\end{align*}

Several forms for solutions to the hypergeometric function are given in \cite{AbramowitzEtAl}. Reversing all the substitutions made so far in the proof, one finds a pair of fundamental solutions adapted to $z=0$ (i.e. $r=2M$) is
\begin{align*}
&(r-2M)^\alpha r^\beta F\left(a,b,c,-\frac{r-2M}{2M}\right) ,\\
&(r-2M)^\alpha r^\beta \left(-\frac{r-2M}{2M}\right)^{1-c} 
F\left(a-c+1,b-c+1,2-c,-\frac{r-2M}{2M}\right).
\end{align*}
An alternative way of writing the second solution is
\begin{align*}
(r-2M)^\alpha r^\beta \left(-\frac{r-2M}{2M}\right)^{1-c}\left(\frac{r}{2M}\right)^{c-a-b}
F\left(1-a,1-b,2-c,-\frac{r-2M}{2M}\right).
\end{align*}
Another pair of fundamental solutions adapted to $z=\infty$ (i.e. $r$ approaching the point at infinity on the Riemann sphere except along the negative real axis) is 
\begin{align*}
&(r-2M)^\alpha r^\beta \left(-\frac{r-2M}{2M}\right)^{-a}F\left(a,a-c+1,a-b+1,-\frac{2M}{r-2M}\right) ,\\
&(r-2M)^\alpha r^\beta \left(-\frac{r-2M}{2M}\right)^{-b}F\left(b,b-c+1,b-a+1,-\frac{2M}{r-2M}\right) .
\end{align*}
\end{proof}

In this section, we will prove the integrated decay estimate for both Regge-Wheeler and zerilli equations.

\subsubsection{The Regge-Wheeler case}\label{proof of integrated decay-RW}
\begin{proof}[Proof of Theorem \ref{Integrated decay estimate} for {\bf Regge-Wheeler case}]

We first prove the statement for Regge-Wheeler case, where we have
\begin{align}\label{A-V}
\CurlyALzZero={}&{}\frac{M\KDelta^2}{r^4},\\
\CurlyVLzZero= {}&{}-\frac{5}{2}Mr^{-2}+15M^2r^{-3}-23M^3r^{-4}.
\end{align}
Denoting $\underline{V}=\CurlyVLzZero + 6\uu \cdot \frac{2M}{r},$  we have the lower bound for the Morawetz potential $\CurlyVLzZero + 6\uu  \geq \underline{V}$ and
\begin{equation}\label{V-RW}
\underline{V}= \frac{3}{2}Mr^{-2}-9M^2r^{-3}+13M^3r^{-4}.
\end{equation}
Recalling the $\CurlyALzZero$ \eqref{A-V} and $\underline{V}$ \eqref{V-RW}, we let 
\begin{equation}\label{def-A-V-RW}
A=  \CurlyALzZero, \quad V= \underline{V}.
\end{equation}
Now we apply the same transformations \eqref{hypergeometric-tranform} and find that the resulted ordinary differential equation \eqref{ode-transformed} has a solution taking the following form
\begin{equation}\label{Hypergeometic fun}
u=A^{-\frac{1}{2}}(r-2M)^{\alpha}r^{\beta}F(a;b;c;z),
\end{equation}
where $z=-\frac{r-2M}{2M},$ and $F(a;b;c;z)$ is the associated hypergeometric function.
 We follows Lemma \ref{lemma-hypergeometric-tansformations} to choose the parameters,
\begin{equation*}
\alpha=\frac{1 \pm \sqrt{2}}{2}, \quad \beta=\frac{1 \pm \sqrt{22}}{2}.
\end{equation*}
We take the $+$ sign in $\alpha$ and the $-$ sign in $\beta$.
Then we can immediately read off some quantities in terms of $\alpha$ and $\beta$,
\begin{align*}
c&=1+\sqrt{2},\\
a+b&=1+\sqrt{2}-\sqrt{22}, \\
ab&= \frac{1+\sqrt{2}-\sqrt{22} - 2\sqrt{11} }{2} + 4.
\end{align*}
The remaining two parameters $a,b$ could be solved
\begin{align*}
\{ a , b \}=\{ \frac{1+\sqrt{2}-\sqrt{22} \pm \sqrt{7}}{2} \}.
\end{align*}
We choose $a<b$  such that
\begin{equation*}
a<-2.4<0<0.1<b<0.2<2.0<c.
\end{equation*}
We can use the integral representation \eqref{eq:Hardy-Hypergeometric-Integral-Representation} to show that
the hypergeometric function $F(a;b;c;z)$ with $z<0$ is positive. This further leads to the Hardy estimate \eqref{Hardy}.
Hence the integrated decay estimate follows,
\begin{equation*}
\doubleint_{\M} \frac{\Delta^2}{r^4} |\p_r\solu|^2+\frac{|\solu|^2}{r^2} +\localiseAwayFromPhotonOrbits \left( \frac{|\p_t\solu|^2}{r}  + \frac{ |r\nablaslash\solu|^2}{r} \right) \di t \di r \di \sigma_{S^2} \lesssim E^T(\solu, \tau_0).
\end{equation*}
That is,
\begin{equation}\label{Morawetz-0}
\doubleint_{\M}\frac{\Delta^2}{r^6} |\p_r\solu|^2+\frac{ |\solu|^2}{r^4}+\localiseAwayFromPhotonOrbits \left( \frac{|\p_t\solu|^2}{r^3} + \frac{|r \nablaslash\solu|^2}{r^3} \right) \di \mu_g \lesssim E^T(\solu, \tau_0).
\end{equation}
Therefore we complete the proof for the Regge-Wheeler case.
\end{proof}

\subsubsection{The Zerilli case}\label{proof of integrated decay-Zerilli}

Let $V_g^{RW}$ denote the potential of Regge-Wheeler equation, and $V_g^{Z}$ the Zerilli potential. They are related by $V_g^{Z}=V_g^{RW}(1+\zeta),$ with 
\begin{equation}\label{epsilon}
\zeta= \frac{2\bar\lambda+3}{4\bar\lambda}\left(9 \left( \frac{1}{2}-\frac{M}{\Lambda} \right)^2-\frac{1}{4} \right) -1.
\end{equation}
Here 
\begin{equation}\label{def-bar-lambda}
\Lambda=\bar\lambda r +3M,\,\,\, 2\bar\lambda=\ell(\ell+1)-2.
\end{equation}
Noting that $\ell\geq 2$, we have $\bar\lambda\geq 2$.  We calculate
\begin{equation}\label{dr-epsilon-1}
\partial_r \zeta= \frac{9M \left( 2\bar\lambda +3 \right) }{2} \left( \frac{1}{2} - \frac{M}{\Lambda} \right)\Lambda^{-2}.
\end{equation}
It is easy to check that $$\dr \zeta >0.$$

Before proceeding to the proof for Zerilli case, we first state the main idea and main steps for the proof.
In the Zerilli case, it would be difficult to find a smoothly lower bound for the Morawetz potential $\mathcal{V}+6\uu$ \eqref{bulk-Morwetz-re-1} in the whole region $[2M, \infty)$. The key point is: we are allowed to relax the requirement of smoothly lower bound, and find a continuously lower bound $V_{joint}$ for the Morawetz potential. Then, similar to the Regge-Wheeler case, we could work on the second order ordinary differential equation \eqref{eq:CoreODE} with the potential $V$ being replaced by $V_{joint}$, and find a positive $C^2$ solution.  The Hardy inequality then follows and hence the integrated decay estimate. 

The first step is to separate the estimate in the two regions $[2M, 3M]$ and $(3M, \infty)$ and find a $C^0$ lower bound $V_{joint}$ for the Morawetz potential. Note that, the lower bounding potential $V_{joint}$ is chosen such that the second order ordinary differential equation \eqref{eq:CoreODE} could be transformed to  hypergeometric differential equations in each of the two intervals.  In step two, we will analyze the hypergeometric differential equation associated to the ODE \eqref{eq:CoreODE}, and find out the Frobenius solutions (adapted to $x=0$) in $[2M, 3M]$ and $(3M, \infty)$. The Frobenius solutions (adapted to $x=\infty$) follow by making some transformations on the old Frobenius solutions (adapted to $x=0$).
 At last, in step three, we will construct a $C^1$ solution to the hypergeometric differential equation, which comes from linear combinations of the Frobenius solutions (adapted to $x=0$). We will show that this solution is  $C^2$ and positive by finding a positive lower bound $G(x)$. The construction of $G(x)$ is based on an observation that the ratio of two Frobenius solutions (adapted to $x=0$) has a limitation at infinity. Remarkably, further analysis for the asymptotically expansion at infinity shows that $G(x)$ is proportional to one of the Frobenius solutions adapted to $x=\infty$. We can further make use of the integral representation \eqref{eq:Hardy-Hypergeometric-Integral-Representation} to show the positivity of $G(x)$.

{\bf Step I: The continuous lower bound for the Morawetz potential.}
\begin{lemma}[Lower bound for the Morawetz potential]\label{lemma-Lower bound for the potential}
In the Zerilli case, we have the lower bound for the Morawetz potential $\mathcal{V}+6\uu$ in \eqref{bulk-Morwetz-re-1} as the following:

For case I, $r\leq 3M$ 
\begin{equation}\label{bound-less-3M-general}
\mathcal{V}+6\uu\geq  \underline{V}_{r\leq 3M} \doteq \frac{5M}{2r^2}  - \left(13 + \frac{2}{3} \right) \frac{M^2}{r^3}  +  \frac{18M^3}{r^4}.
\end{equation}

For case II, $r\geq 3M$ 
\begin{equation}\label{bound-greater-3M-general}
\mathcal{V}+6\uu\geq   \underline{V}_{r\geq 3M} \doteq \frac{M}{2r^2} -\frac{4M^2}{r^3}  + \frac{7M^3}{r^4}.
\end{equation}

In particular, we have
\begin{equation}\label{compare-v-3m}
\left(\underline{V}_{r\leq 3M} - \underline{V}_{r\geq 3M} \right) \big|_{r=3M} =0.
\end{equation} 
\end{lemma}

\begin{proof}[Proof of Lemma \ref{lemma-Lower bound for the potential}]
As in the Regge-Wheeler case, we still have the formulation \eqref{bulk-Morwetz-re-1} with $\CurlyALzZero=\frac{M\KDelta^2}{r^4}$ being the same as in \eqref{A-V}. Now we have to estimate the lower bound for $\mathcal{V}+6\uu$ in \eqref{bulk-Morwetz-re-1}.
First, we recall the formula for $\CurlyVLzZero, $  
\begin{equation}\label{eq-formula-pV}
\CurlyVLzZero= \frac{1}{4}\dr(\KDelta\dr(\fnMna\dr(\fnMnb\fnMncClassical))) + \frac{1}{2}\fnMnb\fnMncClassical\dr(zV^Z),
\end{equation} 
where the first term $\frac{1}{4}\dr(\KDelta\dr(\fnMna\dr(\fnMnb\fnMncClassical)))$ is the same as in Regge-Wheeler case. Now we consider the second term $\frac{1}{2}\fnMnb\fnMncClassical\dr(zV^Z)$, which is given by
\begin{equation}\label{eq:pVZ}
\frac{1}{2}\fnMnb\fnMncClassical \dr(zV^Z)=\frac{1}{2}\fnMnb\fnMncClassical \dr(zV^{RW})(1+\zeta)+\frac{1}{2}\fnMnb\fnMncClassical zV^{RW}\dr \zeta.
\end{equation} 
Notice that $$f=-\frac{2(r-3M)}{r^4}.$$

{\bf In case I: $2M\leq r \leq 3M.$ }
\begin{equation}\label{rho-bound}
\begin{split}
 2M(2\bar\lambda+3)\leq &2\Lambda\leq 6M(\bar\lambda+1), \\
(2\bar\lambda+2)r\leq &2\Lambda\leq (2\bar\lambda+3)r.
\end{split}
\end{equation} 
And
\begin{equation}\label{bound-epsilon-1}
-\frac{3}{14} \leq \zeta(2M)=-\frac{3}{2(2\bar\lambda+3)} \leq \zeta \leq \zeta(3M)=-\frac{1}{4(\bar\lambda+1)^2}<0.
\end{equation}
Submitting \eqref{rho-bound} into \eqref{dr-epsilon-1}, we have
\begin{equation}\label{bound-dr-epsilon}
\frac{5}{14r} \leq \dr \zeta  \leq \frac{3}{4r}.
\end{equation} 
Since $2M\leq r \leq 3M,$  we have  $wf\geq0.$ Moreover, in view of \eqref{bound-dr-epsilon} and the fact that $ zV^{RW}<0$, we have the estimate for the second term in \eqref{eq:pVZ}  
\begin{equation}\label{Zerilli-V-1}
 \frac{1}{2}\fnMnb\fnMncClassical zV^{RW}\dr \zeta<0.
\end{equation}
Let us turn to the first term $\frac{1}{2}wf\dr(zV^{RW})(1+\zeta)$ in \eqref{eq:pVZ}.
However, $\dr(zV^{RW})$ has no sign.
Hence we separate $\dr(zV^{RW})$ into positive and negative part:
\begin{equation*}
\dr(zV^{RW})=8Mr^{-5}(3r-8M) =24Mr^{-5}(r-3M) + 8M^2r^{-5},
\end{equation*}
ignoring the negative part, we use the bound of $\zeta$ in \eqref{bound-epsilon-1} to obtain,
\begin{equation*}
\frac{1}{2}wf\dr(zV^{RW}) \zeta \geq - wf \frac{6M^2}{7r^5}.
\end{equation*}
Again, using the bound of $\dr \zeta$ in \eqref{bound-dr-epsilon} to get the lower bound for \eqref{Zerilli-V-1}, we obtain
\begin{equation}\label{v-1}
\begin{split}
\frac{1}{2}wf\dr(zV^Z) &=\frac{1}{2}wf\left(\dr(zV^{RW})(1+\zeta) +zV^{RW}\dr \zeta \right)\\
& \geq  \frac{1}{2}wf \left( (3r-8M) \frac{8M}{r^5} -  \frac{12M^2}{7r^5}  -(r-2M) \frac{6M}{r^5}  \right).
\end{split}
\end{equation}

With the above estimate \eqref{v-1} for the case $2M\leq r \leq 3M,$  the lower bound of the potential $\CurlyVLzZero + 6\uu$ in the Morawetz estimate is given by
\begin{equation}\label{bound-small-3M-0}
\mathcal{V}+6\uu\geq \mathcal{V}+6\uu\frac{2M}{r} \geq \frac{5M}{2r^2}- \left(13 + \frac{5}{7} \right)\frac{M^2}{r^3} +
\left(18+\frac{1}{7} \right) \frac{M^3}{r^4}.
\end{equation}
Additionally, using the fact that $\frac{1}{3} \leq \frac{M}{r} \leq \frac{1}{2},$ we finally get the lower bound for the potential for $r\leq 3M$:
\begin{equation}\label{bound-small-3M}
\mathcal{V}+6\uu \geq \underline{V}_{r\leq 3M} \doteq \frac{5M}{2r^2}  - \left(13 + \frac{2}{3} \right) \frac{M^2}{r^3}  +  \frac{18M^3}{r^4}.
\end{equation}

{\bf In case II:  $r\geq 3M.$} We have $wf\leq 0,$ and $ \dr(zV^{RW})>0.$ Additionally,
$$\zeta \leq \frac{3}{2 \bar\lambda}\leq \frac{3}{4}.$$ Hence in \eqref{eq:pVZ}
the first term has the lower bound $$\frac{1}{2}wf\dr(zV^{RW})(1+\zeta) \geq \frac{7}{8}wf\dr(zV^{RW}).$$
In view of the fact that $zV^{RW}<0$ and $\dr \zeta \geq 0$, the second term 
\begin{equation*}
\frac{1}{2}\fnMnb\fnMncClassical zV^{RW}\dr \zeta \geq 0.
\end{equation*}
Thus, we have the lower bound
\begin{equation}\label{zerlli-v-2}
\frac{1}{2}wf\dr(zV^Z) \geq \frac{7}{8}wf\dr(zV^{RW}).
\end{equation}
We can estimate  $\mathcal{V}+6\uu$ for $r\geq 3M$ as
\begin{equation}\label{bound-large-3M}
\mathcal{V}+6\uu\geq \mathcal{V}+6\uu\frac{3M}{r} \geq \underline{V}_{r\geq 3M} \doteq \frac{M}{2r^2} -\frac{4M^2}{r^3}  + \frac{7M^3}{r^4}.
\end{equation}

In summary, for case I, $wf\geq 0$, we had found a lower bound $L_{r\leq 3M}$ for $\dr(zV^Z)$, such that
\begin{equation}\label{bound-less-3M-general-1}
\mathcal{V}+6\uu\geq  \frac{1}{4}\dr(\KDelta\dr(\fnMna\dr(\fnMnb\fnMncClassical))) + \frac{1}{2}\fnMnb\fnMncClassical L_{r\leq 3M}+6\uu\frac{2M}{r} = \underline{V}_{r\leq 3M}.
\end{equation}
For case II, $wf\leq 0$, we had found a upper bound $U_{r\geq 3M}$ for $\dr(zV^Z)$, such that
\begin{equation}\label{bound-greater-3M-general-1}
\mathcal{V}+6\uu\geq  \frac{1}{4}\dr(\KDelta\dr(\fnMna\dr(\fnMnb\fnMncClassical))) + \frac{1}{2}\fnMnb\fnMncClassical U_{r\geq 3M}+6\uu\frac{3M}{r} = \underline{V}_{r\geq 3M}.
\end{equation}
Both $wf\big|_{r=3M} =0$ and $U\big|_{r=3M}=0$ hold, it would be obviously to see that $\left(\underline{V}_{r\leq 3M} - \underline{V}_{r\geq 3M}\right)\big|_{r=3M}=0$.
In particular, we have
\begin{equation}\label{compare-v}
\underline{V}_{r\leq 3M} - \underline{V}_{r\geq 3M} = (r-3M) \left(2r-\frac{11M}{3} \right) \frac{M}{r^4},
\end{equation} which vanishes at $r=3M.$ 

\end{proof}

{\bf Step II: The hypergeometric functions associated to the Morawetz potential.}
Denote the lower bound for the Morawetz potential $\mathcal{V}+6\uu$ in both cases by
\begin{equation}\label{V1-V2}
\underline{V}_{r\leq 3M}\doteq V_{1}, \quad \text{and} \quad \underline{V}_{r\geq 3M}\doteq V_{2}.
\end{equation}
Notice that \eqref{compare-v} implies
\begin{equation}\label{v12-3M}
V_1|_{r=3M}= V_2|_{r=3M}.
\end{equation}
We therefore define the potential in the whole region $r\geq 2M$ by
\begin{equation}\label{V-global}
 V_{joint} \doteq
\begin{cases}
V_1, &\text{if } 2M \leq r \leq 3M, \\
V_2, &\text{if } r \geq 3M.
\end{cases}
\end{equation}
This is the lower bound for the Morawetz potential in the Zerilli case, and we know that $V_{joint} \in C^0$. In this case, the  Hardy type estimate is reduced to finding a positive solution to the ordinary differential equation 
\begin{equation}\label{ode-v}
-\partial_r( A \partial_r \phi) + V\phi=0,
\end{equation}
on the interval $r\in [2M, +\infty)$ with
\begin{equation} \label{A-V-Zerilli}
A=\CurlyALzZero=\frac{M\KDelta^2}{r^4}, \quad V= V_{joint}.
\end{equation}
If $\phi$ is a solution to equation \eqref{ode-v} with $A$ and $V$ being specified in the Zerilli case by \eqref{A-V-Zerilli}, we apply the transformation (see Lemma \ref{lemma-hypergeometric-tansformations})
\begin{equation} \label{u-v-tranf}
u=A^{\frac{1}{2}} \phi, \quad x=r-2M.
\end{equation}
Then $u$ solves the new ordinary differential equation
\begin{equation}\label{ode-u}
-\partial_x^2u + Wu=0,
\end{equation}
on the interval $x\in [0, +\infty)$ with 
\begin{equation}\label{W-global}
 W \doteq
\begin{cases}
W_1= \frac{1}{6}\frac{15x^2-46Mx+4M^2}{x^2(x+2M)^2}, \,\,\, &x \leq M, \\
W_2=\frac{1}{2}\frac{x^2-12Mx+2M^2}{x^2(x+2M)^2}, \,\,\, &x > M.
\end{cases}
\end{equation}

We will apply the scheme in Lemma \ref{lemma-hypergeometric-tansformations} to this case, and explore the hypergeometric functions associated to \eqref{ode-u}, \eqref{W-global} in each of  the two region $x\leq M$ and $x>M$. To do the calculation explicitly, we may set $M=1.$ Index the intervals $x\leq 1$ and $x>1$ by $i=1,2,$ respectively. Index the hypergeometric (Frobenius) solutions in each interval by $j=1, 2.$  

\begin{figure} 
\centering 
\subfigure[Positive solution $u_{11}$]{
 \label{fig:subfig:u11} 
 %% label for first subfigure 
 \includegraphics[width=2.0in]{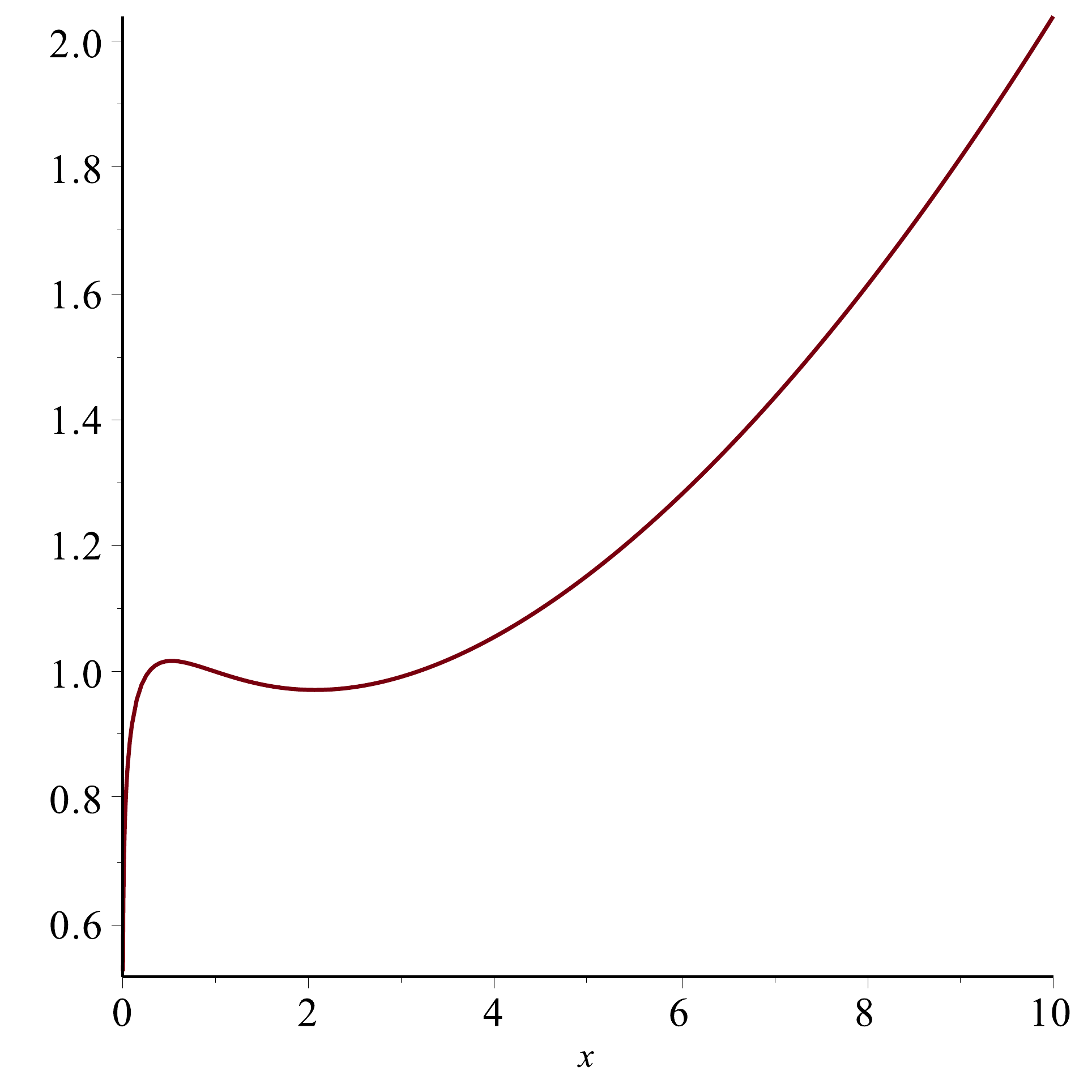}
 } 
  \hspace{0.1in} 
 \subfigure[Positive solution $u_{21}$]{ 
 \label{fig:subfig:u21} 
 %% label for second subfigure
  \includegraphics[width=2.0in]{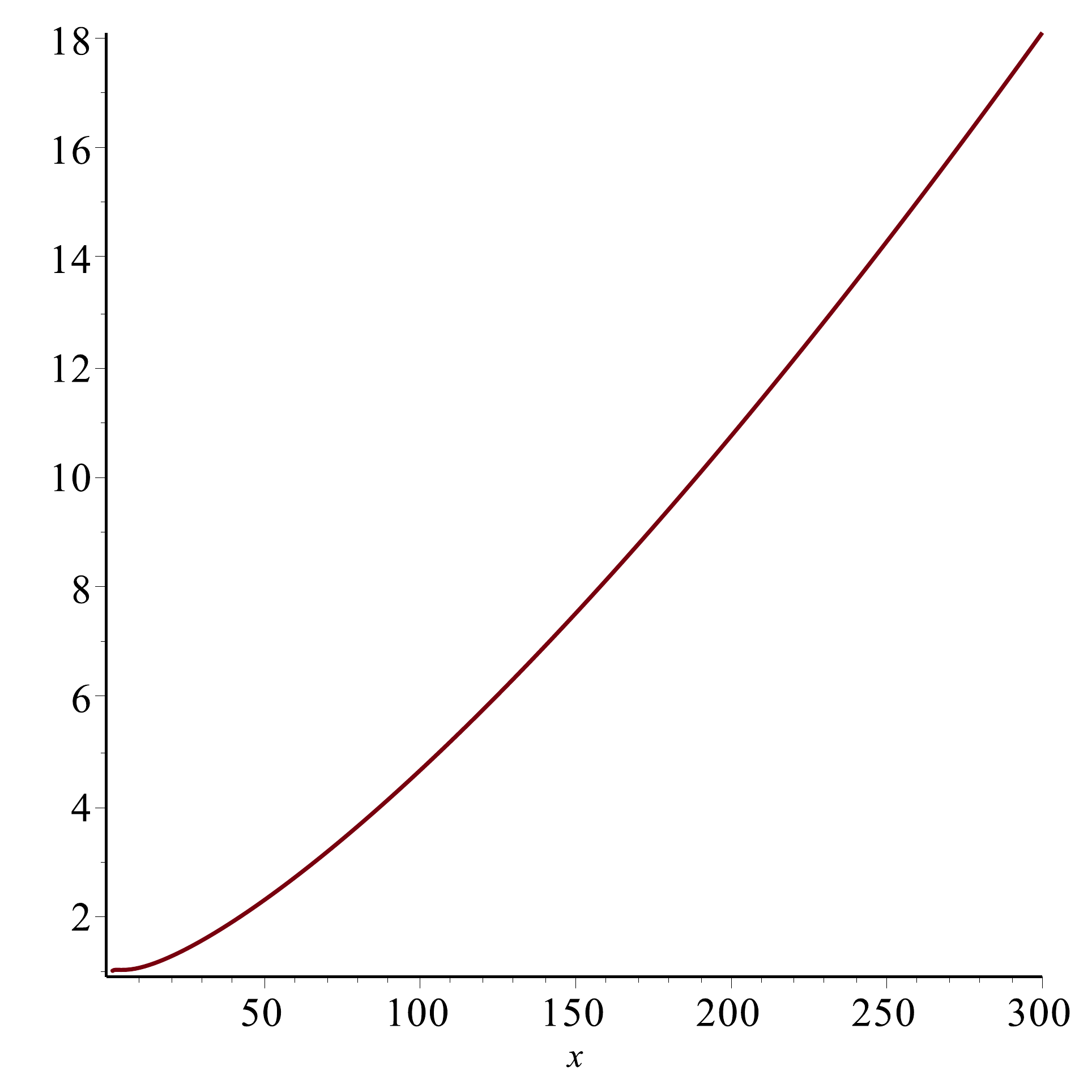}
  } 
  \caption{Positive hypergeometric functions in Zerilli case} \label{fig:subfig-u11-u21}
   %% label for entire figure 
\end{figure}
\begin{lemma}[The hypergeometric differential equation associate to the Zerilli equations]\label{The hypergeometric differential equation associate to the Zerilli equations}
For $x \leq 1$, there are two Frobenius solutions $u_{11}, u_{12}$ (defined in \eqref{u-11-12}) to the ODE \eqref{ode-u}, \eqref{W-global} and $u_{11}$ is positive (see Figure \ref{fig:subfig:u11}).

For $x > 1$, there are two Frobenius solutions $u_{21}, u_{22}$ (defined in \eqref{u-ij} ) to the ODE \eqref{ode-u}, \eqref{W-global}  and $u_{21}$ is positive (see Figure \ref{fig:subfig:u21}).
\end{lemma}

\begin{proof}[Proof of Lemma \ref{The hypergeometric differential equation associate to the Zerilli equations}]\label{lemma-hypergeometric-ode-Zerilli}
We find that two linearly independent solutions (adapted to $x=0$) to \eqref{ode-u} in $x \leq1$ taking the form of
\begin{equation}\label{u-11-12}
\begin{split}
u_{11} &=x^{\alpha_1}(x+2)^{\beta_1}F\left(a_1,b_1,c_1;-\frac{x}{2} \right),\\
 u_{12} &=x^{\alpha_1}(x+2)^{\beta_1}x^{1-c_1}F\left(b_1-c_1+1,a_1-c_1+1,2-c_1;-\frac{x}{2}\right),
\end{split}
\end{equation}
where we follow Lemma \ref{lemma-hypergeometric-tansformations} to calculate the parameters,
\begin{equation*}
\alpha_1=\frac{1}{2} \pm \frac{\sqrt{15}}{6}, \quad \beta_1=\frac{1}{2} \pm \frac{3\sqrt{3}}{2}.
\end{equation*}
We take the $+$ sign in $\alpha_1$ and the $-$ sign in $\beta_1$. Then the $a_1, b_1, c_1$ can be immediately read off
 in terms of $\alpha_1$ and $\beta_1$,
\begin{align*}
c_1 &=\frac{\sqrt{15}}{3}+1, \\
a_1 +b_1 &= 1 + \frac{\sqrt{15}}{3} - 3\sqrt{3}, \\
a_1 b_1 &= \frac{14}{3}+\frac{\sqrt{15}}{6} - \frac{3\sqrt{5}}{2} - \frac{3\sqrt{3}}{2}.
\end{align*}
The remaining $a_1, b_1$ could be solved
\begin{equation}\label{abc-1}
\{ a_1, b_1\} = \{ \frac{3+\sqrt{15}-9\sqrt{3} \pm 3\sqrt{11}}{6} \}.
\end{equation}
We choose $a_1 < b_1$ such that
\begin{equation}\label{abc-1-numerical-values}
a_1\approx-3.11<0<b_1\approx2.21<c_1\approx 2.29.
\end{equation}
To be clear, we also give the value of $\alpha_1$ and $\beta_1$,
\begin{equation}\label{alpha-beta-1}
\alpha_1=\frac{1}{2} + \frac{\sqrt{15}}{6}, \quad \beta_1=\frac{1}{2} - \frac{\sqrt{15}}{6}, \quad \alpha_1 +\beta_1=\frac{a_1+b_1+1}{2}.
\end{equation}
For $F(a_1,b_1,c_1;-\frac{x}{2})$, its second and third parameters satisfying $c_1>b_1>0$. We can use the integral representation \eqref{eq:Hardy-Hypergeometric-Integral-Representation} to show that
\begin{equation*}
F\left(a_1,b_1,c_1;-\frac{x}{2} \right)>0， \quad \text{for} \,\,\,\,x>0,
\end{equation*}
which says that $u_{11}$ is positive.

Similarly, let $u_{2j}, j=1,2$ denote the Frobenius solutions (adapted to $x=0$) to \eqref{ode-u} in $x >1$, 
\begin{equation}\label{u-ij}
\begin{split}
 u_{21} &=x^{\alpha_2}(x+2)^{\beta_2}F\left(a_2,b_2,c_2;-\frac{x}{2}\right), \\
 u_{22} &=x^{\alpha_2}(x+2)^{\beta_2}x^{1-c_2} F\left(a_2-c_2+1,b_2-c_2+1,2-c_2;-\frac{x}{2}\right),
\end{split}
\end{equation}
where the parameters
\begin{equation*}
\alpha_2=\frac{1 \pm \sqrt{2} }{2}, \quad \beta_2=\frac{1 \pm 4}{2}.
\end{equation*}
We will follow Lemma \ref{lemma-hypergeometric-tansformations} to chose the parameters.
We take the $+$ sign in $\alpha_2$ and the $-$ sign in $\beta_2$. Then the $a_2, b_2, c_2$ can be immediately read off
 in terms of $\alpha_2$ and $\beta_2$,
\begin{align*}
c_2 &=1+\sqrt{2}, \\
a_2 +b_2 &= \sqrt{2} - 3, \\
a_2 b_2 &=2 - \frac{3\sqrt{2}}{2}.
\end{align*}
The remaining $a_2, b_2$ could be solved
\begin{equation}\label{abc-2}
\{ a_2, b_2\} = \{\frac{\sqrt{2}-3 \pm \sqrt{3}}{2} \}.
\end{equation}
We choose $a_2 < b_2$ such that
\begin{equation}\label{abc-2-numerical-values}
a_2\approx-1.66<0<b_2\approx0.07<c_2\approx2.41.
\end{equation}
We also remark that here we have
\begin{equation}\label{alpha-beta-2}
\alpha_2=\frac{1}{2}+\frac{\sqrt{2}}{2}, \quad \beta_2=-\frac{3}{2}, \quad \alpha_2+\beta_2=\frac{a_2+b_2+1}{2}.
\end{equation}
Those value will be useful in proving Theorem \ref{thm-Positive hypergeometric function for Zerilli equations} in Step III.
We find that for $F\left(a_2,b_2,c_2;-\frac{x}{2}\right)$, its second and third parameters satisfying $c_2>b_2>0$. Again we can use the integral representation \eqref{eq:Hardy-Hypergeometric-Integral-Representation} to show that 
\begin{equation*}
F\left(a_2,b_2,c_2;-\frac{x}{2}\right)>0， \quad \text{for} \,\,\,\,x>0.
\end{equation*}
That is $u_{21}$ is positive.
\end{proof}

{\bf Step III: Constructing the positive $C^2$ solution.}
We first calculate some quantities which will be useful in constructing the positive $C^2$ solution to the ordinary differential equation \eqref{ode-u}, \eqref{W-global}.
 We normalize $u_{11}$ \eqref{u-11-12}  so that $u_{11}(1)=1$, and let
\begin{equation}\label{def-w-11}
w_{11}=\frac{du_{11}}{dx}\big|_{x=1}.
\end{equation}
We have 
\begin{equation}\label{w-11}
w_{11}\approx 0.6184539934\cdots
\end{equation}
For $j=1,2,$ we normalize $u_{2j}$ \eqref{u-ij} so that $u_{2j}(1)=1$, and let
\begin{equation}\label{def-w-21-22}
w_{2j}=\frac{du_{2j}}{dx}\big|_{x=1}.
\end{equation}
We have 
\begin{equation}\label{w-21-22}
w_{21}=0.7340312856\cdots, \quad w_{22}=-0.3321954186\cdots
\end{equation}
Additionally, we observe that
\begin{equation}\label{limit-u21-u22}
\lim_{x\rightarrow +\infty} \frac{u_{22}}{u_{21}}\doteq -\Lambda =-5.0153723738\cdots
\end{equation}

\begin{figure} 
\centering 
\subfigure[Positive solution $u$]{
 \label{fig:subfig:a} 
 %% label for first subfigure 
 \includegraphics[width=2.0in]{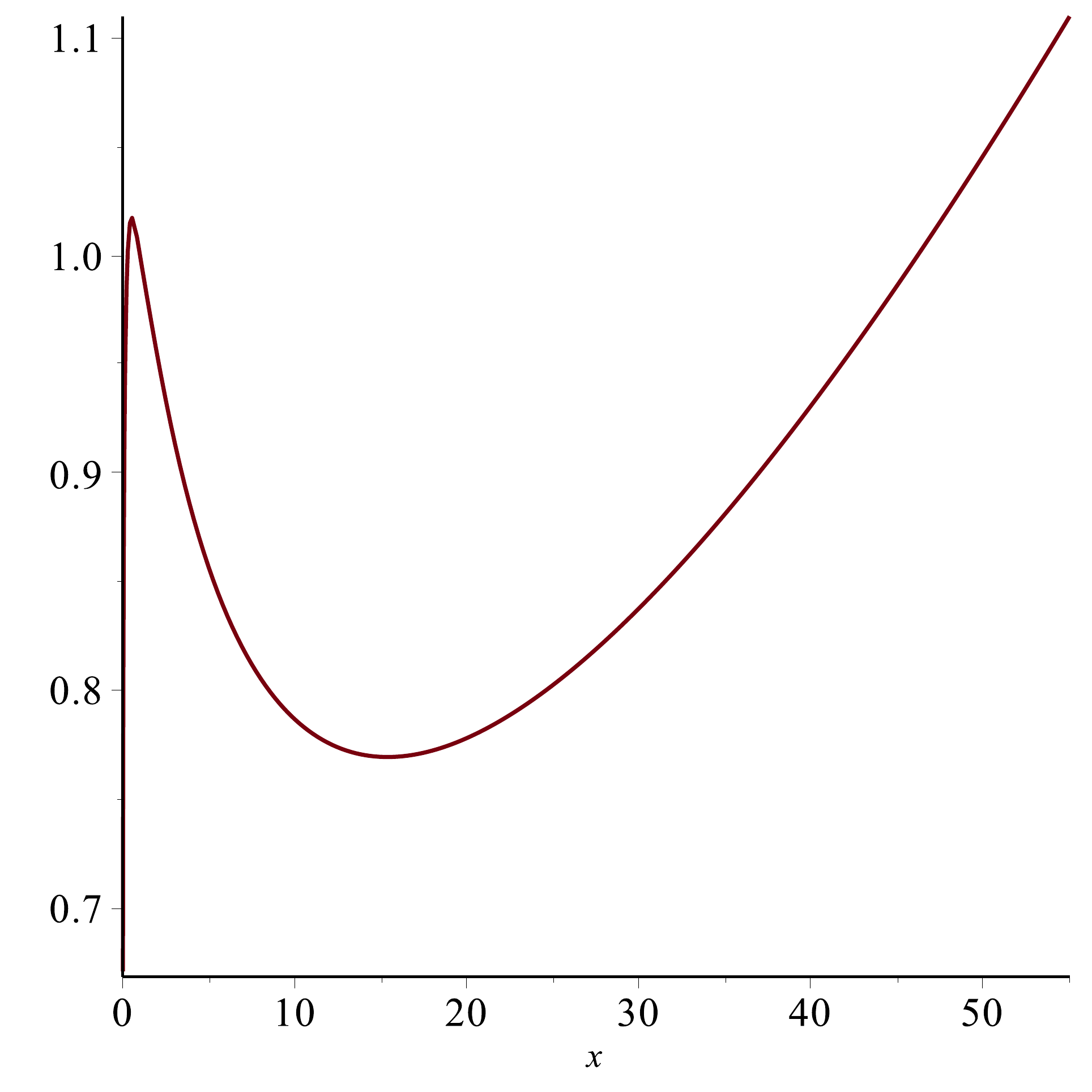}
 } 
  \hspace{0.1in} 
 \subfigure[$u$ near $x=1$]{ 
 \label{fig:subfig:b} 
 %% label for second subfigure
  \includegraphics[width=2.0in]{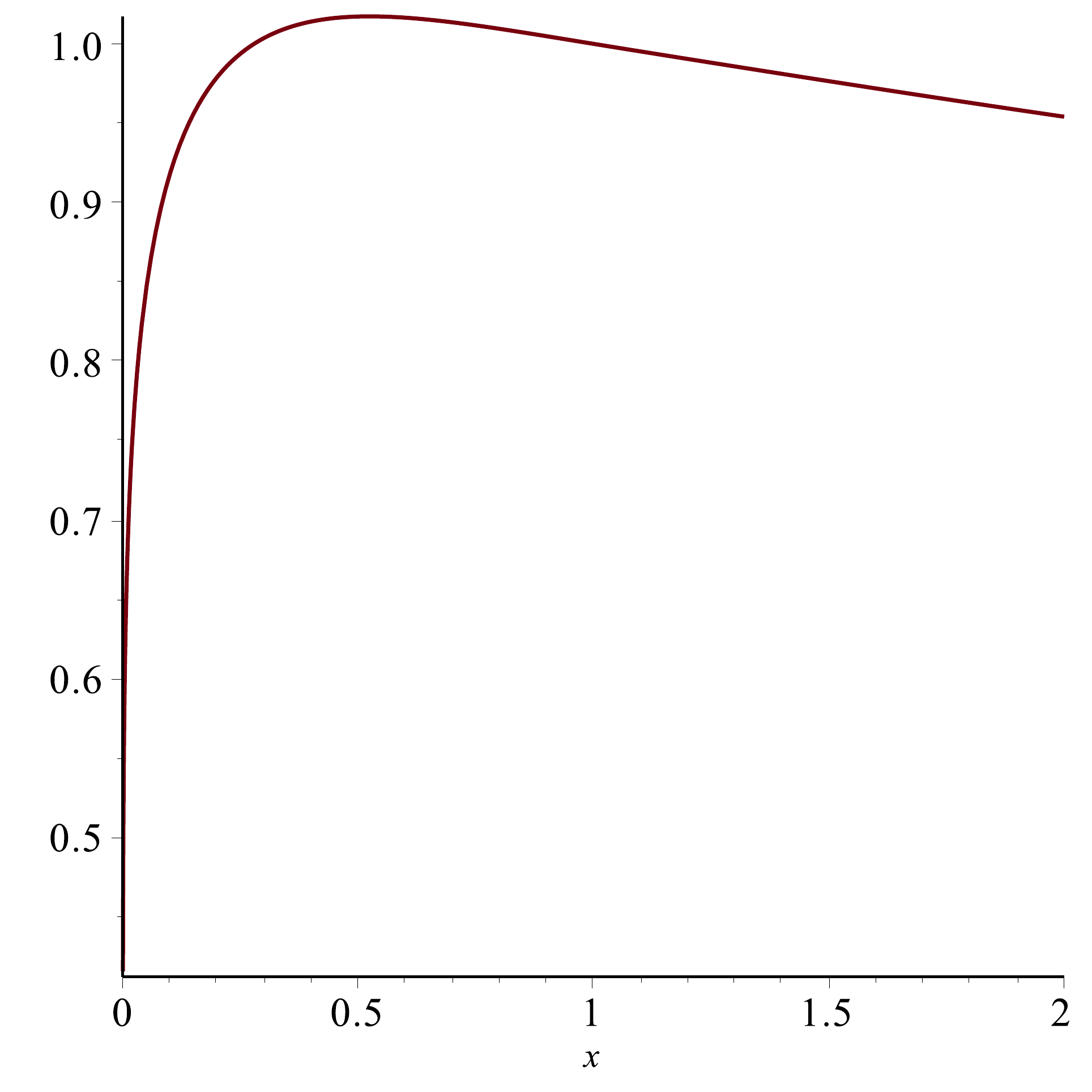}
  } 
  \caption{The positive $C^2$ solution $u$} \label{fig:subfig-u}
   %% label for entire figure 
\end{figure}
\begin{theorem}[Positive hypergeometric function for Zerilli equations]\label{thm-Positive hypergeometric function for Zerilli equations}
We define $u$, normalized to $u(1) = 1$, by
\begin{equation}\label{u-joint}
 u =
\begin{cases}
u_{11}, \,\,\, &x \leq 1 \\
\omega u_{22}+(1-\omega)u_{21}, \,\,\, &x > 1
\end{cases}
\end{equation}
where $\omega$ is given by 
\begin{equation}\label{omega}
\omega=\frac{w_{11}-w_{21}}{w_{22}-w_{21}}=0.1083984220\cdots
\end{equation}
Then $u$ is indeed a positive $C^2$ solution (see Figure \ref{fig:subfig:a}, \ref{fig:subfig:b}) to the ordinary differential equation \eqref{ode-u}, \eqref{W-global}.
\end{theorem}
\begin{remark}
In the proof of Theorem \ref{thm-Positive hypergeometric function for Zerilli equations}, we had used the numerical value of $w_{11}$ \eqref{w-11}, $w_{2j}, \, j=1,2$ \eqref{w-21-22} and $\Lambda$ \eqref{limit-u21-u22} to show the positivity of $u$.
\end{remark}

\begin{proof}[Proof for Theorem \ref{thm-Positive hypergeometric function for Zerilli equations}]
First, we note that, with the above choice \eqref{u-joint}, $u$ is actually $C^1$. Hence, $u$ defined in \eqref{u-joint} is now a $C^1$ solution to \eqref{ode-u}, \eqref{W-global}.

By construction, we have $u>0$ for $0<x\leq1$, since $u_{11}$ is positive (see Lemma \ref{The hypergeometric differential equation associate to the Zerilli equations}). One needs to check that for $x>1$, $u>0$. Notice that $u_{22}$ is positive (see Lemma \ref{The hypergeometric differential equation associate to the Zerilli equations}). We wish to prove that in $x>1$, $(1-\omega)u_{21}$ dominates $\omega u_{22}$, so that $\omega u_{22}+(1-\omega)u_{21} >0$. 

From now on, we will focus on the hypergeometric differential equation with $W=W_2$ in $x>1$,
\begin{equation}\label{hyper-o-d-e-2w2}
-u''(x)+ W_2u(x)=0.
\end{equation}
We set $z=-\frac{x}{2}$. Then \eqref{hyper-o-d-e-2w2} has a solution taking the following form 
\begin{equation}\label{hyper-o-d-e-2-solu}
\bar{u}=x^{\alpha_2}(x+2)^{\beta_2}F(z),
\end{equation}
where $F(z)$ is a solution to the hypergeometric differential equation 
\begin{equation}\label{hypergeometric-differential-equation-w2}
z(1-z)\frac{d^2w}{d^2z}+(c_2-(a_2+b_2+1)z)\frac{dw}{dz}-a_2b_2w=0
\end{equation}
with the parameters $a_2, b_2, c_2$ defined in \eqref{u-ij}. Note that, the hypergeometric function has possibly regular singularities at $ z= 0, 1,\infty$, namely, $ x = 0, -2, \infty$.
For $x>1,$ we only focus on the regular singularity $x=\infty$.
There are the following pairs $f_1(z), f_2(z)$ of Frobenius solutions to \eqref{hypergeometric-differential-equation-w2}, which is adapted to $z=\infty$ (see Lemma \ref{lemma-hypergeometric-tansformations}),
\begin{equation}\label{asym-infty-w2}
\begin{split}
f_1(z)&=z^{-a_2}F\left(a_2, a_2-c_2+1;a_2-b_2+1;\frac{1}{z}\right),\\
f_2(z)&=z^{-b_2}F\left(b_2, b_2-c_2+1;b_2-a_2+1;\frac{1}{z}\right).
\end{split}
\end{equation}
Submitting \eqref{asym-infty-w2} into \eqref{hyper-o-d-e-2-solu}, we know that there are a pair of Frobenius solutions $\bar{u}_{1}, \bar{u}_{2}$ to \eqref{hyper-o-d-e-2w2}, which are adapted to the singularity $x=\infty$,  
\begin{equation*}
\bar{u}_{1}=x^{\alpha_2}(x+2)^{\beta_2}f_1,\quad \bar{u}_{2}=x^{\alpha_2}(x+2)^{\beta_2}f_2.
\end{equation*}
In view of \eqref{alpha-beta-2}, we can calculate that the characteristic exponents of singularity $x=\infty$ are
\begin{equation*}
\frac{b_2-a_2+1}{2} \quad \text{and} \quad \frac{a_2-b_2+1}{2},
\end{equation*}
which could be further specified by the parameters $a_2, \, b_2$ in \eqref{u-ij}. As a result, we have the asymptotic expansion for $\bar{u}_{1}$ and $\bar{u}_{2}$
\begin{equation}\label{asym-u2-infty}
\bar{u}_{1} \sim x^{\frac{b_2-a_2+1}{2}} \rightarrow \infty, \quad \bar{u}_{2} \sim x^{\frac{a_2-b_2+1}{2}} \rightarrow 0, \quad \text{as} \,\,x \rightarrow \infty.
\end{equation}
As a remark, the parameters $a_2, b_2$ could be chosen in various ways, but the resulting characteristic exponents would be always the same.
Additionally, we could see that $(-1)^{b_2} \bar{u}_2$ is positive. For
\begin{equation}\label{bar-u-1-positive-1}
\begin{split}
(-1)^{b_2} \bar{u}_2 &= 2^{b_2} x^{\alpha_2-b_2}(x+2)^{\beta_2}F\left(b_2-c_2+1, b_2, ;b_2-a_2+1;-\frac{2}{x}\right),\\
&= 2^{b_2} x^{\alpha_2-b_2}(x+2)^{\beta_2}F\left(b_2-c_2+1, b_2, ;b_2-a_2+1;-\frac{2}{x}\right).
\end{split}
\end{equation}
where we had used the fact that the hypergeometric is symmetric in its first two arguments: $F(a,b;c;z)=F(b,a;c;z)$. And for $F\left(b_2-c_2+1, b_2, ;b_2-a_2+1;-\frac{2}{x}\right)$, the second and third two arguments satisfying $b_2-a_2+1> b_2>0,$ we can use the integral representation \eqref{eq:Hardy-Hypergeometric-Integral-Representation} to show that $$F\left(b_2-c_2+1, b_2, ;b_2-a_2+1;-\frac{2}{x}\right)>0 \quad \text{for} \quad x>0.$$

The general solution to \eqref{hyper-o-d-e-2w2}, which could be written as linear combinations of the Frobenius solutions $u_{21}, u_{22}$ (adapted to $x=0$), is either asymptotically decaying as $\bar{u}_{1}$ or as $\bar{u}_{2}$. Due to the observation \eqref{limit-u21-u22}, we will construct a combination of the Frobenius solutions $u_{21}, \, u_{22}$, which will be denoted by $G(x)$ below, such that $G(x)$ is asymptotically decaying as $\bar{u}_{2}$, and hence proportional to $\bar{u}_2$.  Additionally, $G(x)$ serves as a positive lower bound for $\omega u_{22}+(1-\omega)u_{21}$. In this way, we would prove the positivity for $\omega u_{22}+(1-\omega)u_{21}$.

Recalling \eqref{limit-u21-u22}, we have
\begin{equation}\label{Lambda}
\Lambda=5.0153723738\cdots
\end{equation}
We define a new function 
\begin{equation}\label{def-G}
G(x)=u_{22}+\Lambda u_{21}.
\end{equation}
We wish to prove that $G(x)>0$ for $x>0$.
Now $G(x)$ is a solution of the differential equation \eqref{hyper-o-d-e-2w2} with $G(1) = 1+\Lambda>0$. Moreover, by the construction, we know that 
\begin{equation}\label{def-G-infty-1}
 \lim_{x \rightarrow \infty}\frac{G(x)}{u_{21}(x)}= 0.
\end{equation}
With respect to the exponents of the two  Frobenius solutions \eqref{asym-u2-infty} adapted to singularity $x=\infty$, we have either
\begin{equation*}
G(x) \sim x^{\frac{b_2-a_2+1}{2}} \rightarrow \infty \quad \text{or} \quad G(x) \sim x^{\frac{a_2-b_2+1}{2}} \rightarrow 0 \quad \text{as} \,\, x \rightarrow \infty.
\end{equation*}
The fact of \eqref{def-G-infty-1} further shows that 
\begin{equation*}
G(x) \sim x^{\frac{a_2-b_2+1}{2}} \rightarrow 0 \quad \text{as} \quad x \rightarrow \infty.
\end{equation*}
 As a summary, $G(x)$ is a solution of the ordinary differential equation \eqref{hyper-o-d-e-2w2} with 
\begin{equation}\label{bdy-G}
G(1) = 1+\Lambda, \quad \lim_{x \rightarrow \infty}G(x)= 0.
\end{equation}
On the other hand, $G(x)$ could be expressed in terms of linear combinations of the  Frobenius solutions $\bar{u}_{1}$ and  $\bar{u}_{2}$: Suppose 
\begin{equation}\label{express-G}
G(x)=p \bar{u}_{1}(x)+ q\bar{u}_{2}(x),
\end{equation}
where $p,q$ are some constants. Taking value at $x=\infty$ yields that
\begin{equation}\label{express-G-infty}
\lim_{x \rightarrow \infty}G(x)=p \lim_{x \rightarrow \infty}\bar{u}_{1}(x)+ q \lim_{x \rightarrow \infty}\bar{u}_{2}(x),
\end{equation}
which gives $0=p\cdot \infty + q\cdot 0.$  Hence $p=0$ and
\begin{equation}\label{approx-G-1}
G(x)=q\bar{u}_{2}(x).
\end{equation}
We have known that $(-1)^{b_2} \bar{u}_2$ is positive. Hence \eqref{approx-G-1} implies that $G(x)$ does not change sign. Besides, we know that $G(1) = 1+\Lambda>0$, therefore $G(x)>0$ for $x>0.$

Next, we turn back to $\omega u_{22}+(1-\omega )u_{21}$, which could be written as 
\begin{equation}\label{Positive-Zerilli-u2}
\omega u_{22}+(1-\omega )u_{21} = \omega \left(u_{22}+\frac{1-\omega}{\omega}u_{21}\right). 
\end{equation}
Viewing the value of $\omega$ in \eqref{omega} and $\Lambda$ in \eqref{Lambda}, we know that
\begin{equation}\label{omega-Lambada}
\frac{1-\omega}{\omega} >\Lambda.
\end{equation}
Additionally, $u_{21}$ is positive (see Lemma \ref{The hypergeometric differential equation associate to the Zerilli equations}). Hence,
\begin{equation}\label{comp-Zerilli-u2-G}
\omega u_{22}+(1-\omega )u_{21} > \omega(u_{22}+\Lambda u_{21})=\omega G(x)>0.
\end{equation}
Therefore, we had proved that the $C^1$ solution $u$ defined by \eqref{u-joint} is a positive (see Figure \ref{fig:subfig:a}) solution  to \eqref{ode-u}.  Actually, \eqref{ode-u} tells that $u$ is also $C^2$ (see Figure \ref{fig:subfig:b}), since the transformation \eqref{u-v-tranf} takes  $V_{joint}$ to $W$, and $V_{joint} \in C^0$, hence $W$ is continuous too.

Finally, $\phi$ defined as \eqref{u-v-tranf} would be then a positive $C^2$ solutions to \eqref{ode-v}.
\end{proof}

Now we are ready to prove the integrated decay estimate for the Zerilli case.

\begin{proof}[Proof of Theorem \ref{Integrated decay estimate} for Zerilli equations]

We have constructed a $C^2$ function $\phi$ \eqref{u-v-tranf}, which is a positive solution to \eqref{ode-v} with $V=V_{joint}$.
Thus the Hardy inequality follows.
%\begin{equation}\label{inter-per}
%\int_{2M}^{\infty} \left( A(\partial_r \phi)^2 + (\mathcal{V}+6\uu) \phi^2 \right) \di r \geq  \int_{2M}^{\infty}  \left( \frac{\Delta^2}{r^4}(\partial_r \phi)^2 + \varepsilon_{Hardy,1} \frac{\phi^2}{r^2} \right)  \di r.
%\end{equation}
Hence we complete the proof for the Zerilli case.
\end{proof}

\begin{remark}\label{rk-sigma-t-l-i-d-e}
We can apply the proof of Theorem \ref{Integrated decay estimate} to the spacetime foliation $\bigcup_{\tau} \Sigma_\tau$, hence for $\tau_0 < \tau$ 
\begin{equation}\label{Morawetz-2-nontrap-sigma-t}
\begin{split}
&\int_{\tau_0}^{\tau} \di t\int_{\Sigma^\prime_t} \left( (1-\mu) |\p_r\solu|^2+ |\solu|^2+ |\p_t\solu|^2 +|\nablaslash\solu|^2 \right) \di \mu_g  \\
\lesssim &\int_{\Sigma_{\tau_0}} \sum_{i\leq 1} P^T_{\alpha}(T^i\solu)n_{\Sigma_{\tau_0}}^{\alpha},
\end{split}
\end{equation}
where $\Sigma^\prime_t= \Sigma_t \cap \{r \leq R\}$.
\end{remark}

\subsection{Non-degenerate energy}\label{non-deg-energy}

With the integrated decay estimate, we could use the vector field $N$ to prove the uniform boundedness for the non-degenerated energy.

\begin{proof}[Proof of Theorem \ref{uniform-bound}]
For solution $\solu$ of the equation \eqref{eq-V}, taking the vector field $\xi =T,$ the associated energy on $t=\tau$ slice is 
\begin{equation*}
E^T(\solu,\tau)=\frac{1}{2}
\int_{t=\tau} \left( |\p_t\solu|^2  +|\p_{r^\ast}\solu|^2+ (1-\mu)( |\nablaslash\solu|^2+V_g\solu^2) \right) r^2\di r^\ast \di \sigma_{S^2}.
\end{equation*}
In presence of the factor $(1-\mu)$, the energy $E^T(\solu, \tau)$ is degenerate on the horizon $\hori$. We shall use the globally time-like vector field $N$ and consider the associated energy. 

Away from the horizon $\Sigma^e_\tau = \{\M |t=\tau \} \cap  \{  r> r_{NH} \}$.
Note that $N=T$ for $r>r_{NH}$. For solutions of Regge-Wheeler or Zerilli equations, they both have angular frequency $\ell\geq2$, which implies that \eqref{angular-int} holds.
This indeed gives the positivity of energy on $\Sigma^e_\tau$:  In the case of Regge-Wheeler equation,
\begin{equation*}
E^N(\solu,  \Sigma^e_\tau)\geq\frac{1}{2}
\int_{\Sigma^e_\tau} \left( |\p_t\solu|^2 + |\p_{r^\ast}\solu|^2 +(1-\mu) \left( \frac{|\nablaslash\solu|^2}{4}+\frac{ \solu^2}{r^2} \right) \right)  r^2\di r^\ast \di \sigma_{S^2}>0.
\end{equation*}
In the case of Zerill equation, 
\begin{equation*}
\begin{split}
\int_{\Sigma^e_\tau} |\nablaslash\solu|^2+V^Z_g \solu^2 {}&{}=\int_{\Sigma^e_\tau}  \left(2\lambda +2 - \frac{ \mu (2\lambda+3) ( 2\lambda +\frac{3}{2}\mu )}{ (\lambda +\frac{3}{2}\mu)^2 } \right) \frac{\solu^2}{r^2} \\
{}&{} \geq \int_{\Sigma^e_\tau}  \left( 2\lambda -2 + \frac{6}{ 2\lambda +3 } \right) \frac{\solu^2}{r^2} >0.
\end{split}
\end{equation*}
Note that, $(1-\mu)= 1-\frac{2M}{r} \approx 1$ for $r>r_{NH}$.

Near the horizon, recalling that $ \Sigma^i_\tau\doteq \{\M |v=\tau +r_{NH}^\ast \} \cap  \{  r \leq r_{NH} \}$, the energy
\begin{equation}\label{nondeg-energy-horizon}
E^N(\solu,\Sigma^i_\tau)
\approx \int_{\Sigma^i_\tau}\left( \frac{|\p_u\solu|^2}{1-\mu} +(1-\mu)(|\nablaslash\solu|^2+V_g|\solu|^2)\right)\di u\di\sigma_{S^2}.
\end{equation}
The non-degeneracy of $E^N(\solu,\Sigma^i_\tau)$ is more apparent if we write this integral in $(v,r,\theta,\phi)$ coordinate, which up to some constant is $\int_{\Sigma^i_\tau}( |\p_r\solu|^2 +|\nablaslash\solu|^2+V_g\solu^2)\di r\di\sigma_{S^2}.$
Due to the fact that $\ell\geq2$, we have
\begin{equation}\label{non-deg-energy-horizon-estimate}
\int_{\Sigma^i_t} \left( |\p_r\solu|^2 +|\nablaslash\solu|^2+\frac{|\solu|^2}{r^2} \right)\di r\di\sigma_{S^2}\lesssim E^N(\solu,\Sigma^i_\tau).
\end{equation}
Taking $\xi =N$, we apply the energy identity with the momentum vector $P^N(\solu)$. Noting that $N=T$ is killing for $r>r_{NH}$, we have
\begin{equation}\label{energy-id-N}
\begin{split}
&\int_{\Sigma_\tau} P_\alpha^N(\solu)n^\alpha + \int_{\hori}P_\alpha^N(\solu)n^\alpha + \doubleint_{ \{ r<r_0 \} }K^N(\solu) \\
=  &  \int_{\Sigma_{\tau_0}} P_\alpha^N(\solu)n^\alpha - \doubleint_{\{r_0<r<r_{NH} \}}  K^N(\solu).
\end{split}
\end{equation}
For $r<r_0$,
\begin{equation*}
 K^N(\solu)\geq b P^N_{\alpha}(\solu)n_{\Sigma_\tau}^{\alpha},
\end{equation*}
for some constants $b>0$. This is the red-shift effect, which allows us to control the non-degenerate energy on the horizon. 

We should combine  the conservation of the degenerate energy associated to the multiplier $\p_t$ with the integrated decay estimate Theorem \ref{Integrated decay estimate} to derive the uniform boundedness of non-degenerate energy.

Given data on $\Sigma_t$ for $\psi$, one may impose data along the horizon to the past of $\Sigma_t$. This data can be chosen so as to be invariant under the flow along $T$, i.e. to depend only on the angular variables and to match $\psi$ where $\Sigma_t$ meets the horizon. Let $\tilde{\psi}$ denote the solution generated by evolving this data both forward and backward in time. Because $n$ is proportional to $T$ along the horizon, $P^T_\alpha (\psi) n^\alpha$ is proportional to $P^T_\alpha(\psi)T^\alpha=(1/2) (T\psi)^2$, one finds  $P^T_\alpha (\psi) n^\alpha=0$ along the horizon, and in particular the flux through any portion of the horizon is $0$. Thus, $\tilde{\psi}$ is equal to $\psi$ in the future of $\Sigma_t$ and satisfies 
\begin{equation}\label{conservation-degenerate-energy-sigma-t-t}
E^T(\psi,\Sigma_\tau)=E^T(\tilde{\psi},\tau).
\end{equation}
Thus, the last integral of \eqref{energy-id-N}
\begin{equation}\label{integrated-energy-sigma-t-t}
\begin{split}
\doubleint_{\{r_0<r<r_{NH}\}}K^N(\psi) &\lesssim \int_{t=\tau_0} \sum_{i\leq 1}P^T_\alpha(T^i \tilde{\psi})n^\alpha \\
& \lesssim \int_{\Sigma_{\tau_0}} \sum_{i\leq 1} P^T_\alpha(T^i\psi)n^\alpha,
\end{split}
\end{equation}
where the first inequality is due to the integrated decay estimate in Theorem \ref{Integrated decay estimate} and the second inequality follows from \eqref{conservation-degenerate-energy-sigma-t-t}.
\end{proof}

Theorem \ref{Integrated decay estimate} only gives the degenerate Morawetz estimate, since $\Delta$ vanishes at the horizon. In the proof of Theorem \ref{Integrated decay estimate}, we see that combining the energy inequality \eqref{energy-id-N} with the the uniform energy boundedness (Theorem \ref{uniform-bound}) and the local integrated decay estimate (Theorem \ref{Integrated decay estimate}), we get the non-degenerate local integrated decay estimate \cite{S-13} (corollary 4.3), \cite{A-Blue-wavekerr} (Appendix A).
\begin{corollary}[Nondegenerate Integrated Decay Estimate]\label{nondeg-Integrated decay estimate}
Let $\solu$ be a solution of Regge-Wheeler equation \eqref{RW} or Zerilli equations \eqref{Zerilli}, we have for all $R>3M$
\begin{equation}\label{eq-nondeg-integrate-decay-estimate}
\int_{\tau_0}^{\tau} \di t \int_{\Sigma^\prime_t} P^N_{\alpha}(\solu)n_{\Sigma_\tau}^{\alpha}  \lesssim  \int_{\Sigma_{\tau_0}} P^N_{\alpha}(\solu)n_{\Sigma_{\tau_0}}^{\alpha} + P^T_{\alpha}(T\solu)n_{\Sigma_{\tau_0}}^{\alpha},
\end{equation}
where $\Sigma_\tau^\prime=\Sigma_\tau \cap \{r<R\}$.
\end{corollary}

To proceed to higher order case,
we use the non-degenerate radial vector field 
 \begin{equation}\label{hat-Y}
\hat{Y}=\begin{cases}
(1-\mu)^{-1}\partial_u, &\text{if } r \leq r_{NH}, \\
\dr, &\text{if } r > r_{NH}.
\end{cases}
\end{equation}
Notice that, near horizon we can also write $\hat Y$ in $(v, r, \omega)$ coordinate as $\hat Y = \dr,\,\, \text{if } r \leq r_{NH}.$

\begin{corollary}[Nondegenerate High Order Integrated Decay Estimate]\label{coro-nondeg-Integrated decay estimate-high-order}
Let $\solu$ be a solution of Regge-Wheeler equation \eqref{RW} or Zerilli equations \eqref{Zerilli}, we have for all $R >3M$ and all integers $n \in \Naturals$,  
\begin{equation}\label{nondeg-Integrated decay estimate-high-order}
\begin{split}
&\int_{\tau_0}^{\tau} \di t \int_{\Sigma_\tau^\prime} \sum_{k+l+j\leq n} |N^{k}\Omega^l \hat{Y}^{j}\solu|^2\mathrm{d}\mu_g  \\
\lesssim &\int_{\Sigma_{\tau_0}} \sum_{k+l+j\leq n} P^N_{\alpha}(T^k\Omega^l \hat{Y}^j\solu)n_{\Sigma_{\tau_0}}^{\alpha} +  \sum\limits_{i\leq n+1} P^T_\alpha (T^i\solu)n_{\Sigma_{\tau_0}}^{\alpha},
\end{split}
\end{equation}
where $\Sigma_\tau^\prime=\Sigma_\tau\cap \{r<R\}$.
\end{corollary}

\begin{proof}
First, commuting the equation with T, we still have non-degenerate integrated decay estimate for $T\solu$
\begin{equation}\label{eq-nondeg-integrate-decay-estimate-T}
 \int_{\tau_0}^{\tau} \di t \int_{\Sigma_\tau^\prime} \sum_{k=0}^{1} P^N_{\alpha}(T^k\solu)n_{\Sigma_\tau}^{\alpha}  \lesssim  \sum_{k=0}^{1} \int_{\Sigma_{\tau_0}}P^N_{\alpha}(T^k\solu)n_{\Sigma_{\tau_0}}^{\alpha} +   P^T_\alpha (T^{k+1}\solu)n_{\Sigma_{\tau_0}}^{\alpha}.
\end{equation}
The elliptic estimate yields the high order integrated decay estimate away from the horizon. 

Near horizon, say $r<r_0<r_{NH}$, we shall commute the wave operator with $Y= (1-\mu)^{-1}\partial_u$ \cite{D-R-13}. This commutator has a good sign,
\begin{equation*}
\Box_gY\varphi -Y\Box_g\varphi= \kappa Y^2\varphi + f(YT\varphi, T\varphi,Y\varphi),
\end{equation*}
where $f(YT\varphi, T\varphi,Y\varphi)$ is linearly dependent on $YT\varphi, T\varphi,Y\varphi,$  and  $\kappa >0$  is another manifestation of the red-shift effect. After commuting the equation with $Y$ and $T$, we use the energy identity for $N$ to estimate
\begin{equation}\label{i-order-ile-Y}
\begin{split}
&\int_{\Sigma_\tau} P_\alpha^N(Y\solu)n^\alpha + \int_{\hori}P_\alpha^N(Y\solu)n^\alpha + \doubleint_{ \{ r<r_0 \} }K^N(Y\solu) \\
=  &  \int_{\Sigma_{\tau_0}} P_\alpha^N(Y\solu)n^\alpha + \doubleint_{\{r<r_0  \}} \mathcal{E}^N(Y\solu) +\doubleint_{ \{r_0\leq r\leq r_{NH} \} } \mathcal{E}^N(Y\solu)- K^N(Y\solu),
\end{split}
\end{equation}
where  
\begin{equation*}
\begin{split}
 \mathcal{E}^N(Y\solu) &= -NY\solu \left(\kappa Y^2\solu + f(YT\solu , T\solu, Y\solu) -\dr V_g \solu \right) \\
=  &  -\kappa (Y^2\solu)^2-\kappa (N-Y)Y\solu Y^2\solu +\left( f(YT\solu , T\solu, Y\solu)-\dr V_g \solu \right) NY\solu,
\end{split}
\end{equation*}
where $V_g$ takes the valued of Regge-Wheeler $V_{RW}$ or Zerilli potential $V_Z$.
The first term on the right hand side has a good sign. Applying Cauchy-Schwarz inequality and using the fact that $N-Y=T$ on $\hori$, the other terms can be bounded in $r\leq r_0$ by $c K^N(Y\solu) + c^{-1}(YT\solu)^2 +c^{-1}K^N(\solu).$
In view of \eqref{eq-nondeg-integrate-decay-estimate} and \eqref{eq-nondeg-integrate-decay-estimate-T},
\begin{equation*}
\doubleint_{\{r<r_0\}} \mathcal{E}^N(Y\solu) \leq  c\doubleint_{ \{ r<r_0 \} }K^N(Y\solu) +I.
\end{equation*}
Here $I$ depends only the initial data. The terms 
$$\doubleint_{ \{r_0 \leq r \leq r_{NH} \} } \mathcal{E}^N(Y\solu)- K^N(Y\solu)$$ can be estimated by the integrated decay estimate \eqref{eq-nondeg-integrate-decay-estimate-T}.
Noting that 
$P^N_{\alpha}(\solu)n_{\Sigma_\tau}^{\alpha} \leq b K^N(\solu)$ for $r<r_0<r_{NH}$ and some constant $b>0$,  \eqref{i-order-ile-Y} gives the estimate
\begin{equation*}
 \doubleint_{ \{ r<r_0 \} } P^N_{\alpha}(N\solu)n_{\Sigma_\tau}^{\alpha}
\lesssim  \int_{\Sigma_{\tau_0}} \sum\limits_{k, j\leq 1} P^N_{\alpha}(T^kY^j\solu)n^{\alpha} +  \sum\limits_{k\leq 2} P^T_{\alpha}(T^k\solu)n^{\alpha}.
\end{equation*}

Finally, commuting repeatedly with $Y$ and $T$, the above scheme plus elliptic estimate yield the desired estimate.
\end{proof}

The scheme in the proof of Corollary \ref{coro-nondeg-Integrated decay estimate-high-order} plus the elliptic estimates yield high order uniform boundedness
\begin{corollary}[High Order Uniform Boundedness]\label{uniform-bound-high}
Let $\solu$ be a solution of the Regge-Wheeler \eqref{RW} or Zerilli equations \eqref{Zerilli}.
Then for all $n\in\Naturals$, $\tau>\tau_0$,
\begin{equation}\label{uniform-energy-high}
\begin{split}
 &\int_{\Sigma_{\tau}}  \sum_{k+l+j\leq n} P^N_{\alpha}(N^k\Omega^l \hat{Y}^j\solu)n_{\Sigma_{\tau}}^{\alpha}\\
 \lesssim  &\int_{\Sigma_{\tau_0}}  \sum_{k+l+j\leq n} P^N_\alpha (T^i\Omega^l \hat{Y}^j\solu)n_{\Sigma_{\tau_0}}^{\alpha}+  \sum\limits_{i\leq n+1} P^T_\alpha (T^i\solu)n_{\Sigma_{\tau_0}}^{\alpha}.
 \end{split}
\end{equation}

\end{corollary}

\section{Decay estimate}\label{decay estimate-RW-Zerilli}

In this Section we prove quadratic decay of the non-degenerate energy.
First of all, we improve the local integrated decay estimate in Theorem \ref{Integrated decay estimate}.  We shall use the $r^p$ hierarchy estimate for proving energy decay estimate  \cite{S-13}.

\subsection{Energy decay}\label{rp}

Let $\solu$ be a solution of Regge-Wheeler equation \eqref{RW} or Zerilli equations \eqref{Zerilli} and define
\begin{equation}\label{def-Psi-Solu}
\Solu\doteq r\solu,
\end{equation}
 then
\begin{equation}\label{RW-Solu}
\L_{RW}\Solu\doteq\pu\pv\Solu -\eta \laplacianslash\Solu+V^{RW}\Solu=0,
\end{equation}
in Regge-Wheeler case, and 
\begin{equation}\label{Zerilli-Solu}
\L_{Z}\Solu\doteq\pu\pv\Solu -\eta \laplacianslash\Solu+V^{Z}\Solu=0,
\end{equation}
in Zerilli case.
Here 
\begin{align}
V^{RW} &=\eta\hat V^{RW}, \,\, \hat V^{RW}=-\frac{6M}{r^3},   \label{RW-V-hat}\\ 
V^{Z}&=\eta\hat V^{Z}, \,\,\,\,\,\,\,\,\,\,\,\, \hat V^{Z}=-\frac{6M}{r^3} -\frac{8M\epsilon}{r^3} \label{Zerilli-V-hat}.
\end{align}
We recall that  $V_g^{Z}=V_g^{RW}(1+\zeta),$ where $\zeta$ is defined in \eqref{epsilon} and $\eta=1-\mu.$

When there is no confusion, we also refer \eqref{RW-Solu} as Regge-Wheeler equation and \eqref{Zerilli-Solu}  as Zerilli equations.
\begin{figure}
\centering
\includegraphics[width=3.6in]{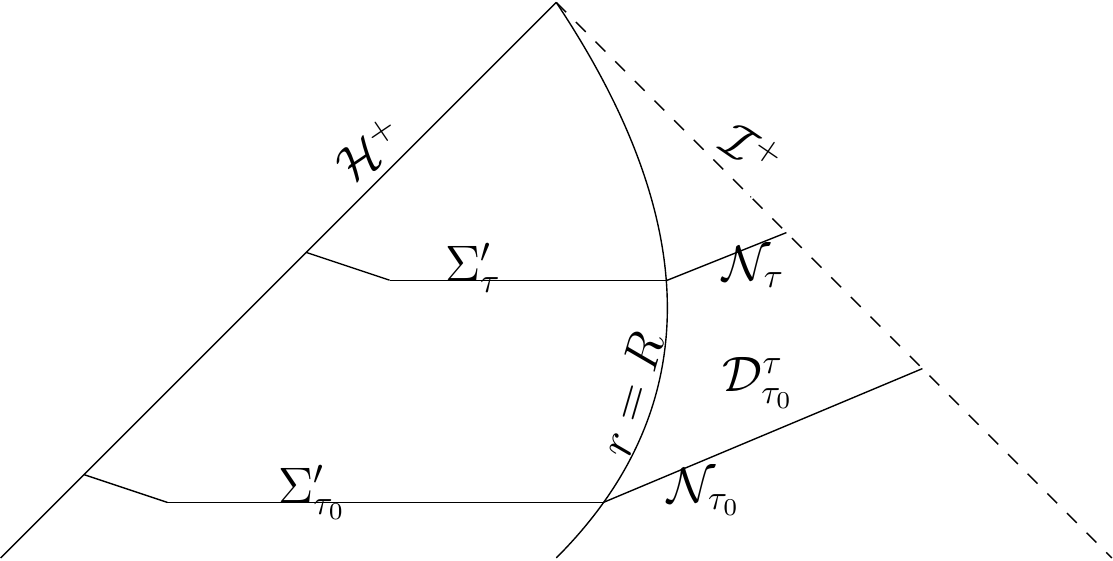}
\caption{The spacetime foliation $\bigcup_{\tau}\Sigma^\prime_{\tau}\cup\N_{\tau}$}
\label{fig:rpfoliation-u}
\end{figure}
We recall the definition of spacetime foliation $\bigcup_{\tau} \Sigma_\tau$ (Section \ref{Energy Estimate}), where $\Sigma_\tau=\Sigma^i_\tau \cup \Sigma^e_\tau$ with 
\begin{equation*}
\Sigma^i_\tau = \{\M |v=\tau +r_{NH}^\ast \} \cap  \{  r < r_{NH} \},
\end{equation*}
and
\begin{equation*}
\Sigma^e_\tau= \{\M |t=\tau \} \cap  \{  r\geq r_{NH} \}.
\end{equation*}
For our current purposes, we foliate the spacetime region by $\bigcup_{\tau}\Sigma^\prime_{\tau}\cup\N_{\tau}$ (Figure \ref{fig:rpfoliation-u}):
Fix $R>3M$ large enough, define the interior region $\cup_{\tau}\Sigma^\prime_{\tau}$, where
\begin{equation}\label{interior-foliation}
\Sigma^\prime_\tau =\Sigma_\tau \cap  \{  r \leq R \}.
\end{equation}
In the exterior region $\{r\geq R\}$, let $\mathcal{N}_\tau$ be the outgoing null hypersurface emerging from the sphere $S^2(\tau,R)$ with constant $t=\tau$ and constant $r=R$, 
\begin{equation}\label{null-hyper}
\mathcal{N}_\tau\doteq \{ u=\tau-R^\ast \} \cap \{ v\geq \tau+R^\ast \}.
\end{equation}
Let us also define a region bounded 
by the two null hypersurfaces and the time-like hypersurface (Figure \ref{fig:rpfoliation-u}):
\begin{equation}\label{outer-region}
\D_{\tau_1}^{\tau_2}=\{   (u,v) \in \M \big|  r(u,v)\geq R \quad  \text{and} \quad \tau_2-R^\ast \geq u \geq \tau_1-R^\ast \}.
\end{equation}

\begin{lemma}[Zero Order $r^p$ Integrated Decay Estimate]\label{lema-rp-hierarchy-0}
Let $\Solu$ be a solution of Regge-Wheeler \eqref{RW-Solu} or Zerilli equations \eqref{Zerilli-Solu}. Consider  the region $\D_{\tau_1}^{\tau_2}$, for $0< p\leq 2$, there is the integrated decay estimate
\begin{equation}\label{ineq-rp-0}
\begin{split}
\int_{\N_{\tau_2}}  r^p |\pv\Solu|^2\di v\di \sigma_{S^2} + \int_{\scri}  r^p \left( |\nablaslash\Solu|^2+ \frac{\Solu^2}{r^2} \right)\di u\di \sigma_{S^2}  {}&{}\\
+\doubleint_{\D_{\tau_1}^{\tau_2}}  r^{p-1} \{ \frac{p}{2}  |\pv\Solu|^2 +  \frac{2-p}{2} \left( |\nablaslash\Solu|^2 +  \frac{|\Solu|^2}{r^2} \right)+ \frac{6pM}{r} \frac{|\Solu|^2}{r^2} \} \di u\di v\di \sigma_{S^2}   {}&{}\\
\lesssim  \int_{\N_{\tau_1}}  r^p |\pv\Solu|^2\di v\di \sigma_{S^2} 
+\int_{\{r=R\}}  \{ |\pv\Solu|^2+ |\nablaslash\Solu|^2 +  |\Solu|^2  \}\di t \di \sigma_{S^2}.{}&{}
\end{split}
\end{equation}
\end{lemma}
\begin{remark}\label{rk-0order-rp}
Comparing with the zero order $r^p$ weighted inequality in \cite{S-13}, we get additional bound for the spacetime integral $\doubleint_{\D_{\tau_1}^{\tau_2}} r^{p-1} \frac{6pM}{r} \frac{|\Solu|^2}{r^2} \di u\di v\di \sigma_{S^2}.$ This would be important in proving first order $r^p$ weighted inequality (for $p=2$) in Lemma \ref{lema-rp-hierarchy-1}.
\end{remark}
\begin{proof}
Multiplying $\eta^{-k} r^{p} \p_v \Psi$ with $0< p\leq 2, k\geq 2$ on the Regge-Wheeler or Zerilli equations \eqref{RW-Solu}, \eqref{Zerilli-Solu}, we choose $R$ sufficiently large and integrate with respect to the measure $\di u \di v \di \sigma_{S^2}$ in $\D_{\tau_1}^{\tau_2}$, to derive the identity
\begin{equation}\label{id-r-p}
\begin{split}
 {}&{} \int_{\N_{\tau_2}}  \left( \frac{r^p}{\eta^k} |\pv\Solu|^2 \right) + \int_{\scri}  \left( \frac{r^p}{\eta^{k-1}}|\nablaslash\Solu|^2  +   \frac{r^pV}{\eta^k}|\Solu|^2  \right) \\
+{}&{} \doubleint_{\D_{\tau_1}^{\tau_2}}   \left( (2-p)r^{p-1}\eta^{-k+2} + 2M(k-1) r^{p-2} \eta^{-k+1} \right) |\nablaslash \Solu |^2  \\
- {}&{} \doubleint_{\D_{\tau_1}^{\tau_2}} \pu \left( \frac{r^p}{\eta^k} \right)|\pv\Solu|^2 +  \pv \left( \frac{r^pV}{\eta^k} \right)|\Solu|^2 \\
= {}&{}  \int_{\N_{\tau_1}}  \left( \frac{r^p}{\eta^k} |\pv\Solu|^2 \right) + \int_{\{r=R\}}\left( \frac{r^p}{\eta^k} |\pv\Solu|^2  + \frac{r^p}{\eta^{k-1}}|\nablaslash\Solu|^2 + \frac{r^pV}{\eta^k}|\Solu|^2    \right), 
\end{split}
\end{equation}
where $V$ could be $V^{RW}$ or $V^Z$ in those two different cases.
For $R$ sufficiently large,
$$ -\pu( \frac{r^p}{\eta^k} )\geq \frac{p}{2}r^{p-1} \text{ for all } 0< p\leq 2 \quad \text{and} \quad k\leq 5. $$
Next, we will prove the positivity of the other bulk terms in \eqref{id-r-p} for Regge-Wheeler and Zerilli cases separately.

{\bf Regge-Wheeler case $V=V^{RW}$:} In the third line of \eqref{id-r-p}, the term involving $|\Solu|^2$ is 
\begin{equation}\label{neg-rw}
\begin{split}
 {}&{} - \doubleint_{\D_{\tau_1}^{\tau_2}} \pv \left( \frac{r^pV^{RW}}{\eta^k} \right) |\Solu|^2= \doubleint_{\D_{\tau_1}^{\tau_2}} \pv\left(\frac{6Mr^{p-3}}{\eta^{k-1}}\right) |\Solu|^2 \\
{}&=\doubleint_{\D_{\tau_1}^{\tau_2}} \eta^{-k+1} \left( 6M(p-3) r^{p-4} +12M^2(4-k-p) r^{p-5} \right) |\Solu|^2.
\end{split}
\end{equation}
\eqref{neg-rw} does not have the good sign.  But there is additional positive terms $ \doubleint_{\D_{\tau_1}^{\tau_2}} 2M(k-1) r^{p-2} \eta^{-k+1}  |\nablaslash \Solu |^2$  in the second line of \eqref{id-r-p}. Using the fact that $\Solu$ has angular frequency $\ell \geq 2$ as stated in remark \ref{rk:mode},  we have
\begin{equation}\label{positive-add-rw}
\int_{S^2(t, r)} 2M(k-1) r^{p-2} \eta^{-k+1} |\nablaslash \Solu |^2 \geq \int_{S^2(t, r)} 12M(k-1) r^{p-4} \eta^{-k+3}|\Solu|^2.
\end{equation}
This additional positive term could be used to absorb the negative terms in \eqref{neg-rw}. 
Therefore, the bulk terms in \eqref{id-r-p} dominate
\begin{equation}\label{positive-bulkterm-1-rw}
\begin{split}
 \doubleint_{\D_{\tau_1}^{\tau_2}}\left( 6M(p-3+2c_1k-2c_1) r^{p-4} +12M^2(4-k-p) r^{p-5} \right)  |\Solu|^2 {}&{} \\
 +\doubleint_{\D_{\tau_1}^{\tau_2}} r^{p-1} |\pv\Solu|^2 +(2-p) r^{p-1}|\nablaslash\Solu|^2 +2M(1-c_1)(k-1) r^{p-2}  |\nablaslash \Solu |^2 {}&{},
\end{split}
\end{equation}
for some constant $c_1\leq 1$. We take $c_1=\frac{1}{2}, k=4$. Then for $R$ sufficiently large, the bulk terms \eqref{positive-bulkterm-1-rw} bound 
$$ \doubleint_{\D_{\tau_1}^{\tau_2}} \frac{p}{2}  r^{p-1} |\pv\Solu|^2 +  \frac{2-p}{2}r^{p-1} \left(|\nablaslash\Solu|^2 +  \frac{|\Solu|^2}{r^2}\right)  +\frac{6pM}{r}r^{p-1} \frac{|\Solu|^2}{r^2} ,$$ which is positive.
Here we have used the fact that $\ell \geq 2$.
Therefore, integrating the identity \eqref{id-r-p} for $0< p\leq 2$ and $k=4$ with respect to the measure $\di u\di v\di \sigma_{S^2}$ in the region $\D_{\tau_1}^{\tau_2}$, we have 

\begin{align}
\int_{\N_{\tau_2}} \frac{r^p}{\eta^4} |\pv\Solu|^2 \di v\di \sigma_{S^2} +  \int_{\scri}  \frac{r^p}{\eta^3} \{|\nablaslash\Solu|^2 - \frac{6M}{r} \frac{|\Solu|^2}{r^2} \}  \di u\di \sigma_{S^2}{}&{} \label{Psi-rpv-energy-rw}\\
+\doubleint_{\D_{\tau_1}^{\tau_2}}  r^{p-1}  \{ \frac{p}{2} |\pv\Solu|^2 +  \frac{2-p}{2} (|\nablaslash\Solu|^2 +  \frac{|\Solu|^2}{r^2}) + \frac{6pM}{r}\frac{|\Solu|^2}{r^2} \} \di u\di v\di \sigma_{S^2}  {}&{} \label{Psi-bulk-rp-rw}\\
\lesssim   \int_{\N_{\tau_1}} \frac{r^p}{\eta^4} |\pv\Solu|^2 \di v\di \sigma_{S^2} 
+\int_{r=R}  \{ |\pv\Solu|^2+ |\nablaslash\Solu|^2 +  |\Solu|^2  \} \di t \di \sigma_{S^2}. {}&{} \label{Psi-bdy-rp-rw}
\end{align}
Due to the fact that $\ell\geq 2$, the integral on the scri $\scri$ in \eqref{Psi-rpv-energy-rw} is actually positive. And $\eta\sim 1$ for $R$ large enough, thus we achieve the integrated decay estimate \eqref{ineq-rp-0} for Regge-Wheeler case.

{\bf Zerilli case $V=V^Z$:}  In \eqref{id-r-p}, the term involving $|\Solu|^2$ is multiplied by
\begin{equation}\label{neg-z}
\begin{split}
 -  \pv \left(\frac{r^pV^Z}{\eta^k} \right) =  \pv\left(\frac{6Mr^{p-3}}{\eta^{k-1}}\right) + \pv\left(\frac{8Mr^{p-3} \zeta}{\eta^{k-1}}\right)  {}&{} \\
>\eta^{-k+1} \left(12M(p-3) r^{p-4} +24M^2(4-k-p) r^{p-5} \right) {}&{},
\end{split}
\end{equation}
where we used the fact that (see \eqref{epsilon}. \eqref{dr-epsilon-1}),
\begin{equation}\label{epsilon-estimate}
-\frac{3}{2(2\bar\lambda+3)}\leq \zeta \leq\frac{3}{2 \bar\lambda},\quad | \zeta | \leq\frac{3}{2 \bar\lambda}, \quad \dr \zeta >0,
\end{equation}
and $0<p\leq 2, \, k \leq 4.$
As above, those terms in \eqref{neg-z} do not have the good sign.  But the additional positive terms $2M(k-1) r^{p-2} \eta^{-k+1}  |\nablaslash \Solu |^2$ in \eqref{id-r-p} could be used to absorb the negative terms in \eqref{neg-z}. Since $\ell \geq 2$, the bulk terms in \eqref{id-r-p} dominate
\begin{equation}\label{positive-bulkterm-1-z}
\begin{split}
\doubleint_{\D_{\tau_1}^{\tau_2}} \left(12M(p-3+c_2k-c_2) r^{p-4} +24M^2(4-k-p) r^{p-5} \right)  |\Solu|^2 {}&{} \\
 +\doubleint_{\D_{\tau_1}^{\tau_2}} r^{p-1} |\pv\Solu|^2 +(2-p) r^{p-1}|\nablaslash\Solu|^2 +2M(1-c_2)(k-1) r^{p-2}  |\nablaslash \Solu |^2 {}&{},
\end{split}
\end{equation}
for some constant $c_2\leq 1$. We take $c_2=1, k=4$, then for $R$ sufficiently large, the bulk terms \eqref{positive-bulkterm-1-z} bound 
$$\doubleint_{\D_{\tau_1}^{\tau_2}} \frac{p}{2}  r^{p-1} |\pv\Solu|^2 +  \frac{2-p}{2}r^{p-1} \left(|\nablaslash\Solu|^2 +  \frac{|\Solu|^2}{r^2} \right)  +  \frac{6pM}{r}r^{p-1} \frac{|\Solu|^2}{r^2} ,$$ which is positive.
Integrating the identity \eqref{id-r-p} for $0< p\leq 2$ and $k=4$, we have

\begin{align*}
\int_{\N_{\tau_2}} \frac{r^p}{\eta^4} |\pv\Solu|^2 \di v\di \sigma_{S^2} +  \int_{\scri} \frac{r^p}{\eta^3} \{|\nablaslash\Solu|^2 + \hat V^{Z} |\Solu|^2\}  \di u\di \sigma_{S^2} {}&{} \\
+\doubleint_{\D_{\tau_1}^{\tau_2}} r^{p-1}  \{ \frac{p}{2} |\pv\Solu|^2 +  \frac{2-p}{2} \left( |\nablaslash\Solu|^2 +  \frac{|\Solu|^2}{r^2} \right) + \frac{6pM}{r}\frac{|\Solu|^2}{r^2} \}  \di u\di v\di \sigma_{S^2}  {}&{} \\
\lesssim   \int_{\N_{\tau_1}} \frac{r^p}{\eta^4} |\pv\Solu|^2 \di v\di \sigma_{S^2} 
+\int_{\{r=R \}}  \{ |\pv\Solu|^2+ |\nablaslash\Solu|^2 +  |\Solu|^2  \} \di t \di \sigma_{S^2}. {}&{} 
\end{align*}
In the same way, the integral on scri $\scri: \, \frac{r^p}{\eta^3} \{|\nablaslash\Solu|^2+ \hat V^{Z} |\Solu|^2 \} $ is positive since $\ell\geq 2$. Thus we achieve the integrated decay estimate \eqref{ineq-rp-0} for Zerilli case.

\end{proof}

The $r^p$ integrated decay estimate and the non-degenerate integrated decay estimate lead to the energy decay as follows \cite{S-13}.

\begin{theorem}[Energy Decay]\label{energy decay}
Let $R>3M,$ and let $\solu$ be a solution of Regge-Wheeler \eqref{RW} or Zerilli equations \eqref{Zerilli}, with the initial data on $\Sigma^\prime_{\tau_0} \cup \N_{\tau_0}$ satisfying
\begin{equation}\label{initial data-energy-decay}
\int_{\N_{\tau_0}}\sum_{k\leq1}  |T^k\pv(r\solu)|^2r^2  \di v \di \sigma_{S^2} + \int_{\Sigma^\prime_{\tau_0} \cup \N_{\tau_0}}  \sum_{k\leq2} P^N_\mu(T^k\solu) n^\mu <\infty,
\end{equation}
then there exist a constant $I$ depending on the initial data \eqref{initial data-energy-decay}, such that
\begin{equation}\label{energy decay-1}
\int_{\Sigma_\tau^\prime\cup\N_\tau} P^N_\mu(\solu) n^\mu \lesssim  \frac{I}{\tau^2}.
\end{equation}
Here
\begin{equation*}
 \Sigma^\prime_\tau= \Sigma_\tau \cap  \{  r\leq R \}, \quad \N_{\tau}=\{\M| u=\tau-R^\ast, v \geq v_0=\tau_0+R^\ast \}.
\end{equation*}
\end{theorem}

\begin{proof}
We will only give the sketch of proof here (refer to \cite{S-13} for more details). 
From the non-degenerate integrated decay estimate (Corollary \ref{nondeg-Integrated decay estimate}) and conservation of $\partial_t$ energy (imposing zero data on null infinity to evolve backwards where necessary)
\begin{equation}\label{non-deg-integrated decay estimate-v}
\begin{split}
\int_{\tau_0}^\tau dt\int_{\Sigma'_t} P^N_\mu(\psi)n^\mu
{}&\lesssim \int_{\Sigma'_\tau\cup\mathcal{N}_\tau} P^N_\mu(\psi)n^\mu 
+\int_{\Sigma_\tau} P^T_\mu(T\psi)n^\mu \\
{}&\lesssim \left(\int_{\Sigma'_\tau\cup\mathcal{N}_\tau} P^N_\mu(\psi)n^\mu +P^T_\mu(T\psi)n^\mu \right).
\end{split}
\end{equation}
Using the spacetime foliation $\bigcup_{\tau}\Sigma_\tau^\prime\cup\N_\tau$, we also have the uniform boundness (Theorem \ref{uniform-bound}), for any $\tau_2 > \tau_1$,
\begin{equation}\label{uniform boundness-v}
\int_{\Sigma_{\tau_2}^\prime\cup\N_{\tau_2}} P^N_\mu(\solu) n^\mu \lesssim  \int_{\Sigma_{\tau_1}^\prime\cup\N_{\tau_1}} P^N_\mu(\solu) n^\mu + P^T_\mu(T\solu) n^\mu.
\end{equation}

We begin with an inequality, for any $\tau_2 > \tau_1$
\begin{equation}\label{non-deg-integrated decay estimate-1}
 \begin{split}
&\int_{\tau_1}^{\tau_{2}} \di \tau \int_{\Sigma_{\tau}^\prime\cup\N_{\tau}} P^N_\mu(\solu) n^\mu \lesssim \doubleint_{\D_{\tau_1}^{\tau_{2}}} \left(   |\pv\Solu|^2 +  |\nablaslash\Solu|^2 +  \frac{|\Solu|^2}{r^2} \right) \di u\di v\di \sigma_{S^2} \\
&\quad \quad \quad \quad \quad \quad \quad \quad \quad \quad + \int_{\Sigma_{\tau_1}^\prime\cup\N_{\tau_1}} P^N_\mu(\solu) n^\mu + P^T_\mu(T\solu) n^\mu,
 \end{split}
\end{equation}
where we had used the non-degenerate integrated decay estimate \eqref{non-deg-integrated decay estimate-v}.
Taking $p=1$ in the $r^p$ weighted inequality of Lemma \ref{lema-rp-hierarchy-0}, we can further estimate \eqref{non-deg-integrated decay estimate-1} by
\begin{equation}\label{non-deg-integrated decay estimate-p1}
\begin{split}
&\int_{\tau_1}^{\tau_{2}} \di \tau \int_{\Sigma_{\tau}^\prime\cup\N_{\tau}} P^N_\mu(\solu) n^\mu \lesssim \int_{\N_{\tau_1}}   r |\pv\Solu|^2  \di v\di \sigma_{S^2}\\
&\quad \quad \quad \quad \quad \quad \quad \quad \quad \quad +  \int_{\Sigma_{\tau_1}^\prime\cup\N_{\tau_1}} P^N_\mu(\solu) n^\mu + P^T_\mu(T\solu) n^\mu.
\end{split}
\end{equation}
Next, we take $p=2$ in the $r^p$ weighted inequality of Lemma \ref{lema-rp-hierarchy-0}, then there exists a dyadic sequence $\{\tau'_j\}_{j\in \Naturals}$ with $\tau'_{j+1}=2\tau'_j$ and $\tau'_0 = \tau_0$
 \begin{equation}\label{non-deg-integrated decay estimate-p2}
 \begin{split}
&\int_{\N_{\tau'_j}}   r |\pv\Solu|^2  \di v\di \sigma_{S^2}\\
\lesssim &\frac{1}{\tau'_j} \left(\int_{\N_{\tau_0}}  r^2 |\pv\Solu|^2 \di v\di \sigma_{S^2} + \int_{\Sigma_{\tau_0}^\prime\cup\N_{\tau_0}} P^T_\mu(\solu) n^\mu + P^T_\mu(T\solu) n^\mu \right).
 \end{split}
\end{equation}
Note that, in proving \eqref{non-deg-integrated decay estimate-p2}, there is the boundary term on $\{r=R\}$ as the right hand side of  \eqref{ineq-rp-0},
we can use the mean value theorem for the integration in $r^\ast$, and the local integrated energy decay \eqref{non-deg-integrated decay estimate-v} to bound it by $\int_{\Sigma_{\tau_0}^\prime\cup\N_{\tau_0}} P^T_\mu(\solu) n^\mu + P^T_\mu(T\solu) n^\mu$.
Combining \eqref{non-deg-integrated decay estimate-p1} and \eqref{non-deg-integrated decay estimate-p2},
we have
 \begin{equation}\label{non-deg-integrated decay estimate-pigeon}
 \begin{split}
&\int_{\tau'_j}^{\tau'_{j+1}} \di \tau \int_{\Sigma_{\tau}^\prime\cup\N_{\tau}} P^N_\mu(\solu) n^\mu\\
\lesssim &\frac{1}{\tau'_j} \left(\int_{\N_{\tau_0}}  r^2 |\pv\Solu|^2 \di v\di \sigma_{S^2} + \int_{\Sigma_{\tau_0}^\prime\cup\N_{\tau_0}} P^T_\mu(\solu) n^\mu + P^T_\mu(T\solu) n^\mu \right)\\
&\quad \quad \quad \quad \quad \quad \quad \quad \quad \quad + \int_{\Sigma_{\tau'_j}^\prime\cup\N_{\tau'_j}} P^N_\mu(\solu) n^\mu + P^T_\mu(T\solu) n^\mu.
 \end{split}
\end{equation}
Besides, with the uniform boundness of energy \eqref{uniform boundness-v}, we may estimate the last term in \eqref{non-deg-integrated decay estimate-pigeon} by
\begin{equation}\label{pigeon-uniform-energy}
\begin{split}
& \int_{\Sigma_{\tau'_j}^\prime\cup\N_{\tau'_j}} \left( P^N_\mu(\solu) n^\mu + P^T_\mu(T\solu) n^\mu \right)\\
\lesssim& \frac{1}{\tau'_j } \left(  \int_{\tau'_{j-1}}^{\tau'_{j+1}} \di \tau \int_{\Sigma_{\tau}^\prime\cup\N_{\tau}} P^N_\mu(\solu) n^\mu + \sum_{i\leq 1} P^T_\mu(T^i\solu) n^\mu \right).
\end{split}
\end{equation}
Again we apply \eqref{non-deg-integrated decay estimate-pigeon} on the above term 
 \begin{equation*}
\int_{\tau'_{j-1}}^{\tau'_{j+1}} \di \tau \int_{\Sigma_{\tau}^\prime\cup\N_{\tau}} P^N_\mu(\solu) n^\mu + \sum_{i\leq 1} P^T_\mu(T^i\solu) n^\mu,
\end{equation*}
to derive
\begin{equation}\label{pigeon-uniform-energy-decay}
\begin{split}
& \int_{\Sigma_{\tau'_j}^\prime\cup\N_{\tau'_j}} \left( P^N_\mu(\solu) n^\mu + P^T_\mu(T\solu) n^\mu \right)\\
\lesssim & \frac{1}{\tau'_j  \tau'_{j-1}} \left( \int_{\N_{\tau_0}} \sum_{k \leq 1}  r^2 |\pv T^k \Solu|^2 \di v\di \sigma_{S^2} + \int_{\Sigma_{\tau_0}^\prime\cup\N_{\tau_0}}  \sum_{i \leq 2} P^T_\mu(T^i\solu) n^\mu \right)\\
&\quad  \quad \quad \quad \quad \quad \quad \quad \quad + \frac{1}{\tau'_j } \left(  \int_{\Sigma_{\tau'_{j-1}}^\prime\cup\N_{\tau'_{j-1}}} P^N_\mu(\solu) n^\mu + \sum_{i\leq 2}  P^T_\mu(T^i\solu) n^\mu\right),
\end{split}
\end{equation}
where by the uniform boundness of energy \eqref{uniform boundness-v}, the last term could be further bounded by
 \begin{equation*}
 \frac{1}{\tau'_j } \left(  \int_{\Sigma_{\tau_0}^\prime\cup\N_{\tau_0}} P^N_\mu(\solu) n^\mu + \sum_{i\leq 2} P^T_\mu(T^i\solu) n^\mu\right).
\end{equation*}

Thus, in view of \eqref{non-deg-integrated decay estimate-pigeon} and \eqref{pigeon-uniform-energy-decay}, we have
 \begin{equation}\label{non-deg-integrated decay estimate-pigeon-1}
 \begin{split}
&\int_{\tau'_j}^{\tau'_{j+2}} \di \tau \int_{\Sigma_{\tau}^\prime\cup\N_{\tau}} P^N_\mu(\solu) n^\mu \\
 \lesssim & \frac{1}{\tau'_j}  \left( \int_{\N_{\tau_0}} \sum_{k \leq 1}  |\pv T^k \Solu|^2r^2 \di u\di v\di \sigma_{S^2} 
 + \int_{\Sigma_{\tau_0}^\prime\cup\N_{\tau_0}} \sum_{i\leq 2} P^N_\mu(T^i\solu) n^\mu \right).
 \end{split}
\end{equation}
Finally performing the pigeon-hole principle, we have \eqref{energy decay-1}.
\end{proof}

We further commute the equations with $\Omega$ repeatedly, to have the high order energy decay.
As a result, we have the pointwise decay estimate, 

\begin{theorem}[Pointwise Decay]\label{them-uni-pointwise decay}
Let $\solu$ be a solution of Regge-Wheeler \eqref{RW} or Zerilli equations \eqref{Zerilli}, with the initial data on $\Sigma^\prime_{\tau_0} \cup \N_{\tau_0}$ satisfying
\begin{equation}\label{initial data-pointwise}
\sum_{l\leq2} \left(\int_{\N_{\tau_0}}  \sum_{k\leq1} |T^k \Omega^l\pv(r\solu)|^2r^2  \di v\di \sigma_{S^2}
+ \int_{\Sigma^\prime_{\tau_0} \cup \N_{\tau_0}} \sum_{k\leq2}  P^N_\mu(T^{k}\Omega^l\solu) n^\mu \right)<\infty,
\end{equation}
then there exist a constant $I$ depending on the initial data \eqref{initial data-pointwise}, such that in the future development of initial hypersurface $J^{+}(\Sigma^\prime_{\tau_0}\cup \N_{\tau_0})$
\begin{equation}\label{pointwise decay-uni}
r^\frac{1}{2} |\solu|(\tau, r) \lesssim  \frac{I}{\tau}.
\end{equation}
\end{theorem}

\begin{proof}
On $\N_{\tau}=\{\M|u=\tau-R^\ast, v\geq \tau+R^\ast\}$, integrating from infinity, by the Cauchy Schwarz inequality, we have for any $v\geq\tau+R^\ast$,
\begin{equation*}
(r^\frac{1}{2} \solu)^2 (u,v) = -2\int_{v}^{\infty}r^\frac{1}{2} \solu \pv(r^\frac{1}{2}\solu) \di v \lesssim \int_{v}^{\infty} |\solu|^2 \di v + \int_{v}^{\infty} |r\pv\solu|^2 \di v.
\end{equation*}
And then using the Sobolev inequality on the sphere, we have
\begin{equation}\label{Sobolev-v}
\begin{split}
(r^\frac{1}{2} \solu)^2  &\lesssim \int_{v}^{\infty} \int_{S^2}  \sum_{k\leq 2} \frac{|\Omega^k\solu|^2}{r^2} r^2\di v \di \sigma_{S^2} + \int_{v}^{\infty}  \int_{S^2} \sum_{k\leq 2}|\pv\Omega^k\solu|^2 r^2 \di v \di \sigma_{S^2}\\
&\lesssim \int_{\N_\tau} \sum_{k\leq2} P^N_\mu(\Omega^k\solu) n^\mu\lesssim  \frac{I}{\tau^2}.
\end{split}
\end{equation}

On $\Sigma_\tau\cap \{2M< r_0 \leq r \leq R \}$, the Sobolev inequality also yields the pointwise decay.

Near the horizon $\Sigma_{\tau}^i$, which is $\{\M| v=\tau+r^\ast_{NH}, u\leq u_0(\tau)\}$ with $u_0(\tau)<\tau-r_0^\ast$ to be chosen later, we proceed a similar argument. 
Integrating from $u=\tau-r^\ast_{NH}$, we have for any $u\geq u_0(\tau)$,
\begin{equation*}
\int_{S^2} \solu^2(u,v) \di \sigma_{S^2}= \int_{S^2}\solu^2(u_0(\tau), v) \di \sigma_{S^2} + 2\int^{u}_{u_0(\tau)}\int_{S^2}  \solu \pu\solu \di u \di \sigma_{S^2}.
\end{equation*}
The Cauchy Schwarz inequality gives
\begin{equation*}
 \int_{\Sigma_{\tau}^i} \solu \pu\solu \di u  \di \sigma_{S^2} \lesssim  \int_{\Sigma_{\tau}^i} (1-\mu)\solu^2 \di u  \di \sigma_{S^2}+\int_{\Sigma_{\tau}^i} \frac{(\pu\solu)^2}{1-\mu} \di u  \di \sigma_{S^2}.
\end{equation*}
Again we apply the Sobolev inequality on the sphere to get the pointwise decay estimate. Besides, applying a pigeon-hole argument in $r$ and replacing $\solu$ in \eqref{energy decay-1} by $\Omega^k\solu$, we obtain that $u_0(\tau)$ could be  chosen so that $$ \int_{S^2}\solu^2(u_0(\tau), v) \di \sigma_{S^2} \lesssim \tau^{-2}.$$
In view of the definition of the non-degenerate energy $E^N(\solu,\Sigma^i_\tau)$ in \eqref{nondeg-energy-horizon} (and with the restriction to $r\leq r_{NH}$, we have absorbed factors of $r$ into the constant $C$), we have
\begin{equation*}
\begin{split}
\solu^2(u,v)  &\leq C \int_{\Sigma^i_\tau}  \sum_{k\leq 2} \left( (1-\mu)|\Omega^k\solu|^2 + \frac{(\pu\Omega^k\solu)^2}{1-\mu}\right) r^2\di u \di \sigma_{S^2} \\
 &+ \int_{S^2}\solu^2(u_0(\tau), v) \di \sigma_{S^2} \lesssim  \frac{I}{\tau^2}.
\end{split}
\end{equation*}

\end{proof}

We next proceed to high order $r^p$ integrated decay estimate. For notational convenience, we denote $K_{p-1}(\Solu)$ the spacetime integral \eqref{Psi-bulk-rp-rw} in the $r^p$ weighted inequality，
\begin{equation}\label{def-K-p}
\begin{split}
K_{p-1}(\Solu) \doteq &\doubleint_{\D_{\tau_1}^{\tau_2} }  \left( \frac{p}{2}  r^{p-1} |\pv\Solu|^2 +  \frac{6pM}{r}r^{p-1} \frac{|\Solu|^2}{r^2} \right) \di u\di v\di \sigma_{S^2}\\
+&\doubleint_{\D_{\tau_1}^{\tau_2} }  \frac{2-p}{2}r^{p-1} \left( |\nablaslash\Solu|^2 +  \frac{|\Solu|^2}{r^2} \right) \di u\di v\di \sigma_{S^2},
\end{split}
\end{equation} 
and $S_p(\Solu)$ the energy on $\scri$,
\begin{equation}\label{def-S-p}
S_{p}(\Solu) \doteq  \int_{\scri}  r^p \left(|\nablaslash \Solu|^2 +\frac{\Solu^2}{r^2} \right)\di u\di \sigma_{S^2}.
\end{equation} 
In fact, when proceeding to the first order case, we will commute the equation with $r\pv$.
Define
\begin{equation}\label{def-So1}
\Solu^{(1)} \doteq r\pv \Solu.
\end{equation}
Unavoidably, there will be the leading error term involving angular derivative $\laplacianslash\Solu$ appearing during the commuting procedure. And we need to control these error terms. As we can see from the right hand side of $r^p$ weighted inequality, the spacetime integral involving $|\nablaslash\Solu|^2$ \eqref{Psi-bulk-rp-rw} vanishes when $p=2$, which implies that we will lost the control for the angular derivative in the error terms when $p=2$. To get around this difficulty, we perform the integration by parts twice on the leading term
\begin{equation}\label{leading-laplacian-1st-order-rp}
\doubleint_{\D_{\tau_1}^{\tau_2} }  -2r^{p}\laplacianslash\Solu \pv\So.
\end{equation} 
Then instead of \eqref{leading-laplacian-1st-order-rp}, we get to deal with 
\begin{equation}\label{leading-angular-1st-order-rp}
\doubleint_{\D_{\tau_1}^{\tau_2} }  (2-p)^2 r^{p-1} |\nablaslash\Solu|^2,
\end{equation} 
as we can see in \eqref{eq-A-2}, \eqref{A-2} and \eqref{esimate-A-2-main}. Due to the presence of the additional factor $(2-p)^2$, \eqref{leading-angular-1st-order-rp} which involves the angular derivative vanishes when $p=2$. Thus our estimate go through even for $p=2$. This idea could be found in \cite{rp-ASF-Mos}. This is the main difference from the first order $r^p$ weighted inequality in Proposition 5.6 of \cite{S-13}. We will prove the following Lemma.

\begin{lemma}[First Order $r^p$ Integrated Decay Estimate]\label{lema-rp-hierarchy-1}
Let $\Solu$ be a solution to Regge-Wheeler equation \eqref{RW-Solu} or Zerilli equations \eqref{Zerilli-Solu}, and 
define 
\begin{equation}\label{def-D-derivative}
\mathbb{D}=\{r\pv, (1-\mu)^{-1}\pu, r\nablaslash   \}.
\end{equation}
Considering in the region $\D_{\tau_1}^{\tau_2}$, there is the integrated decay estimate for $0< p \leq 2$,
\begin{equation}\label{rp-1-D}
\begin{split}
  \sum_{j\leq1}  \int_{\N_{\tau_2}}r^p|\pv \mathbb{D}^j \Solu|^2 \di v\di \sigma_{S^2} +  \int_{\scri}  r^p \left(|\pv \mathbb{D}^j \Solu|^2 +\frac{  \mathbb{D}^j\Solu^2}{r^2} \right)  \di u\di \sigma_{S^2} {}&{} \\
+\doubleint_{\D_{\tau_1}^{\tau_2}} \sum_{j\leq1} p r^{p-1} \left(  |\pv\mathbb{D}^j \Solu|^2  +\frac{M}{r}  \frac{|\mathbb{D}^j \Solu|^2}{r^2} \right)  \di u\di v\di \sigma_{S^2} {}&{}  \\
+\doubleint_{\D_{\tau_1}^{\tau_2}}  \sum_{j\leq1}  (2-p)r^{p-1} \left( |\nablaslash\mathbb{D}^j \Solu|^2 +  \frac{|\mathbb{D}^j \Solu|^2}{r^2}\right)\di u\di v\di \sigma_{S^2} {}&{}  \\
\lesssim    \int_{\N_{\tau_1}}    r^p \left( \sum_{j\leq1}  | \pv  \mathbb{D}^j\Solu |^2  +\sum_{l\leq 2} | \pv \Omega^l\Solu |^2    \right) \di v\di \sigma_{S^2} {}&{}\\
+\int_{\{r=R \}} \sum_{i,j\leq1, l \leq 2} |\pv^i\nablaslash^l\mathbb{D}^j \Solu|^2\di t \di \sigma_{S^2}.{}&{}  
\end{split}
\end{equation}
\end{lemma}
\begin{remark}
In the first order $r^p$ weighted energy inequality of \cite{S-13} (in the proof of Proposition 5.6), the $p$ has only range $(0,2)$. We improve the rang of $p$ to be $(0,2]$ in the above Lemma. Based on this, we can further improve the decay estimate for time derivative as $t^{-2}$, see subsection \ref{Improved interior decay}. This could be compared with the decay of time derivative $t^{-2+\delta}$ in \cite{S-13}.
\end{remark}

\begin{proof}
As our proof is the same for both Regge-Wheeler and Zerilli case, 
we take the Zerilli case for example. We begin with commuting the equation with $r\pv$.
For any smooth function $\varphi\in C^{\infty}(\M)$, we have the commuting identity,
\begin{equation}\label{commutator}
\begin{split}
[\L_Z, r\pv]\varphi &=  \eta\pu\pv\varphi -\eta\pv^2\varphi -2\eta \left(1-\frac{3M}{r} \right) \laplacianslash\varphi  \\
 &\quad - \eta\frac{2M}{r^2}\pv\varphi -r\pv V^Z\varphi,
\end{split}
\end{equation}
where $\L_Z$ is defined as in \eqref{Zerilli-Solu}. 
In view of  the Zerilli equations \eqref{Zerilli-Solu} and commuting identity\eqref{commutator}, we have
\begin{equation}\label{commutator-S1}
\begin{split}
\L_Z\So &=  -\eta\pv^2\Solu+\{\eta^2-2\eta(1-\frac{3M}{r})\}\laplacianslash\Solu  \\
&\quad - \eta\frac{2M}{r^2}\pv\Solu - (V^Z + r\pv V^Z)\Solu.
\end{split}
\end{equation}
It turns out that the first term on the right hand side of \eqref{commutator-S1} has a good sign. Namely, 
\begin{equation}\label{good sing term-Solu1}
-\eta\pv^2\Solu= - \frac{\eta}{r}\pv\So + \frac{\eta^2}{r}\pv\Solu,
\end{equation}
and $- \frac{\eta}{r}\pv\So$ has a good sign. Namely, we
multiply $2r^p(1-\mu)^{-k}\pv\So$ on both sides of  \eqref{commutator-S1}, and integrate on the spacetime region $\D_{\tau_1}^{\tau_2}$,  to yield that,
\begin{equation}\label{rp-ineq-first-order}
\begin{split}
&\int_{\N_{\tau_2}}   r^p|\pv\So|^2 \di v\di \sigma_{S^2}  +K_{p-1}(\So) + S_p(\So)\\
\lesssim & \int_{\N_{\tau_1}}   r^p|\pv\So|^2 \di v\di \sigma_{S^2} +\doubleint_{\D_{\tau_1}^{\tau_2}} -2 \eta^{-k+1} r^{p-1} |\pv\So|^2 \\
& \quad \quad + A_1 + A_2 +A_3 +\text{boundary term on} \{r=R\},
\end{split}
\end{equation}
where $A_i, i=1,\cdots, 3$ are defined in an obvious way,
\begin{align*}
 A_1 + A_2 +A_3 {}&{}\doteq \doubleint_{\D_{\tau_1}^{\tau_2}}  \left(-\frac{4M}{r}\eta^{-k+1} + 2 \eta^{-k+2} \right)r^{p-1} \pv\Solu\pv\So \\
 {}&{} +\doubleint_{\D_{\tau_1}^{\tau_2}} 2\left(\eta^2-2\eta(1-\frac{3M}{r})\right)\eta^{-k} r^{p}\laplacianslash\Solu \pv\So  \\
 {}&{} - \doubleint_{\D_{\tau_1}^{\tau_2}}  (V^Z + r\pv V^Z) \eta^{-k} r^{p+1} \frac{\Solu}{r} \pv\So.
\end{align*}
The bulk term in the second line of \eqref{rp-ineq-first-order} has a good sign. It could be moved to the right hand side of \eqref{rp-ineq-first-order}, and contributes to integrated decay estimate.

Next, we will estimate $A_1,\cdots A_3$ one by one.

For $A_1$, an application of Cauchy-Schwarz inequality yields
\begin{equation}\label{estimate-A-1}
| A_1 |\leq \frac{1}{c} \doubleint_{\D_{\tau_1}^{\tau_2}}   r^{p-1}|\pv\Solu|^2 + c\doubleint_{\D_{\tau_1}^{\tau_2}}   r^{p-1} |\pv\So|^2
\end{equation}
for some universal constant $c$. We choose the constant $c$ to be small enough so that the second term on the right hand side of \eqref{estimate-A-1} can be absorbed by $K_{p-1}(\So)$ which is on the left hand side of \eqref{rp-ineq-first-order}. And the first term $\doubleint_{\D_{\tau_1}^{\tau_2}}   r^{p-1}|\pv\Solu|^2 $ could be controlled by $K_{p-1}(\Solu)$, which is of lower order derivative.

For $A_2$, we rewrite it as 
\begin{equation}\label{eq-A-2}
 A_2 =\doubleint_{\D_{\tau_1}^{\tau_2}} -2r^{p}\laplacianslash\Solu \pv\So + f(r) \laplacianslash\Solu \pv\So,
\end{equation}
with $ |\partial^j_r f|\lesssim r^{p-1-j}, \, j \in \Naturals$.
Integrating by part twice, and using Cauchy-Schwarz inequality, we have
\begin{equation}\label{A-2}
\begin{split}
 |A_2| &\lesssim \doubleint_{\D_{\tau_1}^{\tau_2}} 2r^{p-1} |\pv\Omega\Solu|^2 + r^{p-3}  |\pv\Omega\Solu|^2 + (2-p)^2 r^{p-1} |\nablaslash\Solu|^2  \\
 &+\doubleint_{\D_{\tau_1}^{\tau_2}} r^{p-3}  |\nablaslash\Solu|^2 +  \int_{\scri} r^p \left(  |\pv \Omega \Solu|^2 +  |\nablaslash \Solu|^2\right).
 \end{split}
\end{equation}
Here we omit the boundary terms on $\{r=R\}$.
In \eqref{A-2}, $\doubleint_{\D_{\tau_1}^{\tau_2}} ( 2r^{p-1} + r^{p-3} )|\pv\Omega\Solu|^2 $ could be bounded by $K_{p-1}(\Omega\Solu)$, while 
\begin{equation}\label{esimate-A-2-main}
\doubleint_{\D_{\tau_1}^{\tau_2}} (2-p)^2 r^{p-1} |\nablaslash\Solu|^2 \leq C \cdot K_{p-1}(\Solu), \quad \text{for some} \quad C>2(2-p)
\end{equation}
Note that this holds for all $0<p\leq2$. In particular, \eqref{esimate-A-2-main} holds when $p=2$, while the estimate in the proof of Proposition 5.6 in \cite{S-13} does not hold for $p=2$.
Furthermore, noting that $p\leq 2$, and then $3-p >0$, we have 
\begin{equation*}
\doubleint_{\D_{\tau_1}^{\tau_2}} r^{p-3}  |\nablaslash\Solu|^2 \lesssim \doubleint_{\D_{\tau_1}^{\tau_2}} (3-p) r^{p-3}  |\nablaslash\Solu|^2 \lesssim K_{p-2}(\Solu).
\end{equation*}
Finally, the last term in \eqref{A-2} could be bounded by $S_p(\Solu)$ and $S_p(\Omega\Solu)$.

Similar for $A_3$, noticing that $|\dr^j (r^{p}V^Z)| \lesssim r^{p-3-j}, \, j \in \Naturals$, we integrate by parts, and use Cauchy-Schwarz inequality. Thus
\begin{equation}\label{A-3}
\begin{split}
 &|A_3| \lesssim  \doubleint_{\D_{\tau_1}^{\tau_2}} \left(r^{p-1} +r^{p-3} \right)  |\pv\Solu|^2  + r^{p-3}   \frac{\Solu^2}{r^2}  \\
 +& \doubleint_{\D_{\tau_1}^{\tau_2}}    (2-p)^2 r^{p-1} \frac{\Solu^2}{r^2}+ \int_{\scri}  r^p \left( |\pv \Solu|^2 + \frac{\Solu^2}{r^2} \right).
\end{split}
\end{equation}
In the same way, the bulk integral in the first line of \eqref{A-3} 
%$ \doubleint_{\D_{\tau_1}^{\tau_2}} \left(r^{p-1} +r^{p-3} \right)  |\pv\Solu|^2 + r^{p-3}   \frac{\Solu^2}{r^2} $ 
could be bounded by $K_{p-1}(\Solu)$.  
And 
\begin{equation}\label{esimate-A-3-main}
\doubleint_{\D_{\tau_1}^{\tau_2}} (2-p)^2 r^{p-1}  \frac{\Solu^2}{r^2} \leq C \cdot K_{p-1}(\Solu), \quad \text{for some} \quad C>2(2-p).
\end{equation}
 Note that this holds for all $0<p\leq2$. As in $A_2$, the last term in \eqref{A-3} could be bounded by $S_p(\Solu)$.

As a summary, we have
\begin{equation}\label{rp-ineq-first-order-sum-So}
\begin{split}
&\int_{\N_{\tau_2}}   r^p|\pv\So|^2 \di v\di \sigma_{S^2}  +K_{p-1}(\So) + S_p(\So)\\
\lesssim & \int_{\N_{\tau_1}}   r^p|\pv\So|^2 \di v\di \sigma_{S^2}  + K_{p-1}(\Solu)  + K_{p-2}(\Solu)+  K_{p-1} (\Omega\Solu) \\
& \,\,\,\,\,\,+S_p(\Solu) + S_p(\Omega\Solu)  +\text{boundary terms on} \,\, \{r=R\}.
\end{split}
\end{equation}
Applying Lemma \ref{lema-rp-hierarchy-0} to estimate $K_{p-1}(\Solu)  + K_{p-2}(\Solu)+  K_{p-1} (\Omega\Solu) + S_p(\Solu)$, we get
\begin{equation}\label{rp-1}
\begin{split}
 \sum_{j\leq1} \int_{\N_{\tau_2}} r^p | \pv \left( (r\pv)^j\Solu \right) |^2 \di v \di \sigma_{S^2} + \sum_{j,l\leq1} \int_{\scri}r^p \left( |r^{j-1}\pv^j\Omega^l\Solu|^2 +|\nablaslash\Solu|^2  \right)  {}&{} \\
+ \sum_{j\leq1}  \doubleint_{\D_{\tau_1}^{\tau_2}} pr^{p-1} \left( |\pv(r\pv)^j\Solu|^2 + \frac{M}{r}  \frac{|(r\pv)^j\Solu|^2}{r^2} \right)  \di u\di v\di \sigma_{S^2} {}&{}  \\
+\sum_{j\leq1} \doubleint_{\D_{\tau_1}^{\tau_2}}  (2-p)r^{p-1}  \left( |\nablaslash(r\pv)^j\Solu|^2 +  \frac{|(r\pv)^j \Solu|^2}{r^2} \right) \di u\di v\di \sigma_{S^2} {}&{}  \\
\lesssim    \int_{\N_{\tau_1}}  r^p \left( \sum_{j\leq1}  | \pv  (r\pv)^j\Solu |^2  +\sum_{l\leq 1} | \pv \Omega^l\Solu |^2    \right)  \di v \di \sigma_{S^2} {}&{}\\
+ \sum_{j,l\leq2}   \int_{\{r=R \}}  |\pv^j\nablaslash^l\Solu|^2 \di t \di \sigma_{S^2}.{}&{}  
\end{split}
\end{equation}

Next, we commute the equation \eqref{RW-Solu} or \eqref{Zerilli-Solu} with $\pu$ and $\Omega$,
\begin{equation}\label{commutator-u-omega}
[\L_Z, \pu] = 2\frac{\eta}{r}(1-\frac{3M}{r})\laplacianslash -\pu V^Z.
\end{equation}
We know that $ |\pu V^Z| \lesssim \frac{M}{r^4}.$ Thus making use of Cauchy Schwarz inequality, we have the energy inequality, 
\begin{equation}\label{rp-pu-energy-inequality}
\begin{split}
\int_{\N_{\tau_2}}  r^p | \pv \pu \Solu |^2 \di v \di \sigma_{S^2}  &\\
+\doubleint_{\D_{\tau_1}^{\tau_2}}  p r^{p-1} \left( |\pv \pu \Solu |^2 + \frac{M}{r}  \frac{ |\pu \Solu|^2}{r^2} \right)  \di u\di v\di \sigma_{S^2} &  \\
+\doubleint_{\D_{\tau_1}^{\tau_2}}   (2-p) r^{p-1}  \left( |\nablaslash \pu\Solu|^2 +  \frac{ |\pu \Solu|^2}{r^2} \right) \di u\di v\di \sigma_{S^2} &  \\
\lesssim    \int_{\N_{\tau_1}}   r^p | \pv \pu \Solu |^2 \di v \di \sigma_{S^2} + \text{boundary term on}\, \{r=R\}. & \\
 + \doubleint_{\D_{\tau_1}^{\tau_2}}  \epsilon r^{p-1} |\pv \pu \Solu |^2 +\frac{1}{\epsilon} r^{p-1}\left( |\laplacianslash\Solu|^2 +  \frac{ |\Solu|^2}{r^2} \frac{M^2}{r^4} \right) \di u\di v\di \sigma_{S^2}. & 
\end{split}
\end{equation}
We chose $\epsilon$ to be small enough, then $\doubleint_{\D_{\tau_1}^{\tau_2}}  \epsilon r^{p-1} |\pv \pu \Solu |^2$ could be absorbed by the second line of \eqref{rp-pu-energy-inequality}. Moreover, since $ |\laplacianslash\Solu|^2 \lesssim  \sum_{j\leq 2} r^{-2}|\Omega^j\Solu|^2,$  we have $$\doubleint_{\D_{\tau_1}^{\tau_2}}  r^{p-1}\left( |\laplacianslash\Solu|^2 +  \frac{ |\Solu|^2}{r^2} \frac{M^2}{r^4} \right) \di u\di v\di \sigma_{S^2} \lesssim \sum_{j\leq 2} K_{p-1}(\Omega^j\Solu).$$
Combining with Lemma \ref{lema-rp-hierarchy-0}, we then have
\begin{equation}\label{rp-pu-energy-inequality-sum}
\begin{split}
&\int_{\N_{\tau_2}}  r^p | \pv \pu \Solu |^2 \di v \di \sigma_{S^2}  \\
+ &\doubleint_{\D_{\tau_1}^{\tau_2}}  p r^{p-1} \left( |\pv \pu \Solu |^2 + \frac{M}{r}  \frac{ |\pu \Solu|^2}{r^2} \right) \di u\di v\di \sigma_{S^2}   \\
+ &\doubleint_{\D_{\tau_1}^{\tau_2}}   (2-p) r^{p-1}  \left( |\nablaslash \pu\Solu|^2 +  \frac{ |\pu \Solu|^2}{r^2} \right) \di u\di v\di \sigma_{S^2}   \\
\lesssim  &  \int_{\N_{\tau_1}}   r^p \left( | \pv \pu \Solu |^2 + \sum_{j\leq 2} | \pv \Omega^j\Solu |^2 \right)\di v \di \sigma_{S^2} \\
 + &\text{boundary term on}\, \{r=R\}. 
\end{split}
\end{equation}

Finally, we commute the equation with $\Omega$ which are killing vector fields, and hence $ [\L_Z, \Omega] = 0$. The statement follows straightforwardly.

\end{proof}

The general high order $r^p$ integrated decay estimate follows by a simple induction.
\begin{corollary}[$r^p$ High Order Integrated Decay Estimate]\label{rp-hierarchy-n}
Let $\Solu$ be a solution of the Regge-Wheeler \eqref{RW-Solu} or Zerilli equations \eqref{Zerilli-Solu} in the region 
$\D_{\tau_1}^{\tau_2},$ and define the weighted derivatives $\mathbb{D}=\{r\pv, (1-\mu)^{-1}\pu, r\nablaslash   \}$.
For $0< p\leq 2$,  there is the integrated decay estimate for all $n\in \Naturals,$
\begin{equation}\label{rp-psi-i-1}
\begin{split}
\int_{\N_{\tau_2}}  \sum_{j=0}^{n} r^p|\pv\mathbb{D}^j\Solu|^2 \di v\di \sigma_{S^2} + {}&{} \\
\doubleint_{\D_{\tau_1}^{\tau_2}} \sum_{j=0}^{n}  r^{p-1} \{|\pv\mathbb{D}^j\Solu|^2  +  (2-p) \left(|\nablaslash\mathbb{D}^j\Solu|^2  +  \frac{|\mathbb{D}^j\Solu|^2 }{r^2} \right) + \frac{6M}{r} \frac{|\mathbb{D}^j\Solu|^2 }{r^2}\} \di u\di v\di \sigma_{S^2}  {}&{}  \\
\lesssim   \int_{\N_{\tau_1}}   \sum_{j=0}^{n} r^p |\pv\mathbb{D}^j\Solu|^2 \di v\di \sigma_{S^2} 
+\int_{r=R} \sum_{j=0}^{n} \{|\pv\mathbb{D}^j\Solu|^2+ |\nablaslash\mathbb{D}^j\Solu|^2 +  |\mathbb{D}^j\Solu|^2 \}\di t \di \sigma_{S^2}.{}&{} 
\end{split}
\end{equation}
\end{corollary}

\subsection{Improved decay estimate}\label{Improved interior decay}
Due to the $r^p$ first order integrated decay estimate (Lemma \ref{lema-rp-hierarchy-1}), we could improve the decay of first order energy.
\begin{corollary}[Improved Pointwise Decay]\label{Improved interior first order energy decay}
Let $\solu$ be a solution of Regge-Wheeler \eqref{RW} or Zerilli equations \eqref{Zerilli}, with initial data on $\Sigma^{\prime}_{\tau_0}\cup \N_{\tau_0}$ satisfying
\begin{equation}\label{initial data-t}
\begin{split}
& \int_{\N_{\tau_0}}  \left( \sum_{k\leq1,j\leq2} r^{2} |T^k  (r\pv)^j\solu|^2 + \sum_{k\leq4,l\leq 2} r^{2} |T^k\Omega^l \pv(r\solu)|^2 \right) \di v\di \sigma_{S^2} \\
 +& \int_{\Sigma^{\prime}_{\tau_0}\cup \N_{\tau_0}}   \sum\limits_{k\leq4,l\leq2} P^N_\mu(T^k\Omega^l \solu) n^\mu + \sum\limits_{k\leq5} P^N_\mu(T^k\solu) n^\mu<\infty.
\end{split}
\end{equation}
Then there is a constant $I$ depending on the initial data \eqref{initial data-t},
\begin{equation}\label{improved-nondeg-Integrated decay estimate-t}
\int_{\Sigma^{\prime}_\tau \cup \N_\tau}  P^N_\mu(T\solu) n^\mu  \lesssim  \frac{I}{\tau^{4}},
\end{equation}
where $\Sigma^{\prime}_\tau=\Sigma_{\tau} \cap \{r\leq R\}, \,\, \N_{\tau}=\{\M| u=\tau-R^\ast, v \geq v_0=\tau_0+R^\ast \} $.
\end{corollary}

\begin{proof}
As our proof is the same for both Regge-Wheeler and Zerilli case, 
we take the Zerilli case for example.
Let
\begin{equation}\label{def-so1}
\psifirst\doteq \pv\Solu.
\end{equation}
Recalling that $\Solu = r\solu$, $\So=r\psifirst$ in the first order $r^p$ weighted inequality \eqref{rp-1}, we make use of the zero order $r^p$ weighted estimate in Lemma \ref{lema-rp-hierarchy-0}, thus for $2< p \leq4$
\begin{equation}\label{rp-psi-1-1}
\begin{split}
\int_{\N_{\tau_2}}    r^p|\pv\psifirst|^2\di v\di \sigma_{S^2} +\doubleint_{\D_{\tau_1}^{\tau_2}} ( p-2 ) r^{p-1} \left( |\pv\psifirst|^2  +\frac{M}{r}  \frac{|\psifirst|^2}{r^2} \right)  \di u\di v\di \sigma_{S^2} {}&{} \\
+\doubleint_{\D_{\tau_1}^{\tau_2} }   (4-p)  r^{p-1} \left(|\nablaslash\psifirst|^2  +  \frac{|\psifirst|^2 }{r^2} \right) \di u\di v\di \sigma_{S^2}  {}&{}  \\
\lesssim  \int_{\N_{\tau_1}} \{ r^p |\pv\psifirst|^2+\sum_{j\leq 1} r^{p-2} |\pv\Omega^j\Solu|^2\} \di v\di \sigma_{S^2}
+\int_{\{r=R \}}  \sum_{j,l\leq 2} |\nablaslash^j\pv^l\Solu|^2 \di t\di \sigma_{S^2}. {}&{} 
\end{split}
\end{equation}
We are interested in the quantity $\pv T\Solu = \pv^2\Solu+ \pv\partial_u\Solu$.
In view of the Zerilli equations \eqref{Zerilli-Solu}, we have
\begin{equation}\label{pv-T-Solu}
|\pv T\Solu|^2 \lesssim |\pv\psifirst|^2+|\laplacianslash \Solu|^2 +\frac{|M|^2}{r^2} \frac{|\Solu|^2}{r^4}.
\end{equation}
For the second term $ |\laplacianslash \Solu|^2 \lesssim \frac{ |\nablaslash\Omega\Solu|^2}{r^2}$, we repeat the proof of Theorem \ref{energy decay} with $\Omega\Solu$ in the place of $\Solu$, thus
\begin{equation}\label{decay-omega-Solu}
\int_{\N_{\tau}} r^2 |\laplacianslash \Solu|^2\di v \di \sigma_{S^2} \lesssim \int_{\N_{\tau}} |\nablaslash\Omega\Solu|^2\di v \di \sigma_{S^2} \lesssim \int_{ \N_\tau}  P^N_\mu(\Omega\solu) n^\mu \lesssim  \frac{I}{\tau^2}.,
\end{equation}
where $I$ depends only on the initial data \eqref{initial data-t}.
Hence, we have
\begin{equation}\label{bound-r2-pv-T-N}
 \int_{\N_{\tau}} r^{2}|\pv T\Solu|^2 \di v\di \sigma_{S^2}  \lesssim  \int_{\N_{\tau}} \left( r^2 |\pv\psifirst|^2  + \frac{|M|^2}{r^2} \frac{|\Solu|^2}{r^2} \right) \di v\di \sigma_{S^2} +  \frac{I}{\tau^2}.
\end{equation}
Notice that,
\begin{equation}\label{p-p-2-Solu}
\doubleint_{\D_{\tau_1}^{\tau_2} }  r^{p-1} \frac{M^2}{r^2} \frac{|\Solu|^2}{r^4}  \di u \di v \di \sigma_{S^2} \\
\lesssim  \doubleint_{\D_{\tau_1}^{\tau_2} }  r^{p-2-1} \frac{M}{r} \frac{|\Solu|^2}{r^2}  \di u \di v \di \sigma_{S^2}.
\end{equation}
We shall use \eqref{rp-psi-1-1} for $2<p\leq 4$ and Lemma \ref{lema-rp-hierarchy-0} for $0<p-2\leq 2$, thus
\begin{equation}\label{rp-pv-T-Solu}
\begin{split}
&\doubleint_{\D_{\tau_1}^{\tau_2} }  r^{p-1}\left( |\pv\psifirst|^2 +\frac{|M|^2}{r^2} \frac{|\Solu|^2}{r^4} \right) \di u \di v \di \sigma_{S^2} \\
 \lesssim  &\int_{\N_{\tau_1}} \{ r^p |\pv\psifirst|^2+ \sum_{j\leq 1} r^{p-2} |\pv\Omega^j\Solu|^2\} \di v\di \sigma_{S^2} \\
+ &\text{ boundary term at}\,\, \{r=R\}.
\end{split}
\end{equation}
Next, we would apply the $r^p$ hierarchy estimate to improve the energy decay.
 Taking $p=4$ in \eqref{rp-psi-1-1} and $p=2$ in \eqref{ineq-rp-0} with $\Solu$ being replaed by $\Omega \Solu$, we apply the pigeon-hole principle. There exists a sequence $\{ \tau\}_{j \in \mathbb{N}}$ such with $\tau_{j+1} = 2\tau_j$,
\begin{equation}\label{rp-pv-T-Solu-p4}
\begin{split}
&\int_{\N_{\tau_j} } \{ r^3 |\pv\psifirst|^2+ \sum_{j\leq 1} r |\pv\Omega^j\Solu|^2\} \di v\di \sigma_{S^2}  \\
 \lesssim  & \frac{1}{\tau_j}   \int_{\N_{\tau_0} } \{ r^4 |\pv\psifirst|^2+ \sum_{j\leq 1} r^2 |\pv\Omega^j\Solu|^2\} \di v\di \sigma_{S^2}  \\
+ & \frac{1}{\tau_j}   \int_{\Sigma^{\prime}_{\tau_0} \cup \N_{\tau_0} } \sum_{k\leq 2, l \leq 1}  P^N_\mu(T^k \Omega^l \solu) n^\mu + \sum_{k \leq 3} P^T(T^k\solu).
\end{split}
\end{equation}
Furthermore, taking $p=3$ in \eqref{rp-psi-1-1} and $p=1$ in \eqref{ineq-rp-0}, we have
\begin{equation}\label{rp-pv-T-Solu-p3}
\begin{split}
&\doubleint_{\D_{\tau_j}^{\tau_{j+1}} }   \left( r^2 |\pv\psifirst|^2 +\frac{M^2}{r^2} \frac{|\Solu|^2}{r^2} \right) \di u \di v \di \sigma_{S^2} \\
 \lesssim  &\int_{\N_{\tau_i} } \{ r^3 |\pv\psifirst|^2+ \sum_{j\leq 1} r |\pv\Omega^j\Solu|^2\} \di v\di \sigma_{S^2}  \\
+ & \int_{\Sigma^{\prime}_{\tau_j} \cup \N_{\tau_j} } \sum_{k\leq 2, l \leq 1}  P^N_\mu(T^k \Omega^l \solu) n^\mu + \sum_{k \leq 3} P^T(T^k\solu).
\end{split}
\end{equation}
Viewing \eqref{rp-pv-T-Solu-p4} and \eqref{rp-pv-T-Solu-p3}, we have
\begin{equation}\label{rp-pv-T-Solu-p3-p4}
\begin{split}
&\int_{\tau_j}^{\tau_{j+1}}  \int_{ \N_{\tau} } \left( r^2 |\pv\psifirst|^2 +\frac{M^2}{r^2} \frac{|\Solu|^2}{r^2} \right) \di u \di v \di \sigma_{S^2} \\
 \lesssim  & \frac{1}{\tau_j}   \int_{\N_{\tau_0} } \{ r^4 |\pv\psifirst|^2+ \sum_{j\leq 1} r^2 |\pv\Omega^j\Solu|^2\} \di v\di \sigma_{S^2}  \\
 & \,\,\,\,\,\,\,\,\,\,\,\,\,\,\,\,  + \frac{1}{\tau_j}   \int_{\Sigma^{\prime}_{\tau_0} \cup \N_{\tau_0} } \sum_{k\leq 2, l \leq 1}  P^N_\mu(T^k \Omega^l \solu) n^\mu + \sum_{k \leq 3} P^T(T^k\solu)\\
 & \,\,\,\,\,\, \,\,\,\,\,\,\,\,\,\,\,\,\,\,\,\,\,\,\,\,  +\int_{\Sigma^{\prime}_{\tau_j} \cup \N_{\tau_j} } \sum_{k\leq 2, l \leq 1}  P^N_\mu(T^k \Omega^l \solu) n^\mu + \sum_{k \leq 3} P^T(T^k\solu).
\end{split}
\end{equation}
As in the proof of Theorem \ref{energy decay}, we make use of the uniform boundness of energy \eqref{uniform boundness-v} to estimate the last term in \eqref{rp-pv-T-Solu-p3-p4}. Therefore, in the same way, we have
\begin{equation}\label{rp-pv-T-Solu-p3-p4-sum}
\begin{split}
&\int_{\tau_j}^{\tau_{j+2}}  \di \tau  \int_{\N_{\tau} } \left( r^2 |\pv\psifirst|^2 +\frac{M^2}{r^2} \frac{|\Solu|^2}{r^2} \right) \di r^\ast \di \sigma_{S^2} \\
 \lesssim  & \frac{1}{\tau_j}   \int_{\N_{\tau_0} } \{ r^4 |\pv\psifirst|^2+ \sum_{j\leq 1} r^2 |\pv\Omega^j\Solu|^2\} \di v\di \sigma_{S^2}  \\
 & \,\,\,\,\,\,\,\,\,\,\,\,\,\,\,\,  + \frac{1}{\tau_j}   \int_{\Sigma^{\prime}_{\tau_0} \cup \N_{\tau_0} } \sum_{k\leq 3, l \leq 1}  P^N_\mu(T^k \Omega^l \solu) n^\mu + \sum_{k \leq 4} P^T(T^k\solu).
 \end{split}
\end{equation}
We again use the pigeon-hole principle to obtain 
\begin{equation}\label{rp-pv-T-Solu-decay-t2}
\begin{split}
& \int_{\N_{\tau} } \left( r^2 |\pv\psifirst|^2 +\frac{M^2}{r^2} \frac{|\Solu|^2}{r^2} \right) \di r^\ast \di \sigma_{S^2} \\
 \lesssim  & \frac{1}{\tau^2}  \int_{\N_{\tau_0}} \left( r^{4}|\pv\psifirst|^2+  \sum_{k\leq3,l\leq 1} r^{2} |T^k\Omega^l \pv\Solu|^2\right)  \di v\di \sigma_{S^2} \\
  + &\frac{1}{\tau^2} \int_{\Sigma^\prime_{\tau_0} \cup \N_{\tau_0}}   \sum\limits_{k\leq3,l\leq1} P^N_\mu(T^k\Omega^l \solu) n^\mu + \sum\limits_{k\leq4} P^T_\mu(T^k\solu) n^\mu.
 \end{split}
\end{equation}
Regarding \eqref{bound-r2-pv-T-N}, we have established
\begin{equation}\label{improved-nondeg-Integrated decay estimate-T-1}
\begin{split}
 &\quad \quad \quad \quad \quad  \int_{\N_{\tau}} r^{2}|\pv T\Solu|^2 \di v\di \sigma_{S^2} \\
\lesssim & \frac{1}{\tau^2}  \int_{\N_{\tau_0}} \left( r^{4}|\pv\psifirst|^2+  \sum_{k\leq3,l\leq 1} r^{2} |T^k\Omega^l \pv\Solu|^2\right)  \di v\di \sigma_{S^2} \\
  + &\frac{1}{\tau^2} \int_{\Sigma^\prime_{\tau_0} \cup \N_{\tau_0}}   \sum\limits_{k\leq3,l\leq1} P^N_\mu(T^k\Omega^l \solu) n^\mu + \sum\limits_{k\leq4} P^T_\mu(T^k\solu) n^\mu.
\end{split}
\end{equation}

Taking $p=2$ and $p=1$ in the $r^p$ weighted energy inequality \eqref{ineq-rp-0} with $\Solu$ being replaced by $T\Solu$, we repeat the proof of Theorem \ref{energy decay}, therefore
 for a dyadic sequence $\{\tau_j\}_{j\in \Naturals}$,
\begin{equation}\label{improved-nondeg-Integrated decay estimate-T-2}
\begin{split}
&\int_{\Sigma^\prime_{\tau_{j+1}} \cup \N_{\tau_{j+1}}} P^N_\mu(T\solu) n^\mu \lesssim \\
  &\frac{1}{\tau_j^2} \left(\int_{\N_{\tau_j}}  \sum_{k\leq2} r^{2} |T^k\pv\Solu|^2  \di v \di \sigma_{S^2} + \int_{\Sigma^\prime_{\tau_j} \cup \N_{\tau_j}} \sum_{k\leq3}  P^N_\mu(T^k\solu) n^\mu \right).
  \end{split}
\end{equation}
As a result of \eqref{improved-nondeg-Integrated decay estimate-T-1} and Theorem \ref{energy decay}, we use the pigeon-hole principle for \eqref{improved-nondeg-Integrated decay estimate-T-2}, and \eqref{improved-nondeg-Integrated decay estimate-t} follows.

\end{proof}

Based on the improved first order energy decay, we can improve the pointwise decay.
\begin{theorem}[Improved Interior Pointwise Decay]\label{Improved interior first order pointwise}
Let $R>3M$ and $\solu$ be a solution of Regge-Wheeler \eqref{RW} or Zerillir equation \eqref{Zerilli}, with initial data on $\Sigma^\prime_{\tau_0}\cup \N_{\tau_0}$ satisfying 
\begin{equation}\label{initial data-improved decay-1}
\begin{split}
I=& \int_{\N_{\tau_0}}    \sum_{k+l\leq8,l\leq 4}  |T^k \Omega^l  (r\pv)^j\solu|^2r^2 \di v\di \sigma_{S^2}\\
 +&\int_{\Sigma^{\prime}_{\tau_0}\cup \N_{\tau_0}}   \sum\limits_{k+l\leq9,l\leq4} P^N_\mu(T^k\Omega^l \solu) n^\mu<\infty.
\end{split}
\end{equation}
Then we have in the future development of initial hypersurface
\begin{equation}\label{improved-pointwise decay}
 \quad  r^{\frac{1}{2}}|\partial_t\solu| (\tau,r) \lesssim \frac{I}{\tau^2},\,\,\, \text{in} \,\,\,  J^{+}(\Sigma^\prime_{\tau_0}\cup \N_{\tau_0}).
\end{equation}
and the improved interior decay estimate, 
\begin{equation}\label{improved-pointwise decay-pt}
|\solu| \lesssim \frac{I}{\tau^{\frac{3}{2}}}, \,\,\text{for} \,\, r<R.
\end{equation}

\end{theorem}

\begin{proof}
In Theorem \ref{them-uni-pointwise decay}, we had used the Sobolev inequalities to prove the pointwise decay estimate:
$r^{\frac{1}{2}}  |\solu| \lesssim \frac{I}{\tau}$ with $I$ being a constant depending on the initial data \eqref{initial data-pointwise}.
Similarly, based on the improved first order energy decay (Corollary \ref{Improved interior first order energy decay}), we have $r^{\frac{1}{2}}|\partial_t\solu| \lesssim \frac{I}{\tau^2}$,  where $I$ is a constant depending on the initial data \eqref{initial data-improved decay-1}.

Next we interpolate between $\solu$ and $\partial_t\solu$ to improve the pointwise decay for $|\solu|$ \cite{S-13}.
The basic observation underlying this argument is that for $t_1>t_0$
\begin{equation}\label{interp-integrate-t}
\begin{split}
&r \psi^2(t_1, r) = r \psi^2(t_0, r) + \int_{t_0}^{t_1} 2 \psi \dt\psi r \di t\\
\leq & r \psi^2(t_0, r) + t^{-1}_0  \int_{t_0}^{t_1} \psi^2(t,r) r \di t +  t_0  \int_{t_0}^{t_1}  (\dt\psi)^2(t,r) r \di t.
\end{split}
\end{equation}
For $r_0<r<R$, 
\begin{equation}\label{interp-integrate-r}
r \psi^2(t, r) \lesssim  R \psi^2(t, R)+ \int_{r^\ast}^{R^\ast} \psi^2(t, r)  \di r^\ast +   \int_{r^\ast}^{R^\ast} (\p_{r^\ast}\psi)^2(t,r) r^{2} \di r^\ast.
\end{equation}
Thus, integrating on $t$ and using the Sobolev inequality on the sphere, we have
\begin{equation}\label{interp-integrate-basic-ineq}
\begin{split}
\int_{t_0}^{t_1} r\psi^2(t, r) \di t &\lesssim \int_{t_0}^{t_1} \di t  \int_{S^2} \di \sigma_{S^2} \sum_{l\leq1} R(\Omega^l \psi)^2(t,R)  \\
& +\int_{t_0}^{t_1}\di t \int_{r^\ast}^{R^\ast} \di r^\ast \int_{S^2} \di \sigma_{S^2} \sum_{l\leq1} r^{2}  \left( \frac{(\Omega^l \psi)^2}{r^2} + (\p_{r^\ast} \Omega^l \psi)^2 \right) .
\end{split}
\end{equation}
That is,
\begin{equation}\label{interp-integrate-basic-ineq-1}
\int_{t_0}^{t_1} r\psi^2(t, r) \di t \lesssim \int_{\Sigma^\prime_{t_0}} \sum_{k, l\leq1} P^T(T^k\Omega^l \psi).
\end{equation}
Letting $\bar{t}_1 = 2\bar{t}_0$, by Theorem \ref{energy decay}, there exists  $t^\prime_0 \in (\bar{t}_0,  \bar{t}_1)$ such that
\begin{equation}\label{interp-decay-initial-t}
 r\psi^2(t^\prime_0, r) \lesssim \frac{I}{t^{\prime3}_0}.
\end{equation}
Now considering a dyadic sequence $\{t^\prime_{j}\}_{j \in \mathbb{N}}$ with $t^\prime_{j+1}=2t^\prime_j,\, j\geq 0$, as an application of \eqref{interp-integrate-basic-ineq-1}, we have
\begin{equation}\label{interp-integrate-basic-ineq-1-dt}
\begin{split}
\int_{t^\prime_j}^{t^\prime_{j+1}} r(\psi)^2(t, r) \di t &\lesssim \int_{\Sigma^\prime_{t_j}} \sum_{k, l\leq1} P^T(T^k\Omega^l \psi), \\
\int_{t^\prime_j}^{t^\prime_{j+1}} r(\dt\psi)^2(t, r) \di t &\lesssim \int_{\Sigma^\prime_{t_j}} \sum_{k, l\leq1} P^T(T^k\Omega^l \dt\psi).
\end{split}
\end{equation}
In regard of Theorem \ref{energy decay} and Corollary \ref{Improved interior first order energy decay},  \eqref{interp-integrate-t} yields that
\begin{equation}\label{interp-integrate-t-decay}
r \psi^2(t^\prime_{j+1}, r) \lesssim r \psi^2(t^\prime_{j}, r) + \frac{1}{t^\prime_j} \frac{I}{(t^\prime_j)^2}  +  t^\prime_j \frac{I}{(t^\prime_j)^4}.
\end{equation}
By induction on $j\in \mathbb{N}$ and noting that  \eqref{interp-decay-initial-t} for $j=0$, we have
\begin{equation}\label{interp-integrate-t-decay-final}
r \psi^2(t^\prime_{j}, r) \lesssim \frac{I}{(t^\prime_j)^3}, \quad j\in \mathbb{N} \cup \{0\}.
\end{equation}
The pigeon-hole principle will give the conclusion.

For $2M\leq r < r_0$, the same interpolation \eqref{interp-integrate-t} by integration along lines of constant radius $r< r_0$ can be
carried out. While \eqref{interp-integrate-r} could be replaced by integration on $v=\frac{1}{2} (t+r_0^\ast)$,
\begin{equation}\label{interp-integrate-u}
\psi^2(t, u) \lesssim  \psi^2(t, u_0)+ \int_{u_0}^{u} (1-\mu)\psi^2 \di u +   \int_{u_0}^{u} \frac{(\p_u \psi)^2}{1-\mu} \di u.
\end{equation}
where $u_0=\frac{1}{2} (t- r_0^\ast)$. 
%And hence \eqref{interp-integrate-basic-ineq-1} is changed into 
%\begin{equation}\label{interp-integrate-basic-ineq-hori}
%\int_{t_0}^{t_1} \psi^2(t, r) \di t \lesssim \int_{\Sigma^i_{t_0}} \sum_{k, l\leq1} P^N(T^k\Omega^l \psi).
%\end{equation}
The above proof could be extended to $2M\leq r < r_0$ by replacing  $P^T$ by $P^N$ on the right hand sides of \eqref{interp-integrate-basic-ineq-1}, \eqref{interp-integrate-basic-ineq-1-dt}.
\end{proof}
\begin{remark}
The argument in proving Theorem \ref{Improved interior first order pointwise} does not hold for $R \rightarrow \infty.$ Otherwise, adapted to $R \rightarrow \infty,$ \eqref{interp-integrate-r} would be replaced by integrating in $v$. To mimic the procedure of deriving  \eqref{interp-integrate-basic-ineq-1} from \eqref{interp-integrate-basic-ineq}, we should take $p=0$ in the zero order $r^p$ inequality \eqref{ineq-rp-0}. But this is impossible, since we require that $p>0$ in \eqref{ineq-rp-0}.
\end{remark}

Combining with the non-degenerate high order integrated decay estimate (Corollary \ref{coro-nondeg-Integrated decay estimate-high-order}),  the weighted high order uniform boundedness (Corollary \ref{uniform-bound-high}), and the $r^p$ high order integrated decay estimate (Corollary \ref{rp-hierarchy-n}), 
we have the improved high order pointwise decay estimate by a simple induction.
\begin{corollary}[High Order Decay Estimate]\label{Improved high order interior first order pointwise}
Let $R>3M$, and define the weighted derivatives $\mathbb{D}=\{r\pv, (1-\mu)^{-1}\pu, r\nablaslash\}$.
Let $\solu$ be a solution of Regge-Wheeler \eqref{RW} or Zerilli equations \eqref{Zerilli}, with initial data on  $\Sigma^\prime_{\tau_0}\cup \N_{\tau_0}$ satisfying
\begin{equation}\label{initial data-improved decay-high-order}
\begin{split}
I=& \sum_{i\leq n} \int_{\N_{\tau_0}}    \sum_{j\leq2, k+l \leq 8 }  |\mathbb{D}^i (r\pv)^j \Omega^l T^k \solu|^2r^2 \di v\di \sigma_{S^2}\\
 +& \sum_{k+l\leq n+9}  \int_{\Sigma^{\prime}_{\tau_0}\cup \N_{\tau_0}}   P^N_\mu(T^k\Omega^l \solu) n^\mu<\infty.
\end{split}
\end{equation}
for all $n\in \Naturals$. 
Then we have the energy decay estimate,
\begin{equation}\label{decay-high-RW-teu}
\sup_{m \in \Naturals }  \int_{\N_\tau}   P^N_\alpha(\mathbb{D}^m\solu)n^\alpha +   \sup_{k,l \in \Naturals }   \int_{\Sigma^{\prime}_{\tau}}  P^N_\mu(T^k\Omega^l \solu) n^\mu \lesssim \frac{I}{\tau^2},
\end{equation}
and
\begin{equation}\label{decay-high-RW-teu-Improved}
\sup_{m \in \Naturals } \int_{\N_\tau}   P^N_\alpha(\mathbb{D}^m T\solu)n^\alpha +  \sup_{k,l \in \Naturals }   \int_{\Sigma^{\prime}_{\tau}}  P^N_\mu(T^k\Omega^l T\solu) n^\mu \lesssim \frac{I}{\tau^4}.
\end{equation}
In the future development of initial hypersurface $J^{+}(\Sigma^\prime_{\tau_0}\cup \N_{\tau_0})$, there is the pointwise decay estimate
\begin{equation}\label{improved-pointwise decay-high-RW}
\sup_{m \in \Naturals} r^{\frac{1}{2}} |\mathbb{D}^m\solu|(\tau, r) \lesssim \frac{I}{\tau}, \quad \sup_{m \in \Naturals}  r^{\frac{1}{2}} |\partial_t\mathbb{D}^m \solu|(\tau, r) \lesssim \frac{I}{\tau^2},
\end{equation}
and the improved interior decay estimate,
\begin{equation}\label{improved-pointwise decay}
 \sup_{m \in \Naturals}  |\mathbb{D}^m\solu| (\tau,r)\lesssim \frac{I}{\tau^{\frac{3}{2}}},  \,\, \text{for} \,\, r<R.
\end{equation}

\end{corollary}

\end{document}